\documentclass{article}
\usepackage[utf8]{inputenc}
\usepackage{amsmath,amsthm, amssymb, amsfonts}
\usepackage{todonotes}
\usepackage{kantlipsum}
\usepackage{wrapfig}
\usepackage{paralist}
\usepackage{wasysym}
\usepackage{verbatim}
\usepackage{fullpage}
\usepackage{mathtools}
\usepackage[pdfencoding=auto]{hyperref}
\usepackage{graphicx}
\usepackage{wrapfig}
\usepackage{mathrsfs}
\usepackage{enumerate}
\usepackage{dsfont}
\usepackage{units}
\usepackage[]{algorithm2e}
\usepackage{ltablex}
\usepackage{varwidth}

\usepackage{xcolor}

\allowdisplaybreaks[1]
\usepackage{sistyle} 
\newcommand\bSI[1]{{\small[\SI{}{#1}]}}

\makeatletter
\newlength\unitwdth
\newlength\numwdth
\newlength\tdima
\newcommand\SIdescr[2]{%
    \setlength\tdima{\linewidth}%
    \addtolength\tdima{\@totalleftmargin}%
    \addtolength\tdima{-\dimen\@curtab}%
    \addtolength\tdima{-\unitwdth}%
    \addtolength\tdima{-\numwdth}%
    \parbox[t]{\tdima}{%
        #1
        \leaders\hbox{$\m@th\mkern \@dotsep mu\hbox{\tiny.}\mkern \@dotsep mu$}%
        \hfill
        \ifhmode\strut\fi
        \makebox[0pt][l]{%
            \makebox[\unitwdth][l]{}%
            \makebox[\numwdth][r]{#2}}}}
\makeatother

\newcommand{\Ycal}{\mathcal{Y}}
\newcommand{\Hil}{\mathcal{H}}

\newcommand{\Z}{\mathbb{Z}}
\newcommand{\N}{\mathbb{N}}

\newcommand{\R}{\mathbb{R}}

\newcommand{\epsilontilde}{{\tilde{\epsilon}}}
\newcommand{\epsilonhat}{{\hat{\epsilon}}}

\newcommand{\Realization}{\mathrm{R}}
\newcommand{\mult}{\Phi^{Z,d,n,l}_{\mathrm{mult};\tilde{\epsilon}}}

\newcommand{\vect}{\mathbf{vec}}
\newcommand{\Paral}{\mathrm{P}}

\newcommand{\identity}{\mathbf{Id}}

\newcommand{\conc}{{\raisebox{2pt}{\tiny\newmoon} \,}}

\newtheorem{theorem}{Theorem}[section]
\newtheorem*{theorem*}{Theorem}
\newtheorem{remark}[theorem]{Remark}
\newtheorem{definition}[theorem]{Definition}
\newtheorem{proposition}[theorem]{Proposition}
\newtheorem{lemma}[theorem]{Lemma}
\newtheorem{corollary}[theorem]{Corollary}
\newtheorem{assumption}[theorem]{Assumption}
\newtheorem*{remark*}{Remark}
\newtheorem*{proposition*}{Proposition}

\numberwithin{equation}{section}

\definecolor{darkgreen}{rgb}{0.0, 0.5, 0.0}

\definecolor{darkcandyapplered}{rgb}{0.64, 0.0, 0.0}

\title{A Theoretical Analysis of Deep Neural Networks and \\ Parametric PDEs}

\author{Gitta Kutyniok \thanks{Institut für Mathematik, Technische Universit\"at Berlin, Straße des 17.~Juni 136, 10623 Berlin, Germany, e-mail: \texttt{$\{$kutyniok, raslan, schneidr$\}$@math.tu-berlin.de}} \thanks{Fakult\"at für Informatik und Elektrotechnik, Technische Universit\"at Berlin}
\thanks{
Department of Physics and Technology, University of Troms\o} \and Philipp Petersen\thanks{University of Vienna, Faculty of Mathematics and Research Plattform Data Science @ Uni Vienna, Oskar Morgenstern Platz 1,
1090 Wien, e-mail: \texttt{philipp.petersen@univie.ac.at}} 
\and Mones Raslan\footnotemark[1] 
\and Reinhold Schneider\footnotemark[1]}

\begin{document}

\maketitle

\begin{abstract}
We derive upper bounds on the complexity of ReLU neural networks approximating the solution maps of parametric partial differential equations. In particular, without any knowledge of its concrete shape, we use the inherent low-dimensionality of the solution manifold to obtain approximation rates which are significantly superior to those provided by classical neural network approximation results. Concretely, we use the existence of a small reduced basis to construct, for a large variety of parametric partial differential equations, neural networks that yield approximations of the parametric solution maps in such a way that the sizes of these networks essentially only depend on the size of the reduced basis.
\end{abstract}
\vspace{.2cm}



\section{Introduction}

In this work, we analyze the suitability of deep neural networks (DNNs) for the numerical solution of parametric problems. Such problems connect a parameter space with a solution state space via a so-called \emph{parametric map}, \cite{ohlberger2015reduced}. 
One special case of such a parametric problem arises when the parametric map results from solving a partial differential equation (PDE) and the parameters describe physical or geometrical constraints of the PDE such as, for example, the shape of the physical domain, boundary conditions, or a source term. 
Applications that lead to these problems include modeling unsteady and steady
heat and mass transfer, acoustics, fluid mechanics, or electromagnetics, \cite{CertReduced}.

Solving a parametric PDE for every point in the parameter space of interest individually, typically leads to two types of problems. 
First, if the number of parameters of interest is excessive---a scenario coined many-query application---then the associated computational complexity could be unreasonably high. 
Second, if the computation time is severely limited, such as in real-time applications, then solving even a single PDE might be too costly. 

A core assumption to overcome the two issues outlined above is that the solution manifold, i.e., the set of all admissible solutions associated with the parameter space, is inherently low-dimensional. This assumption forms the foundation for the so-called reduced basis method (RBM). A reduced basis discretization is then a (Galerkin) projection on a low-dimensional approximation space that is built from snapshots of the parametrically induced manifold, \cite{RozzaPateraRB}. 

Constructing the low-dimensional approximation spaces is typically computationally expensive because it involves solving the PDEs for multiple instances of parameters. 
These computations take place in a so-called \emph{offline} phase---a step of pre-computation, where one assumes to have access to sufficiently powerful computational resources. 
Once a suitable low-dimensional space is found, the cost of solving the associated PDEs for a new parameter value is significantly reduced and can be performed quickly and \emph{online}, i.e., with limited resources, \cite{BALMES1996381, prud2002reduced}. 
We will give a more thorough introduction to RBMs in Section \ref{sec:RedBase}. 
An extensive survey of works on RBMs, which can be traced back to the seventies and eighties of the last century (see for instance \cite{FoxMiura,NoorI,NoorII}), is beyond the scope of this paper. 
We refer, for example, to \cite[Chapter 1.1]{CertReduced}, \cite{QuarteroniIntro,DahmenSampleSolution,HaasdonkIntro} 
and \cite[Chapter 1.9]{CohenDeVoreHighPDE} for (historical) studies of this topic.

In this work, we show that the low-dimensionality of the solution manifold also enables an efficient approximation of the parametric map by DNNs. In this context, the RBM will be, first and foremost, a tool to model this low-dimensionality by acting as a blueprint for the construction of the DNNs.

\subsection{Statistical Learning Problems}

The motivation to study the approximability of parametric maps by DNNs stems from the following similarities between parametric problems and \emph{statistical learning problems}:
Assume that we are given a \emph{domain set} $X\subset \R^n$, $n \in \N$ and a \emph{label set} $Y \subset \R^k$, $k \in \N$. 
Further assume that there exists an unknown probability distribution $\rho$ on $X \times Y$. 

Given a \emph{loss function} $\mathcal{L} \colon Y \times Y\to \R^+$, the goal of a statistical learning problem is to find a function $f$, 
which we will call \emph{prediction rule}, from a hypothesis class $H \subset \{h\colon X \to Y\}$ such that the \emph{expected loss} $\mathbb{E}_{(\mathbf{x},\mathbf{y})\sim \rho}\mathcal{L}(f(\mathbf{x}), \mathbf{y})$ is minimized, \cite{Cucker02onthe}. 
Since the probability measure $\rho$ is unknown, we have no direct access to the expected loss.
Instead, we assume that we are given a set of training data, i.e. pairs $(\mathbf{x}_i, \mathbf{y}_i)_{i = 1}^N$, $N \in \N$, 
which were drawn independently with respect to $\rho$. Then one finds $f$ by minimizing the so-called \emph{empirical loss} 
\begin{align}\label{eq:EmpLoss}
    \sum_{i = 1}^N \mathcal{L}(f(\mathbf{x}_i), \mathbf{y}_i)
\end{align}
over $H$. We will call optimizing the empirical loss the \emph{learning procedure}.

In view of PDEs, the approach proposed above can be rephrased in the following way. 
We are aiming to produce a function from a parameter set to a state space based on a few snapshots only. 
This function should satisfy the involved PDEs as precisely as possible, 
and the evaluation of this function should be very efficient even though the construction of it can potentially be computationally expensive. 

In the above-described sense, a parametric PDE problem almost perfectly matches the definition of a statistical learning problem. 
Indeed, the PDEs and the metric on the state space correspond to a (deterministic) distribution $\rho$ and a loss function. 
Moreover, the snapshots are construed as the training data, and the offline phase mirrors the learning procedure. Finally, the parametric map is the prediction rule. 

One of the most efficient learning methods nowadays is deep learning. 
This method describes a range of learning procedures to solve statistical learning problems where the hypothesis class $H$ is taken to be a set of DNNs, \cite{LeCun2015DeepLearning, Goodfellow-et-al-2016}. 
These methods outperform virtually all classical machine learning techniques in sufficiently complicated tasks from speech recognition to image classification. Strikingly, training DNNs is a computationally very demanding task that is usually performed on highly parallelized machines. 
Once a DNN is fully trained, however, its application to a given input is orders of magnitudes faster than the training process. 
This observation again reflects the offline-online phase distinction that is common in RBM approaches.

Based on the overwhelming success of these techniques and the apparent similarities of learning problems and parametric problems it appears natural to apply methods from deep learning to statistical learning problems in the sense of (partly) replacing the parameter-dependent map by a DNN.
Very successful advances in this direction have been reported in \cite{Khoo, RBNonlinearProblems, lee2018model, yang2018physics, raissi2018deep, DalSantoParametric}.

\subsection{Our Contribution}

In the applications \cite{Khoo, RBNonlinearProblems,  lee2018model, yang2018physics, raissi2018deep, DalSantoParametric} mentioned above, 
the combination of DNNs and parametric problems seems to be remarkably efficient.
In this paper, we present a theoretical justification of this approach. 
We address the question to what extent the hypothesis class of DNNs is sufficiently broad to approximately and efficiently represent the associated parametric maps. 
Concretely, we aim at understanding the necessary number of parameters of DNNs required to allow a sufficiently accurate approximation. We will demonstrate that depending on the target accuracy the required number of parameters of DNNs essentially only scales with the intrinsic dimension of the solution manifold, in particular, according to its Kolmogorov $N$-widths. 
We outline our results in Subsection \ref{sec:TwoResults}. 
Then, we present a simplified exposition of our argument leading to the main results in Subsection \ref{sec:SimplifiedOverView}.

\subsubsection{Approximation Theoretical Results}\label{sec:TwoResults}

The main contributions of this work is given by an approximation result with DNNs based on ReLU activation functions. Here, we aim to learn a variation of the parametric map 
$$
    \mathcal{Y} \ni y \mapsto u_y \in \Hil,
$$
where $\mathcal{Y}$ is the parameter space and $\Hil$ is a Hilbert space. In our case, the parameter space will be a compact subset of $\R^p$ for some fixed, but possibly large $p\in \N,$ i.e., we consider the case of finitely supported parameter vectors.

We assume that there exists a basis of a high-fidelity discretization of $\Hil$ which may potentially be quite large.
Let $\mathbf{u}_y$ be the coefficient vector of $u_y$ with respect to the high-fidelity discretization. 
Moreover, we assume that there exists a RB approximating $u_y$ sufficiently accurately for every $y\in \mathcal{Y}$. 

Theorem \ref{thm:NNCoefficientApproximation} then states that, under some technical assumptions, there exists a DNN that approximates the \emph{discretized solution map}  
$$
    \mathcal{Y} \ni y \mapsto \mathbf{u}_y
$$
up to a uniform error of $\epsilon>0$, while having a size that is polylogarithmical in $\epsilon$, cubic in the size of the reduced basis, and at most linear in the size of the high-fidelity basis.

This result highlights the common observation that, 
if a low-dimensional structure is present in a problem, then DNNs are able to identify it and use it advantageously.
Concretely, our results show that a DNN is sufficiently flexible to benefit from the existence of a reduced basis in the sense that its size in the complex task of solving a parametric PDE does not or only weakly depend on the high-fidelity discretization and mainly on the size of the reduced basis.

The main result is based on four pillars that are described in detail in Subsection \ref{sec:SimplifiedOverView}: First, we show that DNNs can efficiently solve linear systems, in the sense that, if supplied with a matrix and a right-hand side, a moderately-sized network outputs the solution of the inverse problem. Second, the reduced-basis approach allows reformulating the parametric problem, as a relatively small and parametrized linear system. 
Third, in many cases, the map that takes the parameters to the stiffness matrices with respect to the reduced basis and right-hand side can be very efficiently represented by DNNs. Finally, the fact that neural networks are naturally compositional allows combining the efficient representation of linear problems with the NN implementing operator inversion.

In practice, the approximating DNNs that we show to exist need to be found using a learning algorithm. In this work, we will not analyze the feasibility of learning these DNNs. The typical approach here is to apply methods based on stochastic gradient descent. Empirical studies of this procedure in the context of learning deep neural networks were carried out in \cite{Khoo, RBNonlinearProblems, lee2018model, yang2018physics, raissi2018deep}. In particular, we mention the recent study in \cite{GeiPM2020}, which analyzes precisely the set-up described in this work and finds a strong impact of the approximation-theoretical behavior of DNNs on their practical performance.

\subsubsection{Simplified Presentation of the Argument}\label{sec:SimplifiedOverView}
In this section, we present a simplified outline of the arguments leading to the approximation result described in Subsection \ref{sec:TwoResults}. In this simplified setup, we think of a ReLU neural network (ReLU NN) as a function 
\begin{align}\label{eq:SimplifiedNetwork}
\R^n\to \R^k,~\mathbf{x} \mapsto T_L\varrho(T_{L-1}\varrho( \dots \varrho(T_1(\mathbf{x})))), 
\end{align}
where $L \in \N$, $T_1, \dots, T_L$ are affine maps, and $\varrho: \R \to \R$, $\varrho(x) \coloneqq \max \{ 0, x\}$ is the ReLU activation function which is applied coordinate-wise in \eqref{eq:SimplifiedNetwork}. 
We call $L$ the number of layers of the NN. Since $T_\ell$ are affine linear maps, we have for all $\mathbf{x} \in \mathrm{dom } \, T_\ell$ that $T_\ell(\mathbf{x}) = \mathbf{A}_\ell(\mathbf{x}) + \mathbf{b}_\ell$ for a matrix $\mathbf{A}_\ell$ and a vector $\mathbf{b}_\ell$. 
We define the size of the NN as the number of non-zero entries of all $\mathbf{A}_\ell$ and $\mathbf{b}_\ell$ for $\ell \in \{1, \dots, L\}$. This definition will later be sharpened and extended in Definition \ref{def:NeuralNetworks}.

\begin{enumerate}
    \item As a first step, we recall the construction of a \emph{scalar multiplication operator by ReLU NNs} due to \cite{YAROTSKY2017103}. 
    This construction is based on two observations. First, defining $g\colon [0,1] \to [0,1]$,  $g(x) \coloneqq \min \{2x, 2-2x\}$, we see that $g$ is a hat function. 
    Moreover, multiple compositions of $g$ with itself produce saw-tooth functions. We set, for $s \in \N$, $g_1 \coloneqq g$ and $g_{s+1} \coloneqq g \circ g_s$. 
    It was demonstrated in \cite{YAROTSKY2017103} that
    \begin{align}\label{eq:YarotskySquare}
        x^2 = \lim_{n \to \infty} f_n(x)  \coloneqq \lim_{n \to \infty} x - \sum_{s = 1}^n \frac{g_s(x)}{2^{2s}}, \quad \text{ for all } x \in [0,1].
    \end{align}
    The second observation for establishing an approximation of a scalar multiplication by NNs is that we can write $g(x) = 2\varrho(x)-4\varrho(x-1/2) + 2\varrho(x-2)$ and therefore $g_s$ can be exactly represented by a ReLU NN. 
    Given that $g_s$ is bounded by $1$, it is not hard to see that $f_n$ converges to the square function exponentially fast for $n \to \infty$. 
    Moreover, $f_n$ can be implemented exactly as a ReLU NN by previous arguments. 
    Finally, the parallelogram identity, $xz = 1/4((x+z)^2 - (x-z)^2)$ for $x,z\in \R$, 
    demonstrates how an approximate realization of the square function by ReLU NNs yields an approximate realization of scalar multiplication by ReLU NNs.  
    
    It is intuitively clear from the exponential convergence in \eqref{eq:YarotskySquare} and proved in \cite[Proposition 3]{YAROTSKY2017103} that the size of a NN approximating the scalar multiplication on $[-1,1]^2$ up to an error of $\epsilon>0$ is $\mathcal{O}(\log_2(1/\epsilon))$.
    
    \item As a next step, we use the approximate scalar multiplication to approximate a \emph{multiplication operator for matrices by ReLU NNs}. 
    A matrix multiplication of two matrices of size $d \times d$ can be performed using $d^3$ scalar multiplications. 
    Of course, as famously shown in \cite{strassen1969gaussian}, a more efficient matrix multiplication can also be carried out with less than $d^3$ multiplications. However, for simplicity, we focus here on the most basic implementation of matrix multiplication.   
    Hence, the approximate multiplication of two matrices with entries bounded by $1$ can be performed by NN of size $\mathcal{O}(d^3 \log_2(1/\epsilon))$ with accuracy $\epsilon>0$. 
    We make this precise in Proposition \ref{prop:Multiplikation}. Along the same lines, we can demonstrate how to construct a \emph{NN emulating matrix-vector multiplications}.
    
    \item Concatenating multiple matrix multiplications, we can implement \emph{matrix polynomials by ReLU NNs}. In particular, for $\mathbf{A}\in \R^{d\times d}$ such that $\|\mathbf{A}\|_2 \leq 1-\delta$ for some $\delta \in (0,1)$, the map $\mathbf{A} \mapsto \sum_{s = 0}^m \mathbf{A}^s$ can be approximately implemented by a ReLU NN with an accuracy of $\epsilon>0$ and which has a size of $\mathcal{O}(m \log_2^2(m)d^3 \cdot ( \log(1/\epsilon)+\log_2(m) )$, where the additional $\log_2$ term in $m$ inside the brackets appears since each of the approximations of the sum needs to be performed with accuracy $\epsilon/m$. 
    It is well known, that the \emph{Neumann series} $\sum_{s = 0}^m \mathbf{A}^s$ converges exponentially fast to $(\mathbf{Id}_{\R^d}-\mathbf{A})^{-1}$ for $m \to \infty$. 
    Therefore, under suitable conditions on the matrix $\mathbf{A}$, we can construct a NN $\Phi_{\epsilon}^{\mathrm{inv}}$ that \emph{approximates the inversion operator}, 
    i.e. the map $\mathbf{A} \mapsto \mathbf{A}^{-1}$ up to accuracy $\epsilon>0$. 
    This NN has size $\mathcal{O}(d^3\log_2^q(1/\epsilon))$ for a constant $q>0$. This is made precise in Theorem \ref{thm:Inverse}.
    
    \item The existence of $\Phi_{\epsilon}^{\mathrm{inv}}$ and the emulation of approximate matrix-vector multiplications yield that there exists a NN that for a given matrix and right-hand side approximately solves the associated linear system. 
    Next, we make two assumptions that are satisfied in many applications as we demonstrate in Subsection \ref{sec:Examples}: 
    \begin{itemize}
        \item The map from the parameters to the associated stiffness matrices of the Galerkin discretization of the parametric PDE with respect to a reduced basis can be well approximated by NNs.  
        \item The map from the parameters to the right-hand side of the parametric PDEs discretized according to the reduced basis can be well approximated by NNs. 
    \end{itemize}
    From these assumptions and the existence of $\Phi_{\epsilon}^{\mathrm{inv}}$ and a ReLU NN emulating a matrix-vector multiplication, it is not hard to see that there is a NN that \textit{approximately implements the map from a parameter to the associated discretized solution with respect to the reduced basis}. 
    If the reduced basis has size $d$ and the implementations of the map yielding the stiffness matrix and the right-hand side are sufficiently efficient then, by the construction of $\Phi_{\epsilon}^{\mathrm{inv}}$, 
    the resulting NN has size $\mathcal{O}(d^3\log_2^q(1/\epsilon))$. We call this NN $\Phi_{\epsilon}^{\mathrm{rb}}$.
    \item Finally, we build on the construction of $\Phi_{\epsilon}^{\mathrm{rb}}$ to establish the result of Section \ref{sec:TwoResults}. First of all, 
    let $D$ be the size of the high-fidelity basis. 
    If $D$ is sufficiently large, then every element from the reduced basis can be approximately represented in the high-fidelity basis. 
    Therefore, one can perform an \emph{approximation to a change of bases} by applying a linear map $\mathbf{V} \in \R^{D \times d}$ to a vector with respect to the reduced basis. 
    The first statement of Subsection \ref{sec:TwoResults} now follows directly by considering the NN $\mathbf{V} \circ \Phi_{\epsilon}^{\mathrm{rb}}$. 
    Through this procedure, the size of the NN is increased to $\mathcal{O}(d^3\log_2^q(1/\epsilon)) + d D)$. 
    The full argument is presented in the proof of Theorem \ref{thm:NNCoefficientApproximation}.

\end{enumerate}

\subsection{Potential Impact and Extensions}

We believe that the results of this article have the potential to significantly impact the research on NNs and parametric problems in the following ways:

\begin{itemize}
\item \emph{Theoretical foundation:} We offer a theoretical underpinning for the empirical success of NNs for parametric problems which was observed in, e.g., \cite{Khoo, RBNonlinearProblems,  lee2018model, yang2018physics, raissi2018deep}. Indeed, our result, Theorem \ref{thm:NNCoefficientApproximation}, indicates that properly trained NNs are as efficient in solving parametric PDEs as RBMs if the complexity of NNs is measured in terms of free parameters. 
    On a broader level, linking deep learning techniques for parametric PDE problems with approximation theory opens the field up to a new direction of thorough mathematical analysis.    
    \item \emph{Understanding the role of the ambient dimension:}
    It has been repeatedly observed that NNs seem to offer approximation rates of high-dimensional functions that do not deteriorate exponentially with increasing dimension, 
    \cite{mallat2016understanding, Goodfellow-et-al-2016}.

    In this context, it is interesting to identify the key quantity determining the achievable approximation rates of DNNs. Possible explanation for approximation rates that are essentially independent from the ambient dimension have been identified if the functions to be approximated have special structures such as compositionality, \cite{mhaskar2016learning, poggio2017and}, or invariances, \cite{mallat2016understanding, PetV2018OptApproxReLU}. 
    In this article, we identify the highly problem-specific notion of the \emph{dimension of the solution manifold} as a key quantity determining the achievable approximation rates by NNs for parametric problems. We discuss the connection between the approximation rates that NNs achieve and the ambient dimension in detail in Section \ref{sec:CurseDiscussion}.

    \item \emph{Identifying suitable architectures:}
    One question in applications is how to choose the right NN architectures for the associated problem. 
    Our results show that NNs of sufficient depth and size are able to produce very efficient approximations. 
    Nonetheless, it needs to be mentioned that our results do not yield a lower bound on the number of layers and thus it is not clear whether deep NNs are indeed necessary.      
\end{itemize}

This work is a step towards establishing a theory of deep learning-based solutions of parametric problems. However, given the complexity of this field, 
it is clear that many more steps need to follow. We outline a couple of natural further questions of interest below:

\begin{itemize}
    \item \emph{General parametric problems:} Below we restrict ourselves to coercive, symmetric, and linear parametric problems with finitely many parameters. 
    There exist many extensions to, e.g. noncoercive, nonsymmetric, or nonlinear problems, \cite{veroy2003posteriori, grepl2007efficient, canuto2009posteriori, jung2009reduced,CohenSchwabCurse,ZechSchwabMultilevel}, or to infinite parameter spaces, see e.g. \cite{ bachmayr-cohen-migliorati,BachmayrCohenSPDEs}. It would be interesting to see if the methods proposed in this work can be generalized to these more challenging situations. 
     
    \item \emph{Bounding the number of snapshots:} 
    The interpretation of the parametric problem as a statistical learning problem has the convenient side-effect that various techniques have been established to bound the number of necessary samples $N$, such that the empirical loss \eqref{eq:EmpLoss} is very close to the expected loss.
    In other words, the generalization error of the minimizer of the learning procedure is small, meaning that the prediction rule performs well on unseen data. 
    (Here, the error is measured in a norm induced by the loss function and the underlying probability distribution.). 
    Using these techniques, it is possible to bound the number of snapshots required for the offline phase to achieve a certain fidelity in the online phase. 
    Estimates of the generalization error in the context of high-dimensional PDEs have been deduced in, e.g., \cite{weinan2017deep,grohs2018proof,berner2018analysis, eigel2018variational, Reisinger2019}.

    \item \emph{Special NN architectures:} 
    This article studies the feasibility of standard feed-forward NNs. 
    In practice, one often uses special architectures that have proved efficient in applications. 
    First and foremost, almost all NNs used in applications are convolutional neural networks (CNNs), \cite{lecun1998gradient}. 
    Hence a relevant question is to what extent the results of this work also hold for such architectures. 
    It was demonstrated in \cite{petersen2018equivalence} that there is a direct correspondence between the approximation rates of CNNs and that of standard NNs. 
    Thus we expect that the results of this work translate to CNNs. 
    
    Another successful architecture is that of residual neural networks (ResNets), \cite{he2016deep}.
    These neural networks also admit skip-connections, i.e.,
    do not only connect neurons in adjacent layers. 
    This architecture is by design at least as powerful as a standard NN and hence inherits all approximation properties of standard NNs.
    
    \item \emph{Necessary properties of neural networks:} 
    In this work, we demonstrate the attainability of certain approximation rates by NNs. 
    It is not clear if the presented results are optimal or if there are specific necessary assumptions on the architectures, 
    such as a minimal depth, a minimal number of parameters, or a minimal number of neurons per layer. For approximation results of classical function spaces such lower bounds on specifications of NNs have been established for example in \cite{boelcskeiNeural, ReLUSobolev, PetV2018OptApproxReLU, YAROTSKY2017103}. 
    It is conceivable that the techniques in these works can be transferred to the approximation tasks described in this work.

    \item \emph{General matrix polynomials:} 
    As outlined in Subsection \ref{sec:SimplifiedOverView}, our results are based on the approximate implementation of matrix polynomials. Naturally, this construction can be used to define and construct a ReLU NN based functional calculus. 
    In other words, for any $d \in \N$ and every continuous function $f$ that can be well approximated by polynomials, 
    we can construct a ReLU NN which approximates the map $\mathbf{A} \mapsto f(\mathbf{A})$ for any appropriately bounded matrix $\mathbf{A}$. 
   
    A special instance of such a function of interest is given by  $f(\mathbf{A})\coloneqq e^{t\mathbf{A}},~t>0,$ which is analytic and plays an important role in the treatment of initial value problems.
    
    \item \emph{Numerical studies:} In a practical learning problem, the approximation-theoretical aspect only describes one part of the problem. Two further central factors are the data generation and the optimization process. It is conceivable that in comparison to these issues, approximation theoretical considerations only play a negligible role. To understand the extent to which the result of this paper is relevant for applications, comprehensive studies of the theoretical set-up of this work should be carried out. A first one was published recently in \cite{GeiPM2020}.  
\end{itemize}

\subsection{Related Work}
In this section, we give an extensive overview of works related to this paper. In particular, for completeness, we start by giving a review of approximation theory of NNs without an explicit connection to PDEs. Afterward, we will see how NNs have been employed for the solution of PDEs. 

\subsubsection{Review of Approximation Theory of Neural Networks}\label{sec:approximationTheoryOfDNNs}

The first and most fundamental results on the approximation capabilities of NNs were universality results. 
These results claim that NNs with at least one hidden layer can approximate any continuous function on a bounded domain to arbitrary accuracy if they have sufficiently many neurons,
\cite{Hornik1989universalApprox, Cybenko1989}. 
However, these results do not quantify the required sizes of NNs to achieve these rates. 
One of the first results in this direction was given in \cite{Barron1993}. 
There, a bound on the sufficient size of NNs with sigmoidal activation functions approximating a function with finite Fourier moments is presented. 
Further results describe approximation rates for various smoothness classes by sigmoidal or even more general activation functions, \cite{ Mhaskar:1996:NNO:1362203.1362213, Mhaskar1993, PinkusUniversalApproximation, Maiorov1999LowerBounds}.

For the non-differentiable activation function ReLU, 
first rates of approximation were identified in \cite{YAROTSKY2017103} for classes of smooth functions, in \cite{PetV2018OptApproxReLU} for piecewise smooth functions, and in \cite{grohs2019deep} for oscillatory functions. 
Moreover, NNs mirror the approximation rates of various dictionaries such as wavelets,
\cite{shaham2018provable}, general affine systems, \cite{boelcskeiNeural}, linear finite elements, \cite{he2018relu}, and higher-order finite elements, \cite{OPS19_811}.

\subsubsection{Neural Networks and PDEs}
A well-established line of research is that of solving high-dimensional PDEs by NNs assuming that the NN is the solution of the underlying PDE, e.g., \cite{DGM, berner2018analysis, Curse, JentzenKolmogorov, Reisinger2019, JentzenHeat, weinan2017deep,eigel2018variational}. 
In this regime, it is often possible to bound the size of the involved NNs in a way that does not scale exponentially with the underlying dimension. 
In that way, these results are quite related to our approaches. 
Our results do not seek to represent the solution of a PDE as a NN, but a parametric map. 
Moreover, we analyze the complexity of the solution manifold in terms of Kolmogorov $N$-widths.
Finally, the underlying spatial dimension of the involved PDEs in our case would usually be moderate.
However, the dimension of the parameter space could be immense.

One of the first approaches analyzing NN approximation rates for solutions of parametric PDEs was carried out in \cite{SchwabZech}. In that work, the analyticity of the solution map $y\mapsto u_y$ and polynomial chaos expansions with respect to the parametric variable are used to approximate the map $y\mapsto u_y$ by ReLU NNs of moderate size.
Moreover, we mention the works \cite{Khoo, RBNonlinearProblems,  lee2018model, yang2018physics, raissi2018deep} which apply NNs in one way or another to parametric problems. 
These approaches study the topic of learning a parametric problem but do not offer a theoretical analysis of the required sizes of the involved NNs. These results form our motivation to study the constructions of this paper.

Finally, we mention that the setup of the recent numerical study \cite{GeiPM2020} is closely related to this work.

 \subsection{Outline}
 
 In Section \ref{sec:RedBase}, we describe the type of parametric PDEs that we  consider in this paper, and we recall the theory of RBs. 
 Section \ref{sec:NNCalc} introduces a NN calculus which is the basis for all constructions in this work. 
 There we will also construct the NN that maps a matrix to its approximate inverse in Theorem \ref{thm:Inverse}. 
 In Section \ref{sec:NNPDE}, we construct NNs approximating parametric maps. 
 First, in Theorem \ref{subsec:CoeffNNs}, we approximate the parametric maps after a high-fidelity discretization. 
 Afterward, in Subsection \ref{sec:Examples}, we list two broad examples where all assumptions which we imposed are satisfied.
 
We conclude this paper in Section \ref{sec:CurseDiscussion} with a discussion of our results in light of the dependence of the underlying NN complexities in terms of the governing quantities.

 To not interrupt the flow of reading, we have deferred all auxiliary results and proofs to the appendices.
 
 \subsection{Notation}
 We denote by $\N=\{1,2,...\}$ the set of all \emph{natural numbers} and define $\N_0\coloneqq \N\cup \{0\}.$ Moreover, for $a\in \R$ we set $\lfloor a \rfloor\coloneqq \max\{b\in \Z\colon~b\leq a\} $ and $\lceil a\rceil\coloneqq \min\{b\in \Z\colon~b\geq a\}. $ Let $n,l\in \N.$
 Let $\mathbf{Id}_{\R^n}$ be the \emph{identity} and $\mathbf{0}_{\R^n}$ be the \emph{zero vector} on $\R^n.$ 
 Moreover, for $\mathbf{A}\in \R^{n\times l}$, we denote by $\mathbf{A}^T$ its \emph{transpose,} by $\sigma(\mathbf{A})$ the \emph{spectrum of $\mathbf{A}$}, 
 by $\|\mathbf{A}\|_2$ its \emph{spectral norm} and by $\|\mathbf{A}\|_0\coloneqq \# \{(i,j)\colon\mathbf{A}_{i,j}\neq 0\},$ where $\# V$ denotes the cardinality of a set $V$, 
 the \emph{number of non-zero entries} of $\mathbf{A}$. Moreover, for $\mathbf{v}\in \R^n$ we denote by $|\mathbf{v}|$ its \emph{Euclidean norm}.
 Let $V$ be a vector space. 
 Then we say that $X\subset^{\mathrm{s}}V,$ if $X$ is a \emph{linear subspace} of $V.$ 
 Moreover, if $(V,\|\cdot\|_V)$ is a normed vector space, $X$ is a subset of $V$ and $v\in V,$ we denote by $\mathrm{dist}(v,X)\coloneqq  \inf\{\|x-v\|_V\colon~x\in X\}$ the \emph{distance} between $v,X$ and by $(V^*,\|\cdot\|_{V^*})$ the \emph{topological dual space of $V$,}
 i.e. the set of all scalar-valued, linear, continuous functions equipped with the \emph{operator norm}. 
 For a compact set $\Omega\subset \R^n$ we denote by $C^r(\Omega),~r\in \N_0\cup\{\infty\},$ the spaces of \emph{$r$ times continuously differentiable functions}, 
 by $L^p(\Omega,\R^n),p\in [1,\infty]$ the \emph{$\R^n$-valued Lebesgue spaces}, where we set $L^p(\Omega)\coloneqq L^p(\Omega,\R)$ and by $H^1(\Omega)\coloneqq W^{1,2}(\Omega)$ the \emph{first-order Sobolev space}.

\section{Parametric PDEs and Reduced Basis Methods}\label{sec:RedBase}

In this section, we introduce the type of parametric problems that we study in this paper. 
A parametric problem in its most general form is based on a map $\mathcal{P}\colon \mathcal{Y} \to \mathcal{Z}$, where $\mathcal{Y}$ is the \emph{parameter space} and $\mathcal{Z}$ is called \emph{solution state space} $\mathcal{Z}$. 
In the case of parametric PDEs, $\mathcal{Y}$ describes certain parameters of a partial differential equation, $\mathcal{Z}$ is a function space or a discretization thereof, and $\mathcal{P}(y)\in \mathcal{Z}$ is found by solving a PDE with parameter $y$. 

We will place several assumptions on the PDEs underlying $\mathcal{P}$ and the parameter spaces $\mathcal{Y}$ in Section \ref{subsec:Setup}. Afterward, we give an abstract overview of Galerkin methods in Section \ref{subsec:Galerkin} before recapitulating some basic facts about RBs in Section \ref{subsec:RBIntro}.

\subsection{Parametric Partial Differential Equations}\label{subsec:Setup}

In the following, we will consider parameter-dependent equations given in the variational form 
\begin{align}\label{eq:opeq}
    b_y\left(u_y, v\right) =  f_y(v), \quad \text{ for all } y\in \Ycal,~v\in \Hil,
\end{align}
where
\begin{compactenum}[(i)]
    \item $\mathcal{Y}$ is the \emph{parameter set} specified in Assumption \ref{ass:opeq},
    \item $\mathcal{H}$ is a Hilbert space,
    \item $b_y\colon\Hil\times \Hil\to \mathbb{R}$ is a \emph{continuous bilinear form}, which fulfills certain well-posedness conditions specified in Assumption \ref{ass:opeq},
    \item $f_{y}\in \Hil^*$ is the \emph{parameter-dependent right-hand side of \eqref{eq:opeq}},
    \item  $u_y\in \Hil$ is the \emph{solution of \eqref{eq:opeq}}. 
\end{compactenum}

\begin{assumption} \label{ass:opeq}
Throughout this paper, we impose the following assumptions on Equation \eqref{eq:opeq}.
\begin{itemize}
    \item \textbf{The parameter set $\mathcal{Y}$}: We assume that $\Ycal$ is a compact subset of $\R^p,$ where $p\in \N$ is fixed and potentially large. 
    \begin{remark*}
    In \cite[Section 1.2]{CohenDeVoreHighPDE}, it has been demonstrated that if $\Ycal$ is a compact subset of some Banach space $V$, then one can describe every element in $\Ycal$ by a sequence of real numbers in an affine way. 
    To be more precise, there exist $(v_i)_{i=0}^\infty\subset V$ such that for every $y\in \Ycal$ and some coefficient sequence $\mathbf{c}_y$ whose elements can be bounded in absolute value by 1 there holds $y=v_0 +\sum_{i=1}^\infty (\mathbf{c}_y)_i v_i,$ implying that $\Ycal$ can be completely described by the collection of sequences $\mathbf{c}_y$. 
    In this paper, we assume these sequences $\mathbf{c}_y$ to be finite with a fixed, but possibly large support size.
    \end{remark*}
    
    \item \textbf{Symmetry, uniform continuity, and coercivity of the bilinear forms:} We assume that for all $y\in \Ycal$ the bilinear forms $b_y$ are \emph{symmetric}, i.e.
    \begin{align*}
        b_y(u,v)=b_y(v,u), \qquad \text{ for all } u,v\in \Hil.
    \end{align*}
    Moreover, we assume that the bilinear forms $b_y$ are \emph{uniformly continuous} in the sense that there exists a constant $C_{\mathrm{cont}} > 0$ with
    \begin{align*}
        \left|b_y(u,v) \right| \leq C_{\mathrm{cont}}\|u\|_{\Hil}\|v\|_{\Hil}, \qquad \text{ for all } u\in \Hil,v\in \Hil,y\in \Ycal.
    \end{align*}
    
   Finally, we assume that the involved bilinear forms are \emph{uniformly coercive} in the sense that there exists a constant $C_{\mathrm{coer}}>0$ such that 
      \begin{align*}
       \inf_{u\in\Hil\setminus\{0\}} \frac{b_y(u,u)}{\|u\|_{\Hil}^2}  \geq C_{\mathrm{coer}}, \qquad \text{ for all } u\in \Hil,y\in \Ycal.
   \end{align*}
   Hence, by the Lax-Milgram lemma  (see \cite[Lemma 2.1]{QuarteroniIntro}), Equation \eqref{eq:opeq} is \emph{well-posed}, i.e. for every $y\in \Ycal$ and every $f_y\in \Hil^*$ there exists exactly one $u_y\in \Hil $ such that \eqref{eq:opeq} is satisfied and $u_y$ depends continuously on $f_y$. 
   
    \item \textbf{Uniform boundedness of the right-hand side:} We assume that there exists a constant $C_{\mathrm{rhs}}>0$ such that 
    \begin{align*}
        \left\|f_y\right\|_{\Hil^*}\leq C_{\mathrm{rhs}}, \qquad \text{ for all } y\in \Ycal.
    \end{align*}
    
     \item \textbf{Compactness of the solution manifold:} We assume that the \emph{solution manifold}
     \begin{align*}
   S(\Ycal)\coloneqq \left\{u_y\colon u_y \text{ is the solution of } \eqref{eq:opeq},~ y\in \Ycal \right\}
    \end{align*}
    is compact in $\Hil.$
    
    \begin{remark*}
        The assumption that $S(\Ycal)$ is compact follows immediately if the solution map $y\mapsto u_y$ is continuous. 
        This condition is true (see \cite[Proposition 5.1, Corollary 5.1]{QuarteroniIntro}), if for all $u,v\in \Hil$ the maps $y\mapsto b_y(u,v)$  as well as $y\to f_y(v)$ are Lipschitz continuous. 
        In fact, there exists a multitude of parametric PDEs, for which the maps $y\mapsto b_y(u,v)$ and $y\to f_y(v)$ are even in $C^{r}$ for some $r\in \N\cup\{\infty\}.$ 
        In this case, $\left\{\left(y,u_y\right)\colon y\in \Ycal \right\}\subset \R^p\times \Hil $ is a $p$-dimensional manifold of class $C^r$ (see \cite[Proposition 5.2, Remark 5.4]{QuarteroniIntro}). 
        Moreover, we refer to \cite[Remark 5.2]{QuarteroniIntro} and the references therein for a discussion under which circumstances it is possible to turn a discontinuous parameter dependency into a continuous one ensuring the compactness of $S(\Ycal)$.
    \end{remark*}
\end{itemize}
\end{assumption}

\subsection{High-Fidelity Approximations} \label{subsec:Galerkin}

In practice, one cannot hope to solve \eqref{eq:opeq} exactly for every $y\in \Ycal$. Instead, 
if we assume for the moment that $y$ is fixed,
a common approach towards the calculation of an approximate solution of \eqref{eq:opeq} is given by the \emph{Galerkin method}, which we will describe shortly following \cite[Appendix A]{CertReduced} and \cite[Chapter 2.4]{QuarteroniIntro}. In this framework, instead of solving \eqref{eq:opeq}, one solves a discrete scheme of the form 
\begin{align}\label{eq:discopeq}
    b_y\left(u^{\mathrm{disc}}_y , v\right) = f_y(v) \qquad \text{ for all } v\in U^{\mathrm{disc}},
\end{align}
where $U^{\mathrm{disc}}\subset^{\mathrm{s}}\Hil$ is a subspace of $\Hil$ with $\mathrm{dim}\left(U^{\mathrm{disc}} \right) < \infty$ and $u^{\mathrm{disc}}_y\in U^{\mathrm{disc}}$ is the solution of \eqref{eq:discopeq}. For the solution $u^{\mathrm{disc}}_y$ of \eqref{eq:discopeq} we have that
\begin{align*}
    \left\|u^{\mathrm{disc}}_y \right\|_{\Hil} \leq \frac{1}{C_{\mathrm{coer}}}\left\|f_y \right\|_{\Hil^*}.
\end{align*}
Moreover, up to a constant, we have that $u^{\mathrm{disc}}_y$ is a best approximation of the solution $u_y$ of \eqref{eq:opeq} by elements in $U^{\mathrm{disc}}$. To be more precise, by \emph{Cea's Lemma}, \cite[Lemma 2.2]{QuarteroniIntro},  
\begin{align}\label{eq:Cea}
    \left\|u_y-u^{\mathrm{disc}}_y \right\|_{\Hil} \leq \frac{C_{\mathrm{cont}}}{C_{\mathrm{coer}}} \inf_{w\in U^{\mathrm{disc}}}\left\|u_y-w \right\|_{\Hil }.
\end{align}
 Let us now assume that $U^\mathrm{disc}$ is given. Moreover, if $N\coloneqq \mathrm{dim}\left(U^{\mathrm{disc}}\right)$, let $\left(\varphi_{i} \right)_{i=1}^N$ be a basis for $U^{\mathrm{disc}}$. Then the matrix 
\begin{align*}
    \mathbf{B}_y\coloneqq \left(b_y\left( \varphi_{j},\varphi_{i}\right) \right)_{i,j=1}^N
\end{align*}
is non-singular and positive definite. The solution $u^{\mathrm{disc}}_y$ of \eqref{eq:discopeq} satisfies 
\begin{align*}
    u^{\mathrm{disc}}_y= \sum_{i=1}^N (\mathbf{u}_y)_{i} \varphi_{i},
\end{align*}
where 
\begin{align*}
    \mathbf{u}_y\coloneqq \left(\mathbf{B}_y\right)^{-1}\mathbf{f}_y \in \R^N
\end{align*}
and $\mathbf{f}_y\coloneqq \left( f_y\left(\varphi_{i}\right) \right)_{i=1}^N\in \R^N$. Typically, one starts with a \emph{high-fidelity discretization} 
of the space $\Hil$, i.e. one chooses a finite- but potentially high-dimensional subspace for which the computed discretized solutions are sufficiently accurate for any $y\in \Ycal$. To be more precise, we postulate the following: 

\begin{assumption} \label{ass:high-fidelity}
    We assume that there exists a finite dimensional space $U^\mathrm{h}\subset^{\mathrm{s}}\Hil$ with dimension $D<\infty$ and basis $(\varphi_i)_{i=1}^D$. This space is called \emph{high-fidelity discretization}. 
    For $y\in \Ycal,$ denote by $\mathbf{B}_{y}^{\mathrm{h}}\coloneqq \left(b_y(\varphi_j,\varphi_i)\right)_{i,j=1}^{D}\in \R^{D\times D}$ the stiffness matrix of the high-fidelity discretization, 
    by $\mathbf{f}_{y}^{\mathrm{h}}\coloneqq \left(f_y(\varphi_i) \right)_{i=1}^{D}$ the discretized right-hand side, and by $\mathbf{u}_{y}^{\mathrm{h}}\coloneqq \left(\mathbf{B}_{y}^{\mathrm{h}} \right)^{-1}\mathbf{f}_{y}^{\mathrm{h}}\in \R^{D}$ the coefficient vector of the Galerkin solution with respect to the high-fidelity discretization.
    Moreover, we denote by $u^\mathrm{h}_y\coloneqq \sum_{i=1}^{D} \left(\mathbf{u}_{y}^{\mathrm{h}}\right)_i \varphi_i$ the Galerkin solution. 
    
    We assume that, for every $y\in \Ycal$, $\sup_{y\in \Ycal} \left\|u_y-u^{\mathrm{h}}_y \right\|_\Hil\leq \epsilonhat$ for an arbitrarily small, but fixed $\epsilonhat>0$. In the following, similarly as in \cite{DahmenSampleSolution}, we will not distinguish between $\Hil$ and $U^{\mathrm{h}}$, unless such a distinction matters.
\end{assumption}

In practice, following this approach, one often needs to calculate $u^{\mathrm{h}}_y\approx u_y$ for a variety of parameters $y\in \Ycal$. This, in general, is a very expensive procedure due to the high-dimensionality of the space $U^{\mathrm{h}}.$ 
In particular, given $\left(\varphi_{i} \right)_{i=1}^{D},$ 
one needs to solve high-dimensional systems of linear equations to determine the coefficient vector $\mathbf{u}^{\mathrm{h}}_y.$ A well-established remedy to overcome these difficulties is given by methods based on the theory of reduced bases, which we will recapitulate in the upcoming subsection. 

Before we proceed, let us fix some notation.
We denote by $\mathbf{G}\coloneqq \left(\langle \varphi_i,\varphi_j\rangle_\Hil\right)_{i,j=1}^{D}\in \R^{D\times D}$ the symmetric, 
positive definite \emph{Gram matrix} of the basis vectors $(\varphi_i)_{i=1}^{D}.$ 
Then, for any $v\in U^{\mathrm{h}}$ with coefficient vector $\mathbf{v}$ with respect to the basis $(\varphi_i)_{i=1}^{D}$ we have (see \cite[Equation 2.41]{QuarteroniIntro}) 
\begin{align} \label{eq:NormCoefficientHilbert}
\left| \mathbf{v}\right|_{\mathbf{G}}\coloneqq \left|\mathbf{G}^{1/2}\mathbf{v}\right|= \|v\|_\Hil.
\end{align} 

\subsection{Theory of Reduced Bases} \label{subsec:RBIntro}
In this subsection and unless stated otherwise, we follow \cite[Chapter 5]{QuarteroniIntro} and the references therein. The main motivation behind the theory of RBs lies in the fact that under  Assumption \ref{ass:opeq} the solution manifold $S(\Ycal)$
is a compact subset of $\Hil.$
This compactness property allows posing the question whether, for every $\epsilontilde\geq \epsilonhat$, it is possible to construct a finite-dimensional subspace $U^{\mathrm{rb}}_\epsilontilde$ of $\Hil $ such that $d(\epsilontilde)\coloneqq \mathrm{dim}\left(U^{\mathrm{rb}}_\epsilontilde\right) \ll D$  and such that
\begin{align}\label{eq:RBError}
    \sup_{y\in \Ycal} \inf_{w\in U^{\mathrm{rb}}_\epsilontilde}\left\|u_y- w \right\|_{\Hil } \leq \epsilontilde,
\end{align}
or, equivalently, if there exist linearly independent vectors $\left(\psi_{i} \right)_{i=1}^{d(\epsilontilde)}$ with the property that  
\begin{align*}
    \left\|\sum_{i=1}^{d(\epsilontilde)} (\mathbf{c}_y)_i \psi_{i} - u_y \right\|_\Hil \leq \epsilontilde, \quad \text{ for all } y\in \Ycal \text{ and some coefficient vector } \mathbf{c}_y\in \R^{d(\epsilontilde)}.
\end{align*}
The starting point of this theory lies in the concept of the Kolmogorov $N$-width which is defined as follows.

\begin{definition}[\cite{DahmenSampleSolution}]
     For $N\in\N,$ the \emph{Kolmogorov $N$-width} of a bounded subset $X$ of a normed space $V$ is defined by 
    $$
    W_N(X) \coloneqq \inf_{\substack{V_N \subset^\mathrm{s} V \\\mathrm{dim}(V_N) \leq N}}\sup_{x\in X} \mathrm{dist}\left(x,V_N\right).
    $$
\end{definition}
    This quantity describes the best possible uniform approximation error of $X$ by an at most $N$-dimensional linear subspace of $V$.We discuss concrete upper bounds on $W_N(S(\Ycal))$ in more detail in Section \ref{sec:CurseDiscussion}.
    The aim of RBMs is to construct the spaces $U_{\epsilontilde}^{\mathrm{rb}}$ in such a way that the quantity $ \sup_{y\in \Ycal}\mathrm{dist}\left(u_y,U^{\mathrm{rb}}_{\epsilontilde}\right)$ is close to $W_{d(\epsilontilde)}\left(S(\Ycal)\right)$. 
    
    The identification of the basis vectors $(\psi_i)_{i=1}^{d(\epsilontilde)}$ of  $U_{\epsilontilde}^{\mathrm{rb}}$ usually happens in an \emph{offline phase} in which one has considerable computational resources available and which is usually based on the determination of high-fidelity discretizations of samples of the parameter set $\Ycal.$ 
    The most common methods are based on {(weak) greedy procedures} (see for instance \cite[Chapter 7]{QuarteroniIntro} and the references therein) or {proper orthogonal decompositions} (see for instance \cite[Chapter 6]{QuarteroniIntro} and the references therein). 
    In the last step, an orthogonalization procedure (such as a Gram-Schmidt process) is performed to obtain an orthonormal set of basis vectors $(\psi_i)_{i=1}^{d(\epsilontilde)}$. 
    
    Afterward, in the \emph{online phase}, one assembles for a given input $y$ the corresponding low-dimensional stiffness matrices and vectors and determines the Galerkin solution by solving a low-dimensional system of linear equations. 
    To ensure an efficient implementation of the online phase, a common assumption which we do \emph{not} require in this paper is the \emph{affine decomposition} of \eqref{eq:opeq}, 
    which means that there exist $Q_b,Q_f\in \N,$ parameter-independent bilinear forms $b^q\colon\Hil\times \Hil \to \R,$ maps $\theta_q\colon\Ycal\to \R$ for $q=1,...,Q_b,$ parameter-independent $f^{q'}\in \Hil^*$  as well as maps  $\theta^{q'}\colon\Ycal\to \R$ for $q'=1,...,Q_f$ such that 
    \begin{align}\label{eq:affineProblem}
        b_y = \sum_{q=1}^{Q_b} \theta_q(y)b^q, \qquad \text{as well as}\qquad f_y = \sum_{q'=1}^{Q_f} \theta^{q'}(y)f^{q'}, \qquad \text{for all } y\in \Ycal.  
    \end{align}
    As has been pointed out in \cite[Chapter 5.7]{QuarteroniIntro}, in principal three types of reduced bases generated by RBMs have been established in the literature - the \emph{Lagrange reduced basis}, 
    the \emph{Hermite reduced basis} and the \emph{Taylor reduced basis.}  
    While the most common type, the Lagrange RB, consists of orthonormalized versions of  high-fidelity \emph{snapshots} $u^{\mathrm{h}}\left(y^1\right)\approx u\left(y^1\right),...,u^{\mathrm{h}}\left(y^{n}\right)\approx u\left(y^{n}\right),$ 
    Hermite RBs consist of snapshots $u^{\mathrm{h}}\left(y^1\right)\approx u\left(y^1\right),...,u^{\mathrm{h}}\left(y^{n}\right)\approx u\left(y^{n}\right),$ as well as their first partial derivatives $\frac{\partial u^{\mathrm{h}}}{\partial y_i}(y^j)\approx \frac{\partial u}{\partial y_i}(y^j) ,i=1,...,p,~j=1,...,n$, 
    whereas Taylor RBs are built of derivatives of the form $\frac{\partial^k u^{\mathrm{h}}}{\partial y_i^k}(\overline{y})\approx\frac{\partial^k u}{\partial y_i^k}(\overline{y}),~i=1,...,p,~k=0,...,n-1$ around a given expansion point $\overline{y}\in \Ycal$.
    In this paper, we will later assume that there exist small RBs $(\psi_i)_{i=1}^{d(\epsilontilde)}$ generated by \emph{arbitrary} linear combinations of the high-fidelity elements $(\varphi_i)_{i=1}^{D}$. Note that all types of RBs discussed above satisfy this assumption. 
    
    The next statement gives a (generally sharp) upper bound which relates the possibility of constructing small snapshot RBs directly to the Kolmogorov $N$-width. 
    
    \begin{theorem}[\protect{\cite[Theorem 4.1.]{CohenDahmenGreedy}}] \label{thm:InnerNWidth}
     Let $N\in \N.$ For a compact subset $X$ of a normed space $V,$ define the \emph{inner $N$-width of $X$} by 
     \begin{align*}
         \overline{W}_N(X)\coloneqq \inf_{V_N\in \mathcal{M}_N} \sup_{x\in X} \mathrm{dist}\left(x,V_N\right),
     \end{align*}
     where $\mathcal{M}_N\coloneqq\left\{V_N\subset^{\mathrm{s}}V\colon~V_N=\mathrm{span}\left(x_i\right)_{i=1}^N,x_1,...,x_N\in X \right\}.$
     Then 
     \begin{align*}
         \overline{W}_N(X) \leq (N+1) W_N(X).
     \end{align*}
    \end{theorem}
    
    Translated into our framework, Theorem \ref{thm:InnerNWidth} states that for every $N\in \N,$ there exist solutions $u^{\mathrm{h}}(y^i)\approx u(y^i),$ ~$i=1,...,N$ of \eqref{eq:opeq} such that 
    \begin{align*}
        \sup_{y\in \Ycal} ~\inf_{w\in\mathrm{span}\left(u^{\mathrm{h}}(y^i)\right)_{i=1}^{N}} \left\|u_y-w \right\|_\Hil \leq (N+1)W_N(S(\mathcal{Y})).
    \end{align*}

    \begin{remark}\label{rem:NWidthExamples}
    We note that this bound is sharp for general $X,V$. However, it is not necessarily optimal for special instances of $S(\Ycal)$. If, for instance, $W_N(S(\Ycal))$ decays polynomially, then $\overline{W}_N(S(\Ycal))$ decays with the same rate (see \cite[Theorem 3.1.]{CohenDahmenGreedy}). Moreover, if $W_N(S(\Ycal)) \leq C e^{-cN^{\beta}}$ for some $c,C,\beta>0$ then by \cite[Corollary 3.3 (iii)]{DeVoreGreedy} we have $ \overline{W}_N(S(\Ycal)) \leq \tilde{C}e^{-\tilde{c}N^{\beta}}$ for some $\tilde{c},\tilde{C}>0.$ 
    \end{remark}
    
    Taking the discussion above as a justification, we assume from now on that for every $\epsilontilde\geq \epsilonhat$ there exists a RB space $U_\epsilontilde^{\mathrm{rb}}=\mathrm{span} \left(\psi_i \right)_{i=1}^{d(\epsilontilde)},$ which fulfills \eqref{eq:RBError}, where the linearly independent basis vectors  $\left(\psi_i \right)_{i=1}^{d(\epsilontilde)}$ are linear combinations of the high-fidelity basis vectors $\left(\varphi_i\right)_{i=1}^{D}$ in the sense that there exists a transformation matrix $\mathbf{V}_\epsilontilde\in \R^{D\times d(\epsilontilde)}$ such that 
\begin{align*}
    \left(\mathbf{\psi}_i\right)_{i=1}^{d(\epsilontilde)}= \left(\sum_{j=1}^{D} (\mathbf{V}_\epsilontilde)_{j,i}\varphi_j\right)_{i=1}^{d(\epsilontilde)} 
\end{align*}
and where $d(\epsilontilde)\ll D$ is chosen to be as small as possible, at least fulfilling $\mathrm{dist}\left(S(\Ycal),U_\epsilontilde^{\mathrm{rb}}\right)\leq \overline{W}_{d(\epsilontilde)}(S(\Ycal)).$
In addition, we assume that the vectors $(\psi_i)_{i=1}^{d(\epsilontilde)}$ form an orthonormal system in $\Hil,$ which is equivalent to the fact that the columns of $\mathbf{G}^{1/2}\mathbf{V}_\epsilontilde$ are orthonormal  (see \cite[Remark 4.1]{QuarteroniIntro}). This in turn implies
\begin{align}\label{eq:NormTrafo}
    \left\|\mathbf{G}^{1/2}\mathbf{V}_\epsilontilde \right\|_2=1, \quad \text{ for all } \epsilontilde\geq \epsilonhat  
\end{align}
as well as 
\begin{align}\label{eq:NormRB}
    \left\|\sum_{i=1}^{d(\epsilontilde)} \mathbf{c}_i \psi_i \right\|_{\Hil} = \left| \mathbf{c}\right|, \quad \text{ for all } \mathbf{c}\in \R^{d(\epsilontilde)}.
\end{align}
For the underlying discretization matrix, one can demonstrate (see for instance \cite[Section 3.4.1]{QuarteroniIntro}) that 
\begin{align*}
    \mathbf{B}_{y,\epsilontilde}^{\mathrm{rb}} \coloneqq  \left(b_y(\psi_j,\psi_i) \right)_{i,j=1}^{d(\epsilontilde)}=  \mathbf{V}_\epsilontilde^T  \mathbf{B}_{y,\epsilontilde}^{\mathrm{h}} \mathbf{V}_\epsilontilde\in  \R^{d(\epsilontilde)\times d(\epsilontilde)}, \quad \text{ for all } y\in \Ycal.
\end{align*}
Moreover, due to the symmetry and the coercivity of the underlying bilinear forms combined with the orthonormality of the basis vectors $(\psi_i)_{i=1}^{d(\epsilontilde)}$, one can show (see for instance \cite[Remark 3.5]{QuarteroniIntro}) that 
\begin{align} \label{eq:NormStiffness}
    C_{\mathrm{coer}} \leq \left\|\mathbf{B}_{y,\epsilontilde}^{\mathrm{rb}} \right\|_2 \leq C_{\mathrm{cont}}, \quad \text{ as well as }\quad \frac{1}{C_{\mathrm{cont}}} \leq  \left\|\left(\mathbf{B}_{y,\epsilontilde}^{\mathrm{rb}}\right)^{-1} \right\|_2\leq \frac{1}{C_{\mathrm{coer}}}, \quad \text{ for all } y\in \Ycal,
\end{align}
 implying that the condition number of the stiffness matrix with respect to the RB remains bounded independently of $y$ and the dimension $d(\epsilontilde).$ Additionally, the discretized right-hand side with respect to the RB is given by 
\begin{align*}
    \mathbf{f}_{y,\epsilontilde}^{\mathrm{rb}} \coloneqq \left( f_y(\psi_i) \right)_{i=1}^{d(\epsilontilde)}= \mathbf{V}_\epsilontilde^T \mathbf{f}_{y,\epsilontilde}^{\mathrm{h}} \in \R^{d(\epsilontilde)}
\end{align*}
and, by the Bessel inequality, we have that $\left|\mathbf{f}_{y,\epsilontilde}^{\mathrm{rb}} \right| \leq \left\|f_{y}\right\|_{\Hil^*} \leq C_{\mathrm{rhs}}.$ Moreover, let %
\begin{align*}
    \mathbf{u}_{y,\epsilontilde}^{\mathrm{rb}}\coloneqq \left(\mathbf{B}^{\mathrm{rb}}_{y,\epsilontilde}\right)^{-1} \mathbf{f}_{y,\epsilontilde}^{\mathrm{rb}} 
\end{align*}
be the coefficient vector of the Galerkin solution with respect to the RB space. Then, the Galerkin solution $u_{y,\epsilontilde}^{\mathrm{rb}}$ can be written as
    \begin{align*}
        u_{y,\epsilontilde}^{\mathrm{rb}}= \sum_{i=1}^{d(\epsilontilde)} \left(\mathbf{u}_{y,\epsilontilde}^{\mathrm{rb}} \right)_i \psi_i = \sum_{j=1}^{D} \left(\mathbf{V}_\epsilontilde \mathbf{u}_{y,\epsilontilde}^{\mathrm{rb}} \right)_j \varphi_j,
    \end{align*}
i.e. 
$$\tilde{\mathbf{u}}_{y,\epsilontilde}^{\mathrm{h}}\coloneqq \mathbf{V}_\epsilontilde  \mathbf{u}_{y,\epsilontilde}^{\mathrm{rb}}\in \R^{D}$$ 
is the coefficient vector of the RB solution if expanded with respect to the high-fidelity basis $\left(\varphi_i \right)_{i=1}^{D}.$
Finally, as in Equation \ref{eq:Cea}, we obtain
\begin{align*}
     \sup_{y\in \Ycal} \left\|u_y- u_{y,\epsilontilde}^{\mathrm{rb}}\right\|_\Hil \leq \sup_{y\in\Ycal} \frac{C_{\mathrm{cont}}}{C_{\mathrm{coer}}} \inf_{w\in U^{\mathrm{rb}}_{\epsilontilde}} \left\|u_y-w \right\|_{\Hil} \leq \frac{C_{\mathrm{cont}}}{C_{\mathrm{coer}}} \epsilontilde.
\end{align*}

In the following sections, we will emulate RBMs with NNs by showing that we are able to construct NNs which approximate the maps $\mathbf{u}_{\cdot,\epsilontilde}^{\mathrm{rb}},\tilde{\mathbf{u}}_{\cdot,\epsilontilde}^{\mathrm{h}}$ such that their complexity
depends only on the size of the reduced basis and at most linearly on $D$.  The key ingredient will be the construction of small NNs implementing an approximate matrix inversion based on Richardson iterations in Section \ref{sec:NNCalc}. In Section \ref{sec:NNPDE}, we then proceed with building the NNs the realizations of which approximate the maps  $\mathbf{u}_{\cdot,\epsilontilde}^{\mathrm{rb}},\tilde{\mathbf{u}}_{\cdot,\epsilontilde}^{\mathrm{h}}$, respectively. 

\section{Neural Network Calculus }\label{sec:NNCalc}
The goal of this chapter is to emulate the matrix inversion by NNs. In Section \ref{subsec:BasicNN}, we introduce some basic notions connected to NNs as well as some basic operations one can perform with these. In Section \ref{subsec:matrInv}, we state a result which shows the existence of NNs the ReLU-realizations of which take a matrix $\mathbf{A}\in \R^{d\times d},~\|\mathbf{A}\|_2<1$ as their input and calculate an approximation of $\left(\mathbf{Id}_{\R^d}-\mathbf{A}\right)^{-1}$ based on its Neumann series expansion. The associated proofs can be found in Appendix \ref{app:NNCalculusProofs}.  

\subsection{Basic Definitions and Operations} \label{subsec:BasicNN}

We start by introducing a formal definition of a NN. Afterward, we introduce several operations, such as parallelization and concatenation that can be used to assemble complex NNs out of simpler ones.
Unless stated otherwise we follow the notion of \cite{PetV2018OptApproxReLU} where most of this formal framework was introduced.
First, we introduce a terminology for NNs that allows us to differentiate between a NN
as a family of weights and the function implemented by the NN. This implemented function will be
called the realization of the NN. 
\begin{definition}\label{def:NeuralNetworks}
Let $n, L\in \N$.
A \emph{NN $\Phi$ with input dimension $\mathrm{dim}_{\mathrm{in}}\left(\Phi\right)\coloneqq n$ and $L$ layers}
is a sequence of matrix-vector tuples
\[
  \Phi = \big( (\mathbf{A}_1,\mathbf{b}_1), (\mathbf{A}_2,\mathbf{b}_2), \dots, (\mathbf{A}_L, \mathbf{b}_L) \big),
\]
where $N_0 = n$ and $N_1, \dots, N_{L} \in \N$, and where each
$\mathbf{A}_\ell$ is an $N_{\ell} \times N_{\ell-1}$ matrix,
and $\mathbf{b}_\ell \in \R^{N_\ell}$.

If $\Phi$ is a NN as above, $K \subset \R^n$, and if
$\varrho\colon \R \to \R$ is arbitrary, then we define the associated
\emph{realization of $\Phi$ with activation function $\varrho$ over $K$}
(in short, the $\varrho$\emph{-realization of $\Phi$ over $K$})
as the map $\Realization_{\varrho}^K(\Phi)\colon K \to \R^{N_L}$ such that
\[
  \Realization_{\varrho}^K(\Phi)(\mathbf{x}) = \mathbf{x}_L ,
\]
where $\mathbf{x}_L$ results from the following scheme:
\begin{equation*}
  \begin{split}
    \mathbf{x}_0 &\coloneqq \mathbf{x}, \\
    \mathbf{x}_{\ell} &\coloneqq \varrho(\mathbf{A}_{\ell} \, \mathbf{x}_{\ell-1} + \mathbf{b}_\ell),
    \qquad \text{ for } \ell = 1, \dots, L-1,\\
    \mathbf{x}_L &\coloneqq \mathbf{A}_{L} \, \mathbf{x}_{L-1} + \mathbf{b}_{L},
  \end{split}
\end{equation*}
and where $\varrho$ acts componentwise, that is,
$\varrho(\mathbf{v}) = (\varrho(\mathbf{v}_1), \dots, \varrho(\mathbf{v}_m))$
for any $\mathbf{v} = (\mathbf{v}_1, \dots, \mathbf{v}_m) \in \R^m$.

We call $N(\Phi) \coloneqq n + \sum_{j = 1}^L N_j$ the
\emph{number of neurons of the NN} $\Phi$ and $L = L(\Phi)$ the
\emph{number of layers}. For $\ell \leq L$ we call $M_\ell(\Phi) \coloneqq \|\mathbf{A}_\ell\|_0 + \|\mathbf{b}_\ell\|_0$ the \emph{number of weights in the $\ell$-th layer} and we define $M(\Phi)\coloneqq \sum_{\ell = 1}^L M_\ell(\Phi)$, which we call the \emph{number of weights of $\Phi.$}
Moreover, we refer to $\mathrm{dim}_{\mathrm{out}}\left(\Phi\right)\coloneqq N_L $ as 
\emph{the output dimension} of $\Phi$.
\end{definition}

First of all, we note that it is possible to concatenate two NNs in the following way.
\begin{definition}
Let $L_1, L_2 \in \N$ and let $\vphantom{\sum_j}
\Phi^1 = \big( (\mathbf{A}_1^1,\mathbf{b}_1^1), \dots, (\mathbf{A}_{L_1}^1,\mathbf{b}_{L_1}^1) \big),
\Phi^2 = \big((\mathbf{A}_1^2,\mathbf{b}_1^2), \dots, (\mathbf{A}_{L_2}^2,\mathbf{b}_{L_2}^2) \big)$
be two NNs such that the input layer of $\Phi^1$ has the same
dimension as the output layer of $\Phi^2$.
Then, $\Phi^1 \conc \Phi^2$ denotes the following $L_1+L_2-1$ layer
NN:
\[
  \Phi^1 \conc \Phi^2
  \coloneqq \big(
       (\mathbf{A}_1^2,\mathbf{b}_1^2),
       \dots,
       (\mathbf{A}_{L_2-1}^2,\mathbf{b}_{L_2-1}^2),
       (\mathbf{A}_{1}^1 \mathbf{A}_{L_2}^2, \mathbf{A}_{1}^1 \mathbf{b}^2_{L_2} + \mathbf{b}_1^1),
       (\mathbf{A}_{2}^1, \mathbf{b}_{2}^1),
       \dots,
       (\mathbf{A}_{L_1}^1, \mathbf{b}_{{L_1}}^1)
     \big). 
\]
We call $\Phi^1 \conc \Phi^2$ the
\emph{concatenation of $\Phi^1,$ $\Phi^2$}.
\end{definition}

In general, there is no bound on $M(\Phi^1 \conc \Phi^2)$ that is linear in $M(\Phi^1)$ and $M(\Phi^2)$. 
For the remainder of the paper, let $\varrho$ be given by the ReLU activation function, i.e., $\varrho(x)=\max\{x,0\}$ for $x\in \R.$ 
We will see in the following, that we are able to introduce an alternative concatenation which helps us to control the number of non-zero weights. 
Towards this goal, we give the following result which shows that we can construct NNs the ReLU-realization of which is the identity function on $\R^n.$ 

\begin{lemma}\label{lem:ReLUidentity}
For any $L\in \N$ there exists a NN $\Phi^{\identity}_{n,L}$ with input dimension $n,$ output dimension $n$ and at most $2n L$ non-zero, $\{-1,1\}$-valued weights such that
\[
    \Realization^{\R^n}_\varrho\left(\Phi^{\identity}_{n,L} \right)=\mathbf{Id}_{\R^n}.
\]
\end{lemma}

We now introduce the sparse concatenation of two NNs.

\begin{definition}\label{def:sparseconc}
Let $\Phi^1,\Phi^2$ be two NNs such that the output dimension of $\Phi^2$ and the input dimension of $\Phi^1$ equal $n\in \N$. Then the \emph{sparse concatenation of $\Phi^1$ and $\Phi^2$} is defined as 
\[
    \Phi^1\odot \Phi^2 \coloneqq \Phi^1\conc \Phi^{\mathbf{Id}}_{n,1}\conc \Phi^2.
\]
\end{definition}

We will see later in Lemma \ref{lem:size} the properties of the sparse concatenation of NNs. We proceed with the second operation that we can perform with NNs. This operation is called parallelization.

\begin{definition}[\cite{PetV2018OptApproxReLU,SchwabOption}]
Let $\Phi^1,...,\Phi^k$ be NNs which have equal input dimension such that there holds $\Phi^i=\left((\mathbf{A}^i_1,\mathbf{b}^i_1),...,(\mathbf{A}^i_L,\mathbf{b}^i_L) \right)$ for some $L\in \N$. Then, we define the \emph{parallelization of $\Phi^1,...,\Phi^k$} by  
\begin{align*}
    \Paral\left(\Phi^1,...,\Phi^k\right)\coloneqq \left(\left(\left( \begin{array}{cccc}
\mathbf{A}^1_1 &  &  &    \\
 & \mathbf{A}^2_1 &  &    \\
 &  & \ddots &   \\
 &  &  & \mathbf{A}^k_1 
\end{array} \right) , \begin{pmatrix}\mathbf{b}^1_1 \\ \mathbf{b}^2_1 \\ \vdots \\ \mathbf{b}^k_1 \end{pmatrix}\right),...,\left(\left(\begin{array}{cccc}
\mathbf{A}^1_L &  &  &    \\
 & \mathbf{A}^2_L &  &    \\
 &  & \ddots &   \\
 &  &  & \mathbf{A}^k_L 
\end{array} \right) , \begin{pmatrix}\mathbf{b}^1_L \\ \mathbf{b}^2_L \\ \vdots \\ \mathbf{b}^k_L \end{pmatrix} \right) \right). 
\end{align*}
Now, let $\Phi$ be a NN and $L\in \N$ such that $L(\Phi)\leq L.$ Then, define the NN 
\begin{align*}
    E_L(\Phi)\coloneqq \begin{cases} \Phi, \qquad &\text{ if } L(\Phi)=L, \\
    \Phi^{\mathbf{Id}}_{\mathrm{dim}_{\mathrm{out}}(\Phi),L-L(\Phi)}\odot\Phi,\qquad &\text{ if } L(\Phi)<L.\end{cases}
\end{align*}
Finally, let $\tilde{\Phi}^1,...,\tilde{\Phi}^k$ be NNs which have the same input dimension and let 
\[\tilde{L}\coloneqq \max\left\{ L\left(\tilde{\Phi}^1\right),...,L\left(\tilde{\Phi}^k\right)\right\}.\] 
Then, we define 
\begin{align*}
    \Paral\left(\tilde{\Phi}^1,...,\tilde{\Phi}^k \right)\coloneqq \Paral\left(E_{\tilde{L}}\left(\tilde{\Phi}^1 \right),...,E_{\tilde{L}}\left(\tilde{\Phi}^k \right)\right).
\end{align*}
We call $\Paral\left(\tilde{\Phi}^1,...,\tilde{\Phi}^k \right)$ the \emph{parallelization of} $\tilde{\Phi}^1,...,\tilde{\Phi}^k.$
\end{definition}

The following lemma was established in \cite[Lemma 5.4]{SchwabOption} and examines the properties of the sparse concatenation as well as of the parallelization of NNs.

\begin{lemma}[\cite{SchwabOption}] 
\label{lem:size}
Let $\Phi^1,...,\Phi^k$ be NNs. 
\begin{itemize}
    \item[(a)] If the input dimension of $\Phi^1$, which shall be denoted by $n_1$, equals the output dimension of $\Phi^2,$ and $n_2$ is the input dimension of $\Phi^2,$ then
    $$
    \mathrm{R}^{\R^{n_1}}_\varrho\left(\Phi^1\right) \circ \mathrm{R}^{\R^{n_2}}_\varrho\left(\Phi^2\right) = \mathrm{R}^{\R^{n_2}}_\varrho\left(\Phi^1 \odot \Phi^2\right)
    $$
    and
    \begin{enumerate}[(i)]
        \item $L\left(\Phi^1 \odot\Phi^2\right) \leq L\left(\Phi^1\right)+L\left(\Phi^2\right),$
        \item $M\left(\Phi^1 \odot\Phi^2\right)\leq M\left(\Phi^1\right)+ M\left(\Phi^2\right)+ M_1\left(\Phi^1\right)+M_{L\left(\Phi^2\right)}\left(\Phi^2\right)\leq 2M\left(\Phi^1\right)+2M\left(\Phi^2\right),$
        \item $M_1\left(\Phi^1\odot \Phi^2\right) =M_1\left(\Phi^2\right),$
        \item $M_{L\left(\Phi^1\odot \Phi^2\right)}\left(\Phi^1\odot \Phi^2\right) = M_{L\left(\Phi^1\right)}\left(\Phi^1\right).$
    \end{enumerate}
    \item[(b)] If the input dimension of $\Phi^i$, denoted by $n$, equals the input dimension of $\Phi^j,$ for all $i,j$, then for the NN $\Paral\left(\Phi^1,\Phi^2,...,\Phi^k\right)$ and all $\mathbf{x}_1,\dots \mathbf{x}_k\in \R^n$ we have
    $$
    \mathrm{R}^{\R^n}_\varrho\left(\Paral\left(\Phi^1,\Phi^2,...,\Phi^k\right)\right)(\mathbf{x}_1,\dots,\mathbf{x}_k) = \left( \mathrm{R}^{\R^n}_\varrho\left(\Phi^1\right)(\mathbf{x}_1), \mathrm{R}^{\R^n}_\varrho\left(\Phi^2\right)(\mathbf{x}_2) , \dots, \mathrm{R}^{\R^n}_\varrho\left(\Phi^k\right)(\mathbf{x}_k)\right)
    $$
    as well as
    \begin{enumerate}[(i)]
        \item $L\left(\Paral\left(\Phi_1,\Phi^2,...,\Phi^k\right)\right) = \max_{i=1,...,k}L\left(\Phi^i\right),$
        \item $M\left(\Paral\left(\Phi^1,\Phi^2,...,\Phi^k\right)\right)\leq 2\left(\sum_{i=1}^k M\left(\Phi^i\right) \right)+4\left(\sum_{i=1}^k \mathrm{dim}_{\mathrm{out}}\left(\Phi^i\right) \right)\max_{i=1,...,k}L\left(\Phi^i\right),$
         \item $M\left(\Paral\left(\Phi^1,\Phi^2,...,\Phi^k\right)\right) =  \sum_{i=1}^k M\left(\Phi^i\right)$, if $L\left(\Phi^1\right) = L\left(\Phi^2\right)  = \ldots = L\left(\Phi^k\right)$,
        \item $M_1\left(\Paral\left(\Phi^1,\Phi^2,...,\Phi^k\right)\right) =\sum_{i=1}^k M_1\left(\Phi^i\right),$
        \item $M_{L\left(\Paral\left(\Phi^1,\Phi^2,...,\Phi^k\right)\right)}\left(\Paral\left(\Phi^1,\Phi^2,...,\Phi^k\right)\right) \leq  \sum_{i=1}^k \max\left\{2\mathrm{dim}_{\mathrm{out}}\left(\Phi^i\right),M_{L\left(\Phi^i\right)}\left(\Phi^i\right) \right\},$
        \item $M_{L\left(\Paral\left(\Phi^1,\Phi^2,...,\Phi^k\right)\right)}\left(\Paral\left(\Phi^1,\Phi^2,...,\Phi^k\right)\right)=\sum_{i=1}^k M_{L\left(\Phi^i\right)}\left(\Phi^i\right),$ if $L\left(\Phi^1\right) = L\left(\Phi^2\right)  = \ldots = L\left(\Phi^k\right)$.
    \end{enumerate}
\end{itemize}
\end{lemma}

\subsection{A Neural Network Based Approach Towards Matrix Inversion} \label{subsec:matrInv}

The goal of this subsection is to emulate the inversion of square matrices by realizations of NNs which are comparatively small in size. In particular, Theorem \ref{thm:Inverse} shows that, for $d\in \N,~\epsilon\in (0,1/4),$ and $\delta \in (0,1),$ we are able to efficiently construct NNs $\Phi^{1-\delta,d}_{\mathrm{inv};\epsilon}$ the ReLU-realization of which approximates the map 
\[
\left\{\mathbf{A}\in \R^{d\times d}\colon\|\mathbf{A}\|_2\leq 1-\delta \right\}\to \R^{d\times d},~\mathbf{A}\mapsto \left(\mathbf{Id}_{\R^d}-\mathbf{A} \right)^{-1}=\sum_{k=0}^\infty \mathbf{A}^k
\]
up to an $\|\cdot \|_2$- error of $\epsilon$.

To stay in the classical NN setting, we employ vectorized matrices in the remainder of this paper. Let $\mathbf{A}\in \R^{d\times l}.$ We write
$$
    \vect(\mathbf{A})\coloneqq (\mathbf{A}_{1,1},...,\mathbf{A}_{d,1},...,\mathbf{A}_{1,l},...,\mathbf{A}_{d,l})^T\in \R^{dl}.
$$  
Moreover, for a vector $\mathbf{v}=(\mathbf{v}_{1,1},...,\mathbf{v}_{d,1},...,\mathbf{v}_{1,d},...,\mathbf{v}_{d,l})^T\in \R^{d l }$ we set
$$
       \mathbf{matr}(\mathbf{v}) \coloneqq (\mathbf{v}_{i,j})_{i=1,...,d,~j=1,...,l}\in \R^{d\times l}.
$$
In addition, for $d,n,l\in \N$ and $Z>0$ we set 
$$
K^Z_{d,n,l}\coloneqq \left\{(\vect(\mathbf{A}),\vect(\mathbf{B}))\colon(\mathbf{A},\mathbf{B})\in \R^{d\times n}\times \R^{n\times l} ,~\|\mathbf{A}\|_2,\|\mathbf{B}\|_2\leq Z\right\}
$$
as well as
$$
K^Z_{d}\coloneqq \left\{\vect(\mathbf{A})\colon\mathbf{A}\in \R^{d\times d},~\|\mathbf{A}\|_2\leq Z\right\}.
$$

The basic ingredient for the construction of NNs emulating a matrix inversion is the following result about NNs emulating the multiplication of two matrices.

\begin{proposition}\label{prop:Multiplikation}
Let $d,n,l\in \N,~\epsilon\in (0,1),~Z> 0.$ There exists a NN $\Phi^{Z,d,n,l}_{\mathrm{mult};\epsilon}$ with $n\cdot(d+l) $- dimensional input, $dl $-dimensional output such that, for an absolute constant $C_{\mathrm{mult}}>0,$ the following properties are fulfilled:

\begin{enumerate}[(i)]
    \item $L\left(\mult\right)\leq C_{\mathrm{mult}}\cdot \left(\log_2\left(1/\epsilon\right) +\log_2\left(n\sqrt{dl}\right)+\log_2\left(\max\left\{1,Z\right\}\right)\right)$,
    \item $M\left(\mult\right) \leq C_{\mathrm{mult}}\cdot \left(\log_2\left(1/\epsilon\right) +\log_2\left(n\sqrt{dl}\right)+\log_2\left(\max\left\{1,Z\right\}\right)\right) dnl, $
    \item $M_1\left(\mult\right) \leq C_{\mathrm{mult}} dnl, \qquad \text{ as well as } \qquad M_{L\left(\mult\right)}\left(\mult\right)  \leq C_{\mathrm{mult}} dnl$,
    \item $\sup_{(\vect(\mathbf{A}),\vect(\mathbf{B}))\in K^Z_{d,n,l}}\left\|\mathbf{A} \mathbf{B}- \mathbf{matr}\left(\Realization_{\varrho}^{K^Z_{d,n,l}}\left(\Phi^{Z,d,n,l}_{\mathrm{mult};\epsilon}\right)(\vect(\mathbf{A}),\vect(\mathbf{B}))\right)\right\|_2 \leq \epsilon$,
    \item for any $(\vect(\mathbf{A}),\vect(\mathbf{B}))\in  K^Z_{d,n,l}$ we have 
    \[
    \left\| \mathbf{matr}\left(\mathrm{R}_{\varrho}^{K^Z_{d,n,l}}\left(\mult\right) (\vect(\mathbf{A}),\vect(\mathbf{B}))\right)\right\|_2\leq \epsilon+\|\mathbf{A}\|_2 \|\mathbf{B}\|_2\leq \epsilon+Z^2 \leq 1+Z^2.
    \]
\end{enumerate}
\end{proposition}

Based on Proposition \ref{prop:Multiplikation}, we construct in Appendix \ref{app:InverseProof} NNs emulating the map $\mathbf{A}\mapsto \mathbf{A}^k$ for square matrices $\mathbf{A}$ and $k\in \N.$ This construction is then used to prove the following result.

\begin{theorem} \label{thm:Inverse}
 For $\epsilon,\delta\in (0,1)$ define
\[
m(\epsilon,\delta)\coloneqq \left\lceil \frac{\log_2\left( 0.5 \epsilon\delta \right)}{\log_{2}(1-\delta)}\right\rceil.
\]
There exists a universal constant $C_{\mathrm{inv}} > 0$ such that for every $d\in \N,$ $\epsilon\in (0,1/4)$ and every $\delta\in (0,1)$ there exists  a NN $\Phi^{1-\delta, d}_{\mathrm{inv};\epsilon}$ with $d^2$-dimensional input, $d^2$-dimensional output and the following properties:
\begin{enumerate}[(i)]
    \item $L\left(\Phi^{1-\delta, d}_{\mathrm{inv};\epsilon}\right)\leq C_{\mathrm{inv}} \log_2\left(m(\epsilon,\delta) \right)\cdot  \left(\log_2\left(1/ \epsilon\right)+\log_2\left(m(\epsilon,\delta) \right)+\log_2(d)\right)$,
    \item $M\left(\Phi^{1-\delta, d}_{\mathrm{inv};\epsilon}\right)\leq C_{\mathrm{inv}} m(\epsilon,\delta) \log_2^2(m(\epsilon,\delta))d^3\cdot  \left(\log_2\left(1/\epsilon\right)+ \log_2\left(m(\epsilon,\delta)\right) +\log_2(d) \right)$,
    
    \item $\sup_{\vect(\mathbf{A})\in K^{1-\delta}_{d}} \left\|\left(\mathbf{Id}_{\R^{d}}-\mathbf{A} \right)^{-1} - \mathbf{matr}\left( \Realization_\varrho^{K^{1-\delta}_{d}}\left(\Phi^{1-\delta,d}_{\mathrm{inv};\epsilon}\right)(\vect(\mathbf{A}))\right)   \right\|_2 \leq \epsilon,$
    \item for any $\vect(\mathbf{A})\in K^{1-\delta}_d$ we have 
    \begin{align*}
        \left\| \mathbf{matr}\left( \Realization_\varrho^{K^{1-\delta}_{d}}\left(\Phi^{1-\delta,d}_{\mathrm{inv};\epsilon}\right)(\vect(\mathbf{A}))\right)   \right\|_2\leq \epsilon + \left\|\left(\mathbf{Id}_{\R^d}-\mathbf{A}\right)^{-1} \right\|_2 \leq \epsilon+ \frac{1}{1-\|\mathbf{A}\|_2} \leq \epsilon + \frac{1}{\delta}.
    \end{align*}
\end{enumerate}
\end{theorem}

\begin{remark}
In the proof of Theorem \ref{thm:Inverse}, we approximate the function mapping a matrix to its inverse via the Neumann series and then emulate this construction by NNs. There certainly exist alternative approaches to approximating this inversion function, such as, for example, via Chebyshev matrix polynomials (for an introduction of Chebyshev polynomials, see for instance \cite[Chapter 8.2]{Sullivan2015}). In fact, approximation by Chebyshev matrix polynomials is more efficient in terms of the degree of the polynomials required to reach a certain approximation accuracy. However, emulation of Chebyshev matrix polynomials by NNs either requires larger networks than that of monomials or, if they are represented in a monomial basis, coefficients that grow exponentially with the polynomial degree. In the end, the advantage of a smaller degree in the approximation through Chebyshev matrix polynomials does not seem to set off the drawbacks described before.
\end{remark}

\section{Neural Networks and Solutions of PDEs Using Reduced Bases}\label{sec:NNPDE}
In this section, we invoke the estimates for the approximate matrix inversion from Section \ref{subsec:matrInv} to approximate the parameter-dependent solution of parametric PDEs by NNs. In other words, for $\epsilontilde \geq \epsilonhat$, we construct NNs approximating the maps

\begin{align*}
    \Ycal\to \R^{D}\colon \quad y\mapsto \tilde{\mathbf{u}}_{y,\epsilontilde}^{\mathrm{h}}, \text{ and }   \Ycal\to \R^{d(\epsilontilde)}\colon \quad y\mapsto \mathbf{u}_{y,\epsilontilde}^{\mathrm{rb}}.
\end{align*}

Here, the sizes of the NNs essentially only depend on the approximation fidelity $\epsilontilde$ and the size $d(\epsilontilde)$ of an appropriate RB, but are independent or at most linear in the dimension of the high-fidelity discretization $D$.

We start in Section \ref{subsec:CoeffNNs} by constructing, under some general assumptions on the parametric problem, a NN emulating the maps $\tilde{\mathbf{u}}_{\cdot,\epsilontilde}^{\mathrm{h}}$ and $\mathbf{u}_{\cdot,\epsilontilde}^{\mathrm{rb}}$. In Section \ref{sec:Examples}, we verify these assumptions on two examples. 

\subsection{Determining the Coefficients of the Solution}\label{subsec:CoeffNNs}

Next, we present constructions of NNs the ReLU-realizations of which approximate the maps $\tilde{\mathbf{u}}_{\cdot,\epsilontilde}^{\mathrm{h}}$ and $\mathbf{u}_{\cdot,\epsilontilde}^{\mathrm{rb}}$, respectively. In our main result of this subsection, the approximation error of the NN approximation $\tilde{\mathbf{u}}_{\cdot,\epsilontilde}^{\mathrm{h}}$ will be measured with respect to the $|\cdot|_{\mathbf{G}}$-norm since we can relate this norm directly to the norm on $\Hil$ via Equation \eqref{eq:NormCoefficientHilbert}. In contrast, the approximation error of the NN approximating $\mathbf{u}^{\mathrm{rb}}_{\cdot,\epsilontilde}$ will be measured with respect to the $|\cdot|$-norm due to Equation \ref{eq:NormRB}.

As already indicated earlier, the main ingredient of the following arguments is an application of the NN of Theorem \ref{thm:Inverse} to the matrix $\mathbf{B}_{y,\epsilontilde}^{\mathrm{rb}}$. As a preparation, we show in Proposition \ref{prop:IndependentAlpha} in the appendix, that we can rescale the matrix $\mathbf{B}_{y,\epsilontilde}^{\mathrm{rb}}$ with a constant factor $\alpha \coloneqq (C_{\mathrm{coer}}+C_{\mathrm{cont}})^{-1}$ (in particular, independent of $y$ and $d(\epsilontilde)$) so that with $C_{\mathrm{coer}}\delta \coloneqq \alpha C_{\mathrm{coer}}$   
$$
\left\|\mathbf{Id}_{\R^{d(\epsilontilde)}}-\alpha \mathbf{B}_{y,\epsilontilde}^{\mathrm{rb}} \right\|_2 \leq 1-\delta <1.
$$

We will fix these values of $\alpha$ and $\delta$ for the remainder of the manuscript. Next, we state two abstract assumptions on the approximability of the map $ \mathbf{B}_{\cdot,\epsilontilde}^{\mathrm{rb}}$ which we will later on specify when we consider concrete examples in Subsection \ref{sec:Examples}.

\begin{assumption}\label{ass:ParameterStiffness}
We assume that, for any $\epsilontilde\geq \epsilonhat,\epsilon>0,$ and for a corresponding RB $(\psi_{i})_{i=1}^{d(\epsilontilde)},$ there exists a NN $\Phi^{\mathbf{B}}_{\epsilontilde,\epsilon}$ with $p$-dimensional input and $d(\epsilontilde)^2$-dimensional output
 such that 
\begin{align*}
    \sup_{y\in \Ycal} \left\|\alpha  \mathbf{B}^{\mathrm{rb}}_{y,\epsilontilde}- \mathbf{matr}\left( \Realization^{\Ycal}_{\varrho}\left(\Phi^{\mathbf{B}}_{\epsilontilde,\epsilon}\right)(y)\right)\right\|_2\leq \epsilon.
\end{align*}
We set $B_{M}\left(\epsilontilde, \epsilon\right) \coloneqq M\left(\Phi^{\mathbf{B}}_{\epsilontilde,\epsilon}\right) \in \N$ and $B_{L}\left(\epsilontilde, \epsilon\right) \coloneqq L\left(\Phi^{\mathbf{B}}_{\epsilontilde,\epsilon}\right)\in \N.$
\end{assumption}

In addition to Assumption \ref{ass:ParameterStiffness}, we state the following assumption on the approximability of the map $ \mathbf{f}_{\cdot,\epsilontilde}^{\mathrm{rb}}$.
\begin{assumption}\label{ass:ParameterRHS}
We assume that for every $\epsilontilde\geq \epsilonhat,\epsilon>0,$ and a corresponding RB $(\psi_{i})_{i=1}^{d(\epsilontilde)}$ there exists a NN $\Phi^{\mathbf{f}}_{\epsilontilde,\epsilon}$ with $p$-dimensional input and $d(\epsilontilde)$-dimensional output such that 
\begin{align*}
    \sup_{y\in \Ycal} \left| \mathbf{f}^{\mathrm{rb}}_{y,\epsilontilde}-  \Realization^{\Ycal}_{\varrho}\left(\Phi^{\mathbf{f}}_{\epsilontilde,\epsilon}\right)(y)\right|\leq \epsilon.
\end{align*}
We set $ F_L\left(\epsilontilde, \epsilon\right) \coloneqq L(\Phi^{\mathbf{f}}_{\epsilontilde,\epsilon})$ and $F_M\left(\epsilontilde, \epsilon\right) \coloneqq M\left(\Phi^{\mathbf{f}}_{\epsilontilde,\epsilon}\right)$.
\end{assumption}

Now we are in a position to construct NNs the ReLU-realizations of which approximate the coefficient maps $\tilde{\mathbf{u}}_{\cdot,\epsilontilde}^{\mathrm{h}},\mathbf{u}_{\cdot,\epsilontilde}^{\mathrm{rb}}.$ 

\begin{theorem}\label{thm:NNCoefficientApproximation}
 Let $\epsilontilde\geq \epsilonhat$ and $\epsilon \in \left(0,\alpha/4 \cdot \min\{1,C_{\mathrm{coer}}\} \right).$ Moreover, define $\epsilon'\coloneqq \epsilon/\max\{6,C_{\mathrm{rhs}}\}$, $\epsilon''\coloneqq \epsilon/3 \cdot C_{\mathrm{coer}}$, $\epsilon'''\coloneqq 3/8\cdot \epsilon'\alpha C_{\mathrm{coer}}^2$ and $\kappa \coloneqq 2 \max\left\{1,C_{\mathrm{rhs}},1/C_{\mathrm{coer}}\right\}$.
 Additionally, assume that Assumption \ref{ass:ParameterStiffness} and Assumption \ref{ass:ParameterRHS} hold. Then there exist NNs $\Phi^{\mathbf{u},\mathrm{rb}}_{\epsilontilde,\epsilon}$ and $\Phi^{\mathbf{u},\mathrm{h}}_{\epsilontilde,\epsilon}$ such that the following properties hold:
\begin{enumerate}[(i)]
\item $\sup_{y\in \Ycal}\left|\mathbf{u}^{\mathrm{rb}}_{y,\epsilontilde} - \Realization_\varrho^\Ycal\left( \Phi^{\mathbf{u},\mathrm{rb}}_{\epsilontilde,\epsilon}\right)(y) \right|\leq \epsilon$, \text{ and }
$
     \sup_{y\in \Ycal}\left|\tilde{\mathbf{u}}^{\mathrm{h}}_{y,\epsilontilde} - \Realization_\varrho^\Ycal\left( \Phi^{\mathbf{u},\mathrm{h}}_{\epsilontilde,\epsilon}\right)(y) \right|_{\mathbf{G}}\leq \epsilon,$
\item there exists a constant $C^{\mathbf{u}}_L= C^{\mathbf{u}}_L(C_{\mathrm{coer}}, C_{\mathrm{cont}}, C_{\mathrm{rhs}})>0$ such that
\begin{align*} 
L\left(\Phi^{\mathbf{u}, \mathrm{rb}}_{\epsilontilde,\epsilon}\right)&\leq L\left(\Phi^{\mathbf{u},\mathrm{h}}_{\epsilontilde,\epsilon}\right) \\ &\leq C^{\mathbf{u}}_L \max\left\{ \log_2(\log_2(1/\epsilon))\left(\log_2(1/\epsilon)+ \log_2(\log_2(1/\epsilon))+\log_2(d(\epsilontilde))\right) + B_L(\epsilontilde, \epsilon'''), F_L\left(\epsilontilde,\epsilon''\right)\right\},
\end{align*}
\item 
there exists a constant $C^{\mathbf{u}}_M= C^{\mathbf{u}}_M(C_{\mathrm{coer}}, C_{\mathrm{cont}}, C_{\mathrm{rhs}})>0$ such that
\begin{align*}
    M\left(\Phi^{\mathbf{u}, \mathrm{rb}}_{\epsilontilde,\epsilon}\right) & \leq C^{\mathbf{u}}_M d(\epsilontilde)^2 \cdot \bigg(d(\epsilontilde)\log_2(1/\epsilon)  \log_2^2(\log_2(1/\epsilon))\big(\log_2(1/\epsilon)+\log_2(\log_2(1/\epsilon)) + \log_2(d(\epsilontilde))\big)... \\  &\qquad \qquad \qquad \qquad \qquad \qquad  ...+ B_L(\epsilontilde,\epsilon''') + F_L\left(\epsilontilde,\epsilon''\right)\bigg)  +2B_M(\epsilontilde,\epsilon''') +F_M\left(\epsilontilde,\epsilon''\right),
\end{align*}

\item $
M\left(\Phi^{\mathbf{u},\mathrm{h}}_{\epsilontilde,\epsilon}\right) \leq 2 D d(\epsilontilde) + 2M\left(\Phi^{\mathbf{u}, \mathrm{rb}}_{\epsilontilde,\epsilon}\right),
$
\item $\sup_{y\in \Ycal}\left|\Realization_\varrho^\Ycal\left( \Phi^{\mathbf{u},\mathrm{rb}}_{\epsilontilde,\epsilon}\right)(y) \right|\leq \kappa^2 + \frac{\epsilon}{3}$, \text{ and }
$
     \sup_{y\in \Ycal}\left|\Realization_\varrho^\Ycal\left( \Phi^{\mathbf{u},\mathrm{h}}_{\epsilontilde,\epsilon}\right)(y) \right|_{\mathbf{G}}\leq \kappa^2 + \frac{\epsilon}{3}.$
\end{enumerate}
\end{theorem}

\begin{remark}
In the proof of Theorem \ref{thm:NNCoefficientApproximation} we construct a NN $\Phi^{\mathbf{B}}_{\mathrm{inv};\epsilontilde,\epsilon}$ the ReLU realization of which $\epsilon$-approximates $\left(\mathbf{B}_{y}^{\mathrm{rb}}\right)^{-1}$ (see Proposition \ref{prop:InverseStiffnessNN}). Then the NNs of Theorem \ref{thm:NNCoefficientApproximation} can be explicitely constructed as \begin{align*}
     \Phi^{\mathbf{u},\mathrm{rb}}_{\epsilontilde,\epsilon} \coloneqq \Phi^{\kappa,d(\epsilontilde),d(\epsilontilde),1}_{\mathrm{mult};\frac{\epsilon}{3}} \odot \Paral\left(\Phi^{\mathbf{B}}_{\mathrm{inv};\epsilontilde,\epsilon'},\Phi^{\mathbf{f}}_{\epsilontilde,\epsilon''} \right)\conc \left(\left(  \begin{pmatrix} \mathbf{Id}_{\R^{p}} \\ \mathbf{Id}_{\R^{p}}  \end{pmatrix} , \mathbf{0}_{\R^{2p}}\right)\right)\quad \text{and} \quad 
     \Phi^{\mathbf{u},\mathrm{h}}_{\epsilontilde,\epsilon} \coloneqq \left(\left(\mathbf{V}_\epsilontilde, \mathbf{0}_{\R^D}\right)\right) \odot\Phi^{\mathbf{u},\mathrm{rb}}_{\epsilontilde,\epsilon},
 \end{align*}
\end{remark}

\begin{remark}
It can be checked in the proof of Theorem \ref{thm:NNCoefficientApproximation}, specifically \eqref{eq:CLDef} and \eqref{eq:CMDef} that the constants $C^{\mathbf{u}}_L, C^{\mathbf{u}}_M$ depend on the constants $C_{\mathrm{coer}},C_{\mathrm{cont}},C_{\mathrm{rhs}}$ in the following way (recall that $ \frac{C_{\mathrm{coer}}}{2C_{\mathrm{cont}}} \leq \delta = \frac{C_{\mathrm{coer}}}{C_{\mathrm{coer}}+C_\mathrm{cont}}\leq \frac{1}{2}$):

\begin{itemize}
    \item $C^{\mathbf{u}}_L$ depends affine linearly on 
    
    $$
    \log^2_2\left(\frac{\log_2(\delta/2)}{\log_2(1-\delta)}\right), \quad   \log_2\left(\frac{1}{C_{\mathrm{coer}}+C_{\mathrm{cont}}}\right),\quad  \log_2\left(\max\left\{1,C_{\mathrm{rhs}}, 1/C_{\mathrm{coer}}\right\}\right).$$
    \item $C^{\mathbf{u}}_M$ depends affine linearly on 
    $$
     \log_2\left(\frac{1}{C_{\mathrm{coer}}+C_{\mathrm{cont}}}\right), \ \frac{\log_2(\delta/2)}{\log_2(1-\delta)}\cdot \log_2^3\left(\frac{\log_2(\delta/2)}{\log_2(1-\delta)} \right), \  \log_2\left(\max\left\{1,C_{\mathrm{rhs}}, 1/C_{\mathrm{coer}}\right\}\right).
    $$
\end{itemize}
\end{remark}

\begin{remark}
Theorem \ref{thm:NNCoefficientApproximation} guarantees the existence of two moderately sized NNs the realizations of which approximate the discretized solution maps:
\begin{align}\label{eq:approximatedfunctions}
    \Ycal\to \R^{D}\colon \quad y\mapsto \tilde{\mathbf{u}}_{y,\epsilontilde}^{\mathrm{h}}, \text{ and } \Ycal\to \R^{d(\epsilontilde)}\colon \quad y\mapsto \mathbf{u}_{y,\epsilontilde}^{\mathrm{rb}}.
\end{align}
Also of interest is the approximation of the parametrized solution of the PDE, i.e., the map $\Ycal \times \Omega \to \R \colon \quad  (y,x)\mapsto u_y(x)$, where $\Omega$ is the domain on which the PDE is defined.
Note that, if either the elements of the reduced basis or the elements of the high-fidelity basis can be very efficiently approximated by realizations of NNs, then the representation 
$$
    u_y(x) \approx \sum_{i= 1}^{d(\epsilontilde)} (\mathbf{u}_{y,\epsilontilde}^{\mathrm{rb}})_i \psi_i(x) = \sum_{i=1}^D (\tilde{\mathbf{u}}_{y,\epsilontilde}^{\mathrm{h}})_i \varphi_i(x)
$$
suggests that $(y,x)\mapsto u_y(x)$ can be approximated with essentially the cost of approximating the respective function in \eqref{eq:approximatedfunctions}. 
Many basis elements that are commonly used for the high-fidelity representation can indeed be approximated very efficiently by realizations of NNs, such as, e.g., polynomials, finite elements, or wavelets \cite{YAROTSKY2017103, he2018relu, OPS19_811, shaham2018provable}. 
\end{remark}

\subsection{Examples of Neural Network Approximation of Parametric Maps}\label{sec:Examples}

In this subsection, we apply Theorem \ref{thm:NNCoefficientApproximation} to a variety of concrete examples in which the approximation of the coefficient maps $\mathbf{u}_{\cdot,\epsilontilde}^{\mathrm{rb}},$ $\tilde{\mathbf{u}}_{\cdot,\epsilontilde}^{\mathrm{h}}$ can be approximated by comparatively small NNs. We show that the sizes of these NNs depend only on the size of associated reduced bases by verifying Assumption \ref{ass:ParameterStiffness} and Assumption \ref{ass:ParameterRHS}, respectively. We will discuss to what extent our results depend on the respective ambient dimensions $D,~p$ in Section \ref{sec:CurseDiscussion}.

We will state the following examples already in their variational formulation and note that they fulfill the requirements of Assumption \ref{ass:opeq}. We also remark that the presented examples represent only a small portion of problems to which our theory is applicable. 

\subsubsection{Example I: Diffusion Equation} \label{subsec:ExDiff}
We consider a special case of \cite[Chapter 2.3.1]{QuarteroniIntro} which can be interpreted as a generalized version of the heavily used example $-\mathrm{div}(a\nabla u)=f,$ where $a$ is a scalar field (see for instance \cite{nadav-shashua1,SchwabZech} and the references therein). 
Let $n\in \N,~\Omega\subset \R^n,$ be a Lipschitz domain and $\Hil\coloneqq H_0^1(\Omega)=\left\{u\in H^1(\Omega)\colon~u|_{\partial \Omega}=0\right\}$. 
We assume that the parameter set is given by a compact set $\mathcal{T}\subset L^{\infty}(\Omega,\R^{n\times n})$ such that for all $\mathbf{T}\in \mathcal{T}$ and almost all $\mathbf{x}\in \Omega$ the matrix $\mathbf{T}(\mathbf{x})$ is symmetric, positive definite with matrix norm that can be bounded from above and below independently of $\mathbf{T}$ and $\mathbf{x}$. 
As we have noted in Assumption \ref{ass:opeq}, we can assume that there exist some $(\mathbf{T}_i)_{i=0}^{\infty}\subset L^{\infty}(\Omega,\R^{n\times n})$ such that for every $\mathbf{T}\in \mathcal{T}$ there exist $(y_i(\mathbf{T}))_{i=1}^\infty\subset [-1,1]$ with $\mathbf{T}= \mathbf{T}_0+\sum_{i=1}^\infty y_i(\mathbf{T}) \mathbf{T}_i.$ 
We restrict ourselves to the case of finitely supported sequences $(y_i)_{i=1}^\infty.$ To be more precise, let $p\in \N$ be potentially very high but fixed, let $\Ycal\coloneqq [-1,1]^{p}$ and consider for $y\in \Ycal$ and some fixed $f\in \Hil^*$  the parametric PDE
\begin{align*}
        b_y(u_y,v)\coloneqq \int_{\Omega} \mathbf{T}_0\nabla u_y \nabla v~d\mathbf{x}+  \sum_{i=1}^{p} y_i \int_{\Omega} \mathbf{T}_i\nabla u_y \nabla v~d\mathbf{x} = f(v), \qquad \text{for all } v\in \Hil.
    \end{align*} 
    
Then, the parameter-dependency of the bilinear forms is linear, hence analytic whereas the parameter-dependency of the right-hand side is constant, hence also analytic, implying that $\overline{W}(S(\Ycal))$ decays exponentially fast. This in turn implies existence of small RBs $(\psi_i)_{i=1}^{d(\epsilontilde)}$ where $d(\epsilontilde)$ depends at most polylogarithmically on $1/\epsilontilde$. In this case, Assumption \ref{ass:ParameterStiffness} and Assumption \ref{ass:ParameterRHS} are trivially fulfilled: for $\epsilontilde>0,\epsilon>0$ we can construct one-layer NNs $\Phi^{\mathbf{B}}_{\epsilontilde,\epsilon}$ with $p$-dimensional input and $d(\epsilontilde)^2$-dimensional output as well as  $\Phi^{\mathbf{f}}_{\epsilontilde,\epsilon}$ with $p$-dimensional input and $d(\epsilontilde)$-dimensional output the ReLU-realizations of which exactly implement the maps $y\mapsto \mathbf{B}^{\mathrm{rb}}_{y,\epsilontilde}$ and $y\mapsto \mathbf{f}^{\mathrm{rb}}_{y,\epsilontilde},$ respectively.  

In conclusion, in this example, we have, for $\epsilontilde, \epsilon>0$,
\begin{align*}
B_L\left(\epsilontilde, \epsilon\right) = 1, \quad   
F_L\left(\epsilontilde, \epsilon\right) = 1, \quad   
B_M\left(\epsilontilde, \epsilon\right) \leq pd(\epsilontilde)^2, \quad    F_M\left(\epsilontilde, \epsilon\right) \leq pd(\epsilontilde).
\end{align*}
Theorem \ref{thm:NNCoefficientApproximation} hence implies the existence of a NN approximating $\mathbf{u}^{\mathrm{rb}}_{\cdot,\epsilontilde}$ up to error $\epsilon$ with a size that is linear in $p,$ polylogarithmic in $1/\epsilon$, and, up to a log factor, cubic in $d(\epsilontilde)$. Moreover, we have shown the existence of a NN approximating $\tilde{\mathbf{u}}^{\mathrm{h}}_{\cdot,\epsilontilde}$ with a size that is linear in $p,$ polylogarithmic in $1/\epsilon,$ linear in $D$ and, up to a log factor, cubic in $d(\epsilontilde).$ 

\subsubsection{Example II: Linear Elasticity Equation}
Let $\Omega\subset \R^3$ be a Lipschitz domain,~ $\Gamma_D,\Gamma_{N_1},\Gamma_{N_2},\Gamma_{N_3}\subset \partial\Omega,$ be disjoint such that $\Gamma_D\cup \Gamma_{N_1}\cup\Gamma_{N_2}\cup\Gamma_{N_3} = \partial\Omega,$ $\Hil\coloneqq [H_{\Gamma_D}^1(\Omega)]^3,$ where $H_{\Gamma_D}^1(\Omega)= \left\{u\in H^1(\Omega)\colon u|_{\Gamma_D}=0 \right\}.$ In variational formulation, this problem can be formulated as an affinely decomposed problem dependent on five parameters, i.e., $p=5.$ Let $\Ycal\coloneqq [\tilde{y}^{1,1},\tilde{y}^{2,1}]\times...\times [\tilde{y}^{1,5},\tilde{y}^{2,5}]\subset \R^5$ such that $[\tilde{y}^{1,2},\tilde{y}^{2,2}]\subset (-1, 1/2 )$ and for $y=(y_1,...,y_5)\in \Ycal$ we consider the problem

\begin{align*}
    b_y(u_y,v)=f_y(v), \quad \text{ for all } v\in \Hil, 
\end{align*}
where 
\begin{itemize}
    \item 
        $b_y(u_y,v) 
        \coloneqq \frac{y_1}{1+y_2} \int_{\Omega} \mathrm{trace}\left(\left(\nabla u_y + (\nabla(u_y)^T \right)\cdot \left(\nabla v + (\nabla v)^T \right)^T\right) d\mathbf{x} + \frac{y_1y_2}{1-2y_2} \int_{\Omega}  \mathrm{div}(u_y)\, \mathrm{div}(v)~d\mathbf{x}, $
    \item 
        $f_y(v)
        \coloneqq y_3\int_{\Gamma_1} \mathbf{n}\cdot v\, d\mathbf{x}+y_4\int_{\Gamma_2} \mathbf{n}\cdot v \, d\mathbf{x} +y_5\int_{\Gamma_3} \mathbf{n}\cdot v \, d\mathbf{x},$
    and where $\mathbf{n}$ denotes the outward unit normal on $\partial \Omega.$ 
\end{itemize}

The parameter-dependency of the right-hand side is linear (hence analytic), whereas the parameter-dependency of the bilinear forms is rational, hence (due to the choice of $\tilde{y}^{1,2},\tilde{y}^{2,2}$) also analytic and $\overline{W}_N(S(\Ycal))$ decays exponentially fast implying that we can choose $d(\epsilontilde)$ to depend polylogarithmically on $\epsilontilde$. It is now easy to see that Assumption \ref{ass:ParameterStiffness} and Assumption \ref{ass:ParameterRHS} are fulfilled with NNs the size of which is comparatively small: By \cite{TelgarskyRational}, for every $\epsilontilde,\epsilon>0$ we can find a NN with $\mathcal{O}(\log^2_2(1/\epsilon))$ layers and $\mathcal{O}(d(\epsilontilde)^2\log^3_2(1/\epsilon))$ non-zero weights the ReLU-realization of which approximates the map $y\mapsto \mathbf{B}_{y,\epsilontilde}^{\mathrm{rb}}$ up to an error of $\epsilon$. Moreover, there exists a one-layer NN $\Phi^{\mathbf{f}}_{\epsilontilde,\epsilon}$ with $p$-dimensional input and $d(\epsilontilde)$-dimensional output the ReLU-realization of which exactly implements the map $y\mapsto \mathbf{f}^{\mathrm{rb}}_{y,\epsilontilde}$. 
In other words, in these examples, for $\epsilontilde, \epsilon>0$,
\begin{align*}
B_L(\epsilontilde, \epsilon)  \in \mathcal{O}\left(\log_2^2(1/\epsilon) \right), \quad 
 F_L\left(\epsilontilde, \epsilon\right) = 1, \quad   
 B_M\left(\epsilontilde, \epsilon\right) \in \mathcal{O}\left(d(\epsilontilde)^2\log_2^3(1/\epsilon) \right), \quad    F_M\left(\epsilontilde, \epsilon\right) \leq 5 d(\epsilontilde).  
\end{align*}
Thus, Theorem \ref{thm:NNCoefficientApproximation} implies the existence of NNs approximating $\mathbf{u}^{\mathrm{rb}}_{\cdot,\epsilontilde}$ up to error $\epsilon$ with a size that is polylogarithmic in  $1/\epsilon$, and, up to a log factor, cubic in $d(\epsilontilde)$. Moreover, there exist NNs approximating $\tilde{\mathbf{u}}^{\mathrm{h}}_{\cdot,\epsilontilde}$ up to error $\epsilon$ with a size that is linear in $D,$ polylogarithmic in  $1/\epsilon$, and, up to a log factor, cubic in $d(\epsilontilde)$. 

For a more thorough discussion of this example (a special case of the linear elasticity equation which describes the displacement of some elastic structure under physical stress on its boundaries), we refer to \cite[Chapter 2.1.2, Chapter 2.3.2, Chapter 8.6]{QuarteroniIntro}.

\section{Discussion: Dependence of Approximation Rates on  Involved Dimensions} \label{sec:CurseDiscussion}

In this section, we will discuss our results in terms of the dependence on the involved dimensions. We would like to stress that the resulting approximation rates (which can be derived from Theorem \ref{thm:NNCoefficientApproximation}) differ significantly from and are often substantially better than alternative approaches. As described in Section \ref{sec:RedBase}, there are three central dimensions that describe the hardness of the problem. These are the dimension $D$ of the high-fidelity discretization space $U^{\mathrm{h}}$, the dimension $d(\epsilontilde)$ of the reduced basis space, and the dimension $p$ of the parameter space $\mathcal{Y}.$ 

\paragraph*{Dependence on $D$:}
Examples I and II above establish approximation rates that depend at most linearly on $D$, in particular, the dependence on $D$ is not coupled to the dependence on $\epsilon$. Another approach to solve these problems would be to directly solve the linear systems from the high-fidelity discretization. Without further assumptions on sparsity properties of the matrices, the resulting complexity would be $\mathcal{O}(D^3)$ plus the cost of assembling the high-fidelity stiffness matrices. Since $D \gg d(\epsilontilde)$, this is significantly worse than the approximation rate provided by Theorem \ref{thm:NNCoefficientApproximation}.

\paragraph*{Dependence on $d(\epsilontilde):$} If one assembles and solves the Galerkin scheme for a previously found reduced basis, one typically needs $\mathcal{O}\left(d(\epsilontilde)^3\right)$ operations. By building NNs emulating this method, we achieve  essentially the same approximation rate of $\mathcal{O}\left(d(\epsilontilde)^3\log_2(d(\epsilontilde))\cdot C(\epsilon)\right)$ where $C(\epsilon)$ depends polylogarithmically on the approximation accuracy $\epsilon.$ 

Note that, while having comparable complexity, the NN-based approach is more versatile than using a Galerkin scheme and can be applied even when the underlying PPDE is fully unknown as long as sufficiently many snapshots are available. 


\paragraph*{Dependence on $p$:} 

We start by comparing our result to naive NN approximation results which are simply based on the smoothness properties of the map $y \mapsto \tilde{\mathbf{u}}_{\cdot,\epsilontilde}^{\mathrm{h}}$ without using its specific structure. For example, if these maps are analytic, then classical approximation rates with NNs (such as those provided by \cite[Theorem 1]{YAROTSKY2017103}, \cite[Theorem 3.1]{PetV2018OptApproxReLU} or \cite[Corollary 4.2]{ReLUSobolev}) promise approximations up to an error of $\epsilon$ with NNs $\Phi$ of size $M(\Phi) \leq c(p,n) D \epsilon^{-p/n}$ for arbitrary $n \in \N$ and a constant $c(p,n)$. In this case, the dependence on $D$ is again linear, but coupled with the potentially quickly growing term $\epsilon^{-p/n}$. Similarly, when approximating the map $y \mapsto \tilde{\mathbf{u}}_{\cdot,\epsilontilde}^{\mathrm{rb}},$ one would obtain an approximation rate of $\epsilon^{-p/n}.$ In addition, our approach is more flexible than the naive approach in the sense that Assumptions \ref{ass:ParameterStiffness} and \ref{ass:ParameterRHS} could even be satisfied if the map $y \mapsto \mathbf{B}^{\mathrm{rb}}_{y,\epsilontilde}$ is non-smooth.

Now we analyze the dependence of our result to $p$ in more detail. We recall from Theorem \ref{thm:NNCoefficientApproximation}, that in our approach the sizes of approximating networks to achieve an error of $\epsilon$ depend only polylogarithmically on $1/\epsilon$, (up to a log factor) cubically on $d(\epsilontilde)$, are independent from or at worst linear in $D$, and depend linearly on $B_{M}(\epsilontilde, \epsilon),$ $B_{M}(\epsilontilde, \epsilon),$ $F_M\left(\epsilontilde, \epsilon\right)$, $F_{L}(\epsilontilde, \epsilon)$, respectively. 
First of all, the dependence on $p$ materializes through the quantities  $B_{M}(\epsilontilde, \epsilon),$ $B_{L}(\epsilontilde, \epsilon),$ $F_M\left(\epsilontilde, \epsilon\right)$, $F_{L}(\epsilontilde, \epsilon)$ from Assumptions \ref{ass:ParameterStiffness} and \ref{ass:ParameterRHS}. We have seen that in both examples above, the associated weight quantities $B_{M}(\epsilontilde, \epsilon),$ $F_{M}(\epsilontilde, \epsilon)$ scale like $pd(\epsilontilde)^2\cdot \mathrm{polylog}(1/\epsilon),$ whereas the depth quantities $B_{L}(\epsilontilde, \epsilon),$ $F_{L}(\epsilontilde, \epsilon)$ scale polylogarithmically in $1/\epsilon.$ 
Combining this observation with the statement of Theorem \ref{thm:NNCoefficientApproximation}, we can conclude that the governing quantity in the obtained approximation rates is given by the dimension of the solution manifold $d(\epsilontilde)$, derived by bounds on the Kolmogorov $N$-width (and, consequently, the inner $N$-width).

For problems  of the type \eqref{eq:affineProblem}, where the involved maps $\theta_q$ are sufficiently smooth and the right-hand side is parameter-independent, one can show (see for instance \cite[Equation 3.17]{BachCohenKolmogorov} or \cite{ohlberger2015reduced} that $W_N(S(\Ycal))$ (and hence also $\overline{W}_N(S(\Ycal))$) scales like $e^{-cN^{Q_b}}$ for some $c>0$. This implies for the commonly studied case $Q_b=p$ (such as in Example I of Section \ref{subsec:ExDiff}) that the dimension $d(\epsilontilde)$ of the reduced basis space scales like $ \mathcal{O}(\log_2(1/\epsilontilde)^p).$ This bound (which is based on a Taylor expansion of the solution map) has been improved only in very special cases of Example I (see for instance \cite{BachCohenKolmogorov,CohenDahmenBachmayr}) for small parameter dimensions $p$. Hence, by Theorem \ref{thm:NNCoefficientApproximation}, the number of non-zero weights necessary to approximate the parameter-to-solution map $\tilde{\mathbf{u}}^{\mathrm{h}}_{y,\epsilontilde}$, can be upper bounded by $\mathcal{O}\left(p \log_2^{3 p}(1/\epsilontilde)\log_2(\log_2(1/\epsilontilde))\log_2^2(1/\epsilon)\log_2^2(\log_2(1/\epsilon)) + D\log_2(1/\epsilontilde)^p\right)$ and the number of layers by $\mathcal{O}\left(p\log^2_2(1/\epsilon)\log_2(1/\epsilontilde)\right).$
This implies that in our results there is a (mild form of a) curse of dimensionality which can only be circumvented if the sensitivity of the Kolmogorov $N$-width with regards to the parameter dimension $p$ can be reduced further.

\section*{Acknowledgments}
 M.R. would like to thank Ingo Gühring 
for fruitful discussions on the topic.

G. K. acknowledges partial support by the Bundesministerium
 f\"ur Bildung und Forschung (BMBF) through the Berlin Institute for the
Foundations of Learning and Data (BIFOLD), Project AP4, RTG DAEDALUS (RTG 2433),
Projects P1, P3, and P8, RTG BIOQIC (RTG 2260), Projects P4 and P9, and by the
Berlin Mathematics Research Center MATH+, Projects EF1-1 and EF1-4.
P.P. was supported by a DFG Research Fellowship ``Shearlet-based energy functionals for anisotropic phase-field methods". R.S. acknowledges partial support by the DFG through grant RTG DAEDALUS (RTG 2433), Project P14.

\small


\normalsize
\appendix 

\section{Proofs of the Results from Section \ref{sec:NNCalc}} \label{app:NNCalculusProofs}

\subsection{Proof of Proposition \ref{prop:Multiplikation}}\label{app:ProofMultiplikation}

In this subsection, we will prove Proposition \ref{prop:Multiplikation}. As a preparation, we first prove the following special instance under which $M(\Phi^1 \conc \Phi^2)$ can be estimated by $\max\left\{ M(\Phi^1), M(\Phi^2)\right\}$.

\begin{lemma}\label{lem:TrivialConcEstimate}
Let $\Phi$ be a NN with $m$-dimensional output and $d$-dimensional input. If $\mathbf{a} \in \R^{1 \times m}$, then, for all $\ell = 1, \dots, L(\Phi)$,
$$
M_\ell(((\mathbf{a}, 0)) \conc \Phi) \leq M_\ell(\Phi).
$$
In particular, it holds that $M((\mathbf{a}, 0) \conc \Phi) \leq M(\Phi).$
Moreover, if $\mathbf{D} \in \R^{d \times n}$ such that, for every $k \leq d$ there is at most one $l_k\leq n$ such that $\mathbf{D}_{k,l_k} \neq 0$, then, for all $\ell = 1, \dots, L(\Phi)$,
$$
M_\ell(\Phi \conc ((\mathbf{D}, \mathbf{0}_{\R^d})) ) \leq M_\ell(\Phi).
$$
In particular, it holds that $M(\Phi \conc ((\mathbf{D}, \mathbf{0}_{\R^d})) ) \leq M(\Phi)$.
\end{lemma}
\begin{proof}
Let $\Phi = \big( (\mathbf{A}_1,\mathbf{b}_1), \dots, (\mathbf{A}_{L},\mathbf{b}_{L}) \big)$, and $\mathbf{a}, \mathbf{D}$ as in the statement of the lemma. Then the result follows if 
\begin{align}\label{eq:EstimateForSmallA}
\|\mathbf{a} \mathbf{A}_L\|_0 + \|\mathbf{a}\mathbf{b}_L\|_{0} \leq \|\mathbf{A}_L\|_0 + \|\mathbf{b}_L\|_{0} 
\end{align}
and 
\begin{align*}
\|\mathbf{A}_1 \mathbf{D}\|_0 \leq \|\mathbf{A}_1\|_0. 
\end{align*}
It is clear that $\|\mathbf{a} \mathbf{A}_L\|_0$ is less than the number of nonzero columns of $\mathbf{A}_L$ which is certainly bounded by $\|\mathbf{A}_L\|_0$. The same argument shows that $\|\mathbf{a}\mathbf{b}_L\|_{0} \leq \|\mathbf{b}_L\|_{0}$. This yields  \eqref{eq:EstimateForSmallA}. 

We have that for two vectors $\mathbf{p},\mathbf{q} \in \R^{k}$, $k \in \N$ and for all $\mu, \nu\in \R$
$$
\|\mu \mathbf{p} + \nu \mathbf{q} \|_{0} \leq I(\mu) \|\mathbf{p}\|_{0} + I(\nu)\|\mathbf{q} \|_{0},
$$
where $I(\gamma) = 0$ if $\gamma = 0$ and $I(\gamma) = 1$ otherwise. Also,
$$
\|\mathbf{A}_1 \mathbf{D}\|_0 = \left\|\mathbf{D}^T \mathbf{A}_1^T \right\|_0 = \sum_{l = 1}^n \left\|\left(\mathbf{D}^T \mathbf{A}_1^T\right)_{l, -}\right\|_0, 
$$
where, for a matrix $\mathbf{G}$, $\mathbf{G}_{l, -}$ denotes the $l$-th row of $\mathbf{G}$. Moreover, we have that for all $l \leq n$
$$
\left(\mathbf{D}^T \mathbf{A}_1^T\right)_{l, -} =  \sum_{k = 1}^d \left(\mathbf{D}^T\right)_{l, k} \left(\mathbf{A}_1^T\right)_{k, -} = \sum_{k = 1}^d \mathbf{D}_{k, l} \left(\mathbf{A}_1^T\right)_{k, -}.
$$
As a consequence, we obtain
\begin{align*}
\|\mathbf{A}_1 \mathbf{D}\|_0 &\leq  \sum_{l = 1}^n \left\| \sum_{k = 1}^d \mathbf{D}_{k, l} \left(\mathbf{A}_1^T\right)_{k, -}\right\|_0 \leq  \sum_{l = 1}^n \sum_{k = 1}^d I\left( \mathbf{D}_{k, l} \right) \left\| \left(\mathbf{A}_1^T\right)_{k, -}\right\|_0\\
&= \sum_{k = 1}^d I\left( \mathbf{D}_{k, l_k} \right) \left\| \left(\mathbf{A}_1^T\right)_{k, -}\right\|_0 \leq \|\mathbf{A}_1\|_0.
\end{align*}
\end{proof}

Now we are ready to prove Proposition \ref{prop:Multiplikation}.

\begin{proof}[Proof of Proposition \ref{prop:Multiplikation}]

Without loss of generality, assume that $Z\geq 1$.
By \cite[Lemma 6.2]{SchwabOption}, there exists a NN 
$\times^Z_{\epsilon}$ with input dimension $2$, output dimension $1$ such that for $\Phi_{\epsilon} \coloneqq \times^Z_{\epsilon}$
\begin{align}
   L\left(\Phi_{\epsilon}\right)&\leq 0.5 \log_2\left(\frac{n\sqrt{dl}}{\epsilon}\right)+\log_2(Z)+6, \label{eq:LayersMult}\\
    M\left(\Phi_{\epsilon}\right)&\leq 
    90 \cdot \left(\log_2\left(\frac{n\sqrt{dl}}{\epsilon}\right)+2\log_2(Z)+6\right),\label{eq:WeightMult}\\ 
    M_1\left(\Phi_{\epsilon}\right) &\leq 16, \text{ as well as } M_{L\left(\Phi_{\epsilon}\right)}\left(\Phi_{\epsilon}\right)\leq 3, \label{eq:WeightMultFirstAndLast}\\
\sup_{|a|,|b| \leq Z}\left| ab - \Realization^{\R^2}_{\varrho}\left(\Phi_{\epsilon}\right) (a,b) \right| &\leq \frac{\epsilon}{n\sqrt{dl}}. \label{eq:ErrorEstMult}
\end{align}

Since $\|\mathbf{A}\|_2,\|\mathbf{B}\|_2\leq Z$ we know that for every $i=1,...,d,~ k=1,...,n,~j=1,...,l$ we have that $|\mathbf{A}_{i,k}|,|\mathbf{B}_{k,j}|\leq Z$. We define, for $i\in \{1, \dots, d\}, k \in \{1, \dots, n\}, j \in \{1, \dots, l\}$, the matrix $\mathbf{D}_{i,k,j}$ such that, for all $\mathbf{A} \in \R^{d\times n}, \mathbf{B} \in \R^{n\times l}$
$$
    \mathbf{D}_{i,k,j}(\vect(\mathbf{A}),\vect(\mathbf{B})) = (\mathbf{A}_{i,k}, \mathbf{B}_{k,j}).
$$
Moreover, let
$$
    \Phi^Z_{i,k,j;\epsilon} \coloneqq \times^Z_{\epsilon} \conc  \left(\left(\mathbf{D}_{i,k,j},\mathbf{0}_{\R^2}\right)\right).
$$
We have, for all $i\in \{1, \dots, d\}, k \in \{1, \dots, n\}, j \in \{1, \dots, l\}$, that $L\left(\Phi^Z_{i,k,j;\epsilon}\right) = L\left(\times^Z_{\epsilon}\right)$ and by Lemma \ref{lem:TrivialConcEstimate} that $\Phi^Z_{i,k,j;\epsilon}$ satisfies \eqref{eq:LayersMult}, \eqref{eq:WeightMult}, \eqref{eq:WeightMultFirstAndLast} with $\Phi_{\epsilon} \coloneqq \Phi^Z_{i,k,j;\epsilon}$. Moreover, we have by \eqref{eq:ErrorEstMult}

\begin{align} \label{eq:ErrorMult}
    \sup_{(\vect(\mathbf{A}),\vect(\mathbf{B}))\in K^{Z}_{d,n,l}} \left|\mathbf{A}_{i,k}\mathbf{B}_{k,j}-\Realization^{K^Z_{d,n,l}}_\varrho \left(\Phi^Z_{i,j,k;\epsilon} \right)(\vect(\mathbf{A}),\vect(\mathbf{B})) \right| \leq \frac{\epsilon}{n\sqrt{dl}}.
\end{align}
As a next step, we set, for $\mathbf{1}_{\R^n} \in \R^n$ being a vector with each entry equal to 1,
\[
\Phi^Z_{i,j;\epsilon}\coloneqq \left(\left( \mathbf{1}_{\R^n},0\right)\right)\conc \Paral\left(\Phi^Z_{i,1,j;\epsilon},...,\Phi^Z_{i,n,j;\epsilon}\right)\conc \left(\left(  \begin{pmatrix} \mathbf{Id}_{\R^{n(d+l)}} \\ \vdots \\ \mathbf{Id}_{\R^{n(d+l)}}  \end{pmatrix}, \mathbf{0}_{\R^{n^2(d+l)}}  \right)\right),
\] 
which by Lemma \ref{lem:size} is a NN with $n\cdot (d+l)$-dimensional input and $1$-dimensional output such that \eqref{eq:LayersMult} holds with $\Phi_{\epsilon} \coloneqq \Phi^Z_{i,j;\epsilon}$. Moreover, by Lemmas \ref{lem:TrivialConcEstimate} and \ref{lem:size} and by \eqref{eq:WeightMult} we have that 
\begin{align}\label{eq:WeightMult2}
    M\left(\Phi^Z_{i,j;\epsilon}\right) \leq M\left(\Paral\left(\Phi^Z_{i,1,j;\epsilon},...,\Phi^Z_{i,n,j;\epsilon}\right)\right) \leq 90 n \cdot \left(\log_2\left(\frac{n\sqrt{dl}}{\epsilon}\right)+2\log_2(Z)+6\right). 
\end{align}
Additionally, by Lemmas \ref{lem:size} and \ref{lem:TrivialConcEstimate} and \eqref{eq:WeightMultFirstAndLast}, we obtain
\begin{align*}
M_1\left(\Phi^Z_{i,j;\epsilon}\right) &\leq M_1\left(\Paral\left(\Phi^Z_{i,1,j;\epsilon},...,\Phi^Z_{i,n,j;\epsilon}\right)\right) \leq 16 n.
\end{align*}
and
\begin{align}
M_{L\left(\Phi^Z_{i,j;\epsilon}\right)}\left(\Phi^Z_{i,j;\epsilon}\right) &= M_{L\left(\Phi^Z_{i,j;\epsilon}\right)}\left(\Paral\left(\Phi^Z_{i,1,j;\epsilon},...,\Phi^Z_{i,n,j;\epsilon}\right)\right) \leq 2 n. \label{eq:WeightMultFirstAndLast2}
\end{align}
By construction it follows that
$$
\Realization^{K^Z_{d,n,l}}_{\varrho}\left(\Phi^Z_{i,j;\epsilon}\right)(\vect(\mathbf{A}),\vect(\mathbf{B})) =  \sum_{k=1}^n \Realization^{K^Z_{d,n,l}}_{\varrho}\left(\Phi^Z_{i,k,j;\epsilon}\right)(\vect(\mathbf{A}),\vect(\mathbf{B}))
$$
and hence we have, by \eqref{eq:ErrorMult},
\begin{align*}
    &\sup_{(\vect(\mathbf{A}),\vect(\mathbf{B}))\in K^Z_{d,n,l}}\left| \sum_{k=1}^n \mathbf{A}_{i,k} \mathbf{B}_{k,j}- \Realization^ {K^Z_{d,n,l}}_{\varrho}\left(\Phi^Z_{i,j;\epsilon}\right)(\vect(\mathbf{A}),\vect(\mathbf{B})) \right| \leq \frac{\epsilon}{\sqrt{dl}}.
\end{align*}
As a final step, we define $\mult\coloneqq \Paral\left(\Phi^Z_{1,1;\epsilon},...,\Phi^Z_{d,1;\epsilon},...,\Phi^Z_{1,l;\epsilon},...,\Phi^Z_{d,l;\epsilon}\right)\conc \left(\left(  \begin{pmatrix} \mathbf{Id}_{\R^{n(d+l)}} \\ \vdots \\ \mathbf{Id}_{\R^{n(d+l)}}  \end{pmatrix}, \mathbf{0}_{\R^{dln(d+l)}}  \right)\right)$.
Then, by Lemma \ref{lem:size}, we have that \eqref{eq:LayersMult} is satisfied for $\Phi_{\epsilon} \coloneqq \mult$. This yields (i) of the asserted statement. Moreover, invoking Lemma \ref{lem:size}, Lemma \ref{lem:TrivialConcEstimate} and \eqref{eq:WeightMult2} yields that 
\begin{align*}
    M\left(\mult\right) &\leq 90 dln \cdot \left(\log_2\left(\frac{n\sqrt{dl}}{\epsilon}\right)+2\log_2(Z)+6\right),
    \end{align*}
which yields (ii) of the result. Moreover, by Lemma \ref{lem:size} and \eqref{eq:WeightMultFirstAndLast2} it follows that 
\begin{align*}
    M_1\left(\mult\right) \leq 16 d l n \text{ and } M_{L\left(\mult\right)}\left(\mult\right) \leq 2 d l n, 
\end{align*}
completing the proof of (iii). By construction and using the fact that for any $\mathbf{N}\in \R^{d\times l}$ there holds
\[
    \|\mathbf{N}\|_2\leq \sqrt{dl} \max_{i,j}|\mathbf{N}_{i,j}|,
\]
we obtain that 
\begin{align} 
    &\sup_{(\vect(\mathbf{A}),\vect(\mathbf{B}))\in K^Z_{d,n,l}}\left\|\mathbf{A} \mathbf{B}- \mathbf{matr}\left(\Realization^{K^Z_{d,n,l}}_{\varrho}\left(\mult\right)(\vect(\mathbf{A}),\vect(\mathbf{B}))\right)\right\|_2 \nonumber\\
    &\leq \sqrt{dl} \sup_{(\vect(\mathbf{A}),\vect(\mathbf{B}))\in K^Z_{d,n,l}}\max_{i=1,...,d,~ j=1,...,l} \left| \sum_{k=1}^n \mathbf{A}_{i,k} \mathbf{B}_{k,j}- \Realization^{K^Z_{d,n,l}}_{\varrho}\left(\Phi^Z_{i,j;\epsilon}\right)(\vect(\mathbf{A}),\vect(\mathbf{B}))  \right| \leq \epsilon. \label{eq:ErrorEstMult3}
\end{align}
Equation \eqref{eq:ErrorEstMult3} establishes (iv) of the asserted result. Finally, we have for any $(\vect(\mathbf{A}),\vect(\mathbf{B}))\in K^Z_{d,n,l}$ that
\begin{align*}
&\left\| \mathbf{matr}\left(\mathrm{R}_{\varrho}^{K^Z_{d,n,l}}\left(\mult\right) (\vect(\mathbf{A}),\vect(\mathbf{B}))\right)\right\|_2 \\
&\leq  \left\| \mathbf{matr}\left(\mathrm{R}_{\varrho}^{K^Z_{d,n,l}}\left(\mult\right) (\vect(\mathbf{A}),\vect(\mathbf{B}))\right)-\mathbf{A} \mathbf{B}\right\|_2  +\|\mathbf{A} \mathbf{B}\|_2 \\ &\leq \epsilon+\|\mathbf{A}\|_2 \|\mathbf{B}\|_2\leq \epsilon+Z^2\leq 1+Z^2.
\end{align*}
This demonstrates that (v) holds and thereby finishes the proof.
\end{proof}

\subsection{Proof of Theorem \ref{thm:Inverse}} \label{app:InverseProof}

The objective of this subsection is to prove of Theorem \ref{thm:Inverse}. Towards this goal, we construct NNs which emulate the map $\mathbf{A}\mapsto \mathbf{A}^k$ for $k\in \N$ and square matrices $\mathbf{A}.$ This is done by heavily using Proposition \ref{prop:Multiplikation}. First of all, as a direct consequence of Proposition \ref{prop:Multiplikation} we can estimate the sizes of the emulation of the multiplication of two squared matrices. Indeed, there exists a universal constant $C_1>0$ such that for all $d \in \N$, $Z>0$, $\epsilon\in (0,1)$ 
\begin{enumerate}[(i)]
    \item $L\left(\Phi_{\mathrm{mult}; \epsilon}^{Z, d, d, d}\right)\leq C_1\cdot \left(\log_2\left(1/\epsilon\right) +\log_2\left(d\right)+\log_2\left(\max\left\{1,Z\right\}\right)\right)$,
    \item $M\left(\Phi_{\mathrm{mult}; \epsilon}^{Z, d, d, d}\right) \leq C_1\cdot \left(\log_2\left(1/\epsilon\right) +\log_2\left(d\right)+\log_2\left(\max\left\{1,Z\right\}\right)\right) d^3$,
    \item  $M_1\left(\Phi_{\mathrm{mult}; \epsilon}^{Z, d, d, d}\right) \leq C_1 d^3, \qquad \text{as well as} \qquad M_{L\left(\Phi_{\mathrm{mult}; \epsilon}^{Z, d, d, d}\right)}\left(\Phi_{\mathrm{mult}; \epsilon}^{Z, d, d, d}\right)  \leq C_1 d^3$,
    \item $\sup_{(\vect(\mathbf{A}),\vect(\mathbf{B}))\in K^Z_{d,d,d}}\left\|\mathbf{A} \mathbf{B}- \mathbf{matr}\left(\Realization_{\varrho}^{K^Z_{d,d,d}}\left(\Phi_{\mathrm{mult}; \epsilon}^{Z, d, d, d}\right)(\vect(\mathbf{A}),\vect(\mathbf{B}))\right)\right\|_2 \leq \epsilon$,
    \item for every $(\vect(\mathbf{A}),\vect(\mathbf{B}))\in  K^Z_{d,d,d}$ we have
    \[
    \left\| \mathbf{matr}\left(\mathrm{R}_{\varrho}^{K^Z_{d,d,d}}\left(\Phi_{\mathrm{mult}; \epsilon}^{Z, d, d, d}\right) (\vect(\mathbf{A}),\vect(\mathbf{B}))\right)\right\|_2\leq \epsilon+\|\mathbf{A}\|_2 \|\mathbf{B}\|_2\leq \epsilon+Z^2 \leq 1+Z^2.
    \]
\end{enumerate}

One consequence of the ability to emulate the multiplication of matrices is that we can also emulate the squaring of matrices. We make this precise in the following definition. 

\begin{definition}\label{def:SquareAndConstant}
 For $d \in \N$, $Z>0$, and $\epsilon\in (0,1)$ we define the NN 
 \[
 \Phi^{Z, d}_{2;\epsilon} \coloneqq  \Phi_{\mathrm{mult}; \epsilon}^{Z, d, d, d} \conc \left(\left(  \begin{pmatrix} \mathbf{Id}_{\R^{d^2}} \\ \mathbf{Id}_{\R^{d^2}}  \end{pmatrix} , \mathbf{0}_{\R^{2d^2}}\right)\right),
 \]  
 which has $d^2 $-dimensional input and $d^2 $-dimensional output. By Lemma \ref{lem:size} we have that there exists a constant $C_{\mathrm{sq}}>C_1$ such that for all $d \in \N$, $Z>0$, $\epsilon\in (0,1)$ 
\begin{enumerate}[(i)]
        \item $L\left( \Phi^{Z, d}_{2;\epsilon}\right)\leq C_{\mathrm{sq}}\cdot \left(\log_2(1/\epsilon)+\log_2(d)+\log_2\left(\max\left\{1,Z\right\}\right)\right),$
       \item $M\left( \Phi^{Z, d}_{2;\epsilon}\right) \leq  C_{\mathrm{sq}} d^3 \cdot \left(\log_2(1/\epsilon)+\log_2(d)+\log_2\left(\max\left\{1,Z\right\}\right)\right), $
       \item  $M_1\left( \Phi^{Z, d}_{2;\epsilon}\right) \leq C_{\mathrm{sq}} d^3,\qquad \text{as well as} \qquad M_{L\left( \Phi^{Z, d}_{2;\epsilon}\right)}\left( \Phi^{Z, d}_{2;\epsilon}\right)  \leq C_{\mathrm{sq}} d^3, $
        \item $\sup_{\vect(\mathbf{A})\in K^Z_{d}}\left\|\mathbf{A}^2- \mathbf{matr}\left(\Realization_{\varrho}^{K^Z_{d}}\left( \Phi^{Z, d}_{2;\epsilon}\right)(\vect(\mathbf{A}))\right)\right\|_2 \leq \epsilon, $
        \item for all $\vect(\mathbf{A})\in K^Z_{d}$ we have 
        \[
        \left\| \mathbf{matr}\left(\mathrm{R}_{\varrho}^{K^Z_{d}}\left( \Phi^{Z, d}_{2;\epsilon}\right) (\vect(\mathbf{A}))\right)\right\|_2\leq \epsilon+ \|\mathbf{A}\|^2 \leq \epsilon+Z^2\leq 1+Z^2.
        \]
\end{enumerate}
\end{definition}
Our next goal is to approximate the map $\mathbf{A}\mapsto \mathbf{A}^k$ for an arbitrary $k\in \N_0.$ We start with the case that $k$ is a power of 2 and for the moment we only consider the set of all matrices the norm of which is bounded by $1/2$.

\begin{proposition} \label{prop:Pow2}
Let $d\in\N,~j\in \N,~$ as well as $\epsilon\in \left(0, 1/4\right)$. Then there exists a NN $\Phi^{1/2, d}_{2^j;\epsilon}$ with $d^2$- dimensional input and $d^2 $-dimensional output with the following properties:
\begin{enumerate}[(i)]
    \item $L\left(\Phi^{1/2, d}_{2^j;\epsilon}\right)\leq C_{\mathrm{sq}} j\cdot\left( \log_2(1/\epsilon)+\log_2(d) \right) +2C_{\mathrm{sq}}\cdot(j-1)$,
    \item $M\left(\Phi^{1/2, d}_{2^j;\epsilon}\right) \leq  C_{\mathrm{sq}} j d^3\cdot \left( \log_2(1/\epsilon)+\log_2(d) \right) +4C_{\mathrm{sq}}\cdot(j-1)d^3$,
    \item $M_1\left(\Phi^{1/2, d}_{2^j;\epsilon}\right) \leq C_{\mathrm{sq}} d^3,\qquad \text{as well as} \qquad M_{L\left(\Phi^{1/2, d}_{2^j;\epsilon}\right)}\left(\Phi^{1/2, d}_{2^j;\epsilon}\right) \leq C_{\mathrm{sq}} d^3$,
    \item $\sup_{\vect(\mathbf{A})\in  K^{1/2}_d}\left\|\mathbf{A}^{2^j}- \mathbf{matr}\left(\Realization_{\varrho}^{K^{1/2}_d}\left(\Phi^{1/2, d}_{2^j;\epsilon}\right)(\vect(\mathbf{A}))\right)\right\|_2 \leq \epsilon$,
    \item for every $\vect(\mathbf{A})\in  K^{1/2}_d$ we have 
    \[
    \left\| \mathbf{matr}\left(\Realization_{\varrho}^{K^{1/2}_d}\left(\Phi^{1/2, d}_{2^j;\epsilon}\right)(\vect(\mathbf{A}))\right)\right\|_2 \leq \epsilon +\left\|\mathbf{A}^{2^j}\right\|_2 \leq \epsilon +\left\|\mathbf{A}\right\|^{2^j}_2 \leq  \frac{1}{4}+\left(\frac{1}{2}\right)^{2^j}\leq \frac{1}{2}.
    \]
    \end{enumerate}
\end{proposition}

\begin{proof}
We show the statement by induction over $j\in \N$. 
For $j=1$, the statement follows by choosing $\Phi^{1/2, d}_{2;\epsilon}$ as in Definition \ref{def:SquareAndConstant}. 
Assume now, as induction hypothesis, that the claim holds for an arbitrary, but fixed $j\in \N$, i.e., there exists a NN $\Phi^{1/2, d}_{2^j;\epsilon}$ such that 
\begin{align}\label{eq:pow-1}
     \left\|\mathbf{matr}\!\left(\!\Realization_{\varrho}^{K_d^{1/2}}\!\left(\Phi^{1/2,d}_{2^{j};\epsilon}\right)\! (\vect(\mathbf{A}))\!\right)\! -\mathbf{A}^{2^{j}} \right\|_2 \leq \epsilon, \qquad \left\|\mathbf{matr}\left(\!\Realization_{\varrho}^{K_d^{1/2}}\!\left(\Phi^{1/2,d}_{2^{j};\epsilon}\right) (\vect(\mathbf{A}))\!\right)\! \right\|_2 \leq \epsilon+ \left(\frac{1}{2}\right)^{2^j}
\end{align}
and $\Phi^{1/2, d}_{2^j;\epsilon}$ satisfies (i),(ii),(iii). Now we define 
\[
\Phi^{1/2, d}_{2^{j+1};\epsilon} \coloneqq \Phi^{1, d}_{2;\frac{\epsilon}{4}}\odot \Phi^{1/2, d}_{2^{j};\epsilon}. 
\]
By the triangle inequality, we obtain for any $\vect(\mathbf{A})\in K_d^{1/2}$
\begin{align}
    &\left\|\mathbf{matr}\left(\Realization_{\varrho}^{K_d^{1/2}}\left(\Phi^{1/2,d}_{2^{j+1};\epsilon} \right)(\vect(\mathbf{A}))  \right)-\mathbf{A}^{2^{j+1}}\right\|_2 \nonumber \\ 
    &\leq \left\|\mathbf{matr}\left(\Realization_{\varrho}^{K_d^{1/2}}\left(\Phi^{1/2,d}_{2^{j+1};\epsilon}\right) (\vect(\mathbf{A})) \right)-\mathbf{A}^{2^{j}} \mathbf{matr}\left(\Realization_{\varrho}^{K_d^{1/2}}\left(\Phi^{1/2,d}_{2^{j};\epsilon}\right) (\vect(\mathbf{A}))\right)\right\|_2\nonumber\\
    &\qquad + \left\| \mathbf{A}^{2^{j}} \mathbf{matr}\left(\Realization_{\varrho}^{K_d^{1/2}}\left(\Phi^{1/2,d}_{2^{j};\epsilon}\right) (\vect(\mathbf{A}))\right) - \left(\mathbf{A}^{2^{j}}\right)^2  \right\|_2. \label{eq:estimate1111}
\end{align}
By construction of $\Phi^{1/2, d}_{2^{j+1};\epsilon}$, we know that 
$$
\left\|\mathbf{matr}\left(\Realization_{\varrho}^{K_d^{1/2}}\left(\Phi^{1/2,d}_{2^{j+1};\epsilon} \right)(\vect(\mathbf{A}))  \right) - \left(\mathbf{matr}\left(\Realization_{\varrho}^{K_d^{1/2}}\left(\Phi^{1/2, d}_{2^{j};\epsilon}\right) (\vect(\mathbf{A}))\right)\right)^2\right\|_2 \leq \frac{\epsilon}{4}.
$$
Therefore, using the triangle inequality and the fact that $\|\cdot \|_2$ is a submultiplicative operator norm, we derive that 
\begin{align}
&\left\|\mathbf{matr}\left(\Realization_{\varrho}^{K_d^{1/2}}\left(\Phi^{1/2,d}_{2^{j+1};\epsilon}\right) (\vect(\mathbf{A})) \right)-\mathbf{A}^{2^{j}} \mathbf{matr}\left(\Realization_{\varrho}^{K_d^{1/2}}\left(\Phi^{1/2,d}_{2^{j};\epsilon}\right) (\vect(\mathbf{A}))\right)\right\|_2\nonumber\\
&\leq \frac{\epsilon}{4} + \left\| \left(\mathbf{matr}\left(\Realization_{\varrho}^{K_d^{1/2}}\left(\Phi^{1/2,d}_{2^{j};\epsilon}\right) (\vect(\mathbf{A}))\right)\right)^2 - \mathbf{A}^{2^{j}} \mathbf{matr}\left(\Realization_{\varrho}^{K_d^{1/2}}\left(\Phi^{1/2,d}_{2^{j};\epsilon}\right) (\vect(\mathbf{A}))\right)\right\|_2\nonumber\\
& \leq \frac{\epsilon}{4} + \left\| \mathbf{matr}\left(\Realization_{\varrho}^{K_d^{1/2}}\left(\Phi^{1/2,d}_{2^{j};\epsilon}\right) (\vect(\mathbf{A}))\right) - \mathbf{A}^{2^{j}}\right\|_2  \left\|\mathbf{matr}\left(\Realization_{\varrho}^{K_d^{1/2}}\left(\Phi^{1/2,d}_{2^{j};\epsilon}\right) (\vect(\mathbf{A}))\right) \right\|_2\nonumber\\
&\leq \frac{\epsilon}{4} + \epsilon \cdot \left(\epsilon + \left(\frac{1}{2}\right)^{2^j}\right) \leq \frac{3}{4} \epsilon, \label{eq:estimate1576}
\end{align}
where the penultimate estimate follows by the induction hypothesis \eqref{eq:pow-1} and $\epsilon <1/4$. 
Hence, since $\|\cdot \|_2$ is a submultiplicative operator norm, we obtain 
\begin{align}
    \left\| \mathbf{A}^{2^{j}} \mathbf{matr}\left(\Realization_{\varrho}^{K_d^{1/2}}\left(\Phi^{1/2,d}_{2^{j};\epsilon}\right) (\vect(\mathbf{A}))\right) - \left(\mathbf{A}^{2^{j}}\right)^2  \right\|_2 \nonumber
    &\leq \left\| \mathbf{matr}\left(\Realization_{\varrho}^{K_d^{1/2}}\left(\Phi^{1/2,d}_{2^{j};\epsilon}\right) (\vect(\mathbf{A}))\right) - \mathbf{A}^{2^{j}} \right\|_2 \left\|\mathbf{A}^{2^{j}}\right\|_2 \\ &\leq \frac{\epsilon}{4}, \label{eq:estimate1613}
\end{align}
where we used $\left\|\mathbf{A}^{2^{j}}\right\|_2\leq 1/4$ and the induction hypothesis \eqref{eq:pow-1}. Applying \eqref{eq:estimate1613} and \eqref{eq:estimate1576} to \eqref{eq:estimate1111} yields 
\begin{align}
    \left\|\mathbf{matr}\left(\Realization_{\varrho}^{K_d^{1/2}}\left(\Phi^{1/2,d}_{2^{j+1};\epsilon}\right) (\vect(\mathbf{A}))  \right)-\mathbf{A}^{2^{j+1}}\right\|_2 \leq \epsilon.\label{eq:InductionStepFinished}
\end{align}
A direct consequence of \eqref{eq:InductionStepFinished} is that 
\begin{align}
    \left\|\mathbf{matr}\left(\Realization_{\varrho}^{K_d^{1/2}}\left(\Phi^{1/2, d}_{2^{j+1};\epsilon}\right) (\vect(\mathbf{A}))\right) \right\|_2 \leq \epsilon + \left\| \mathbf{A}^{2^{j+1}}\right\|_2 \leq \epsilon + \|\mathbf{A}\|^{2^{j+1}}_2.\label{eq:InductionStepFinished2}
\end{align}
The estimates \eqref{eq:InductionStepFinished} and \eqref{eq:InductionStepFinished2} complete the proof of the assertions (iv) and (v) of the proposition statement. Now we estimate the size of  $\Phi^{1/2, d}_{2^{j+1};\epsilon}$. By the induction hypothesis and Lemma \ref{lem:size}(a)(i), we obtain
\begin{align*}
    L\left(\Phi^{1/2, d}_{2^{j+1};\epsilon}\right)&= L\left(\Phi^{1, d}_{2;\frac{\epsilon}{4}}\right)+L\left(\Phi^{1/2, d}_{2^{j};\epsilon}\right) \\
    &\leq C_{\mathrm{sq}} \cdot \left(  \log_2(1/\epsilon) +\log_2(d) + \log_2(4)  + j \log_2(1/\epsilon) + 2\cdot(j-1) +j \log_2(d) \right) \\ &= C_{\mathrm{sq}} \cdot \left((j+1)\log_2(1/\epsilon) +(j+1)\log_2(d)+2j\right),
\end{align*}
which implies (i). Moreover, by the induction hypothesis and Lemma \ref{lem:size}(a)(ii), we conclude that 
\begin{align*}
    M\left(\Phi^{1/2,d}_{2^{j+1};\epsilon}\right) &\leq M\left(\Phi^{1,d}_{2;\frac{\epsilon}{4}}\right) + M\left(\Phi^{1/2, d}_{2^{j};\epsilon}\right) + M_1\left(\Phi^{1,d}_{2;\frac{\epsilon}{4}}\right) + M_{L\left(\Phi^{1/2, d}_{2^{j};\epsilon}\right)}\left(\Phi^{1/2,d}_{2^j;\epsilon}\right) \\
    &\leq C_{\mathrm{sq}} d^3\cdot\left( \log_2(1/\epsilon) +\log_2(d)+\log_2(4) + j \log_2(1/\epsilon) +j  \log_2(d) + 4\cdot(j-1) \right)+ 2C_{\mathrm{sq}}d^3 \\ 
    &= C_{\mathrm{sq}}  d^3\cdot\left( (j+1)\log_2(1/\epsilon) +(j+1)\log_2(d)+ 4j \right),
\end{align*}
implying (ii). Finally, it follows from Lemma \ref{lem:size}(a)(iii) in combination with the induction hypothesis as well Lemma  \ref{lem:size}(a)(iv) that 
\begin{align*}
    M_1\left(\Phi^{1/2, d}_{2^{j+1};\epsilon}\right) = M_1\left(\Phi^{1/2, d}_{2^{j};\epsilon}\right) \leq C_{\mathrm{sq}}   d^3, 
\end{align*}
as well as
\begin{align*}
    M_{L\left(\Phi^{1/2, d}_{2^{j+1};\epsilon}\right)}\left(\Phi^{1/2, d}_{2^{j+1};\epsilon}\right) = M_{L\left(\Phi^{1,d}_{2;\frac{\epsilon}{4}}\right)}\left(\Phi^{1,d}_{2;\frac{\epsilon}{4}}\right) \leq C_{\mathrm{sq}}  d^3,
\end{align*}
which finishes the proof.
\end{proof}

We proceed by demonstrating, how to build a NN that emulates the map $\mathbf{A}\mapsto \mathbf{A}^k$ for an arbitrary $k \in \N_0$. Again, for the moment we only consider the set of all matrices the norms of which are bounded by $1/2$. For the case of the set of all matrices the norms of which are bounded by an arbitrary $Z>0$, we refer to Corollary \ref{cor:pow}.

\begin{proposition}\label{prop:PowGeneral}
Let $d\in\N$, $k\in \N_0$, and $\epsilon\in \left(0,1/4\right)$. Then, there exists a NN $\Phi^{1/2, d}_{k;\epsilon}$ with $d^2$- dimensional input and $d^2$-dimensional output satisfying the following properties:
\begin{enumerate}[(i)]
\item \begin{align*}
        L\left(\Phi^{1/2, d}_{k;\epsilon}\right)&\leq  \left\lfloor \log_2\left(\max\{k,2\}\right)\right\rfloor  L\left( \Phi^{1,d}_{\mathrm{mult};\frac{\epsilon}{4}} \right)+ L\left(\Phi^{1/2, d}_{2^{\left\lfloor \log_2(\max\{k,2\})\right\rfloor};\epsilon}\right)\\
        &\leq  2C_{\mathrm{sq}} \left\lfloor \log_2\left(\max\{k,2\}\right) \right \rfloor \cdot \left(\log_2(1/\epsilon)+\log_2(d)+2 \right),
    \end{align*}
    \item $M\left(\Phi^{1/2, d}_{k;\epsilon}\right) \leq \frac{3}{2}C_{\mathrm{sq}}  d^3\cdot \left\lfloor \log_2\left(\max\{k,2\}\right)\right\rfloor \cdot \left(\left\lfloor\log_2\left(\max\{k,2\}\right)\right\rfloor +1\right)\cdot \left(\log_2(1/\epsilon)+\log_2(d)+4 \right)$,
    \item $M_1\left(\Phi^{1/2, d}_{k;\epsilon} \right) \leq C_{\mathrm{sq}}\cdot \left(\left\lfloor \log_2\left(\max\{k,2\}\right) \right\rfloor +1\right)   d^3, \qquad \text{as well as} \qquad M_{L\left(\Phi^{1/2, d}_{k;\epsilon}\right)}\left( \Phi^{1/2, d}_{k;\epsilon}\right)\leq C_{\mathrm{sq}}  d^3, $
    \item $\sup_{\vect(\mathbf{A})\in  K^{1/2}_d}\left\|\mathbf{A}^k- \mathbf{matr}\left(\Realization_{\varrho}^{K^{1/2}_d}\left(\Phi^{1/2, d}_{k;\epsilon}\right)(\vect(\mathbf{A}))\right)\right\|_2 \leq \epsilon$,
    \item for any $\vect(\mathbf{A})\in  K^{1/2}_d$ we have 
    \[
    \left\| \mathbf{matr}\left(\Realization_{\varrho}^{K^{1/2}_d}\left(\Phi^{1/2, d}_{k;\epsilon}\right)(\vect(\mathbf{A}))\right)\right\|_2 \leq \epsilon +\|\mathbf{A}^k\|_2 \leq \frac{1}{4}+\|\mathbf{A}\|_2^k \leq \frac{1}{4}+\left(\frac{1}{2}\right)^k.
    \]
    \end{enumerate}
\end{proposition}

\begin{proof}
We prove the result per induction over $k\in \N_{0}$. 
The cases $k=0$ and $k=1$ hold trivially by defining the NNs 
\[
    \Phi^{1/2, d}_{0;\epsilon}\coloneqq \left(\left(\mathbf{0}_{\R^{d^2}\times \R^{d^2}},\vect(\mathbf{Id}_{\R^{d}})\right)\right),\qquad \Phi^{1/2, d}_{1;\epsilon}\coloneqq \left(\left(\mathbf{Id}_{\R^{d^2}},\mathbf{0}_{\R^{d^2}}\right)\right).
\]
For the induction hypothesis, we claim that the result holds true for all $k' \leq k \in \N$. If $k$ is a power of two, then the result holds per Proposition \ref{prop:Pow2}, thus we can assume without loss of generality, that $k$ is not a power of two. We define $j \coloneqq \lfloor\log_2(k)\rfloor$ such that, for $t \coloneqq k - 2^{j}$, we have that $0 < t < 2^{j}$. This implies that $A^k= A^{2^j}  A^{t}$.
Hence, by Proposition \ref{prop:Pow2} and by the induction hypothesis, respectively, there exist a NN $\Phi^{1/2,d}_{2^j;\epsilon}$ satisfying (i)-(v) of Proposition \ref{prop:Pow2} and a NN $\Phi^{1/2,d}_{t;\epsilon}$ satisfying (i)-(v) of the statement of this proposition. We now define the NN 
\[
    \Phi^{1/2, d}_{k;\epsilon}\coloneqq \Phi^{1, d,d,d}_{\mathrm{mult};\frac{\epsilon}{4}}\odot \Paral\left(\Phi^{1/2, d}_{2^j;\epsilon},\Phi^{1/2, d}_{t;\epsilon}\right)\conc \left(\left(  \begin{pmatrix} \mathbf{Id}_{\R^{d^2}} \\ \mathbf{Id}_{\R^{d^2}}  \end{pmatrix} , \mathbf{0}_{\R^{2d^2}}\right)\right).
\]
By construction and Lemma \ref{lem:size}(a)(iv), we first observe that 
\begin{align*}
    M_{L\left(\Phi^{1/2,d}_{k;\epsilon} \right)}\left(\Phi^{1/2,d}_{k;\epsilon} \right) = M_{L\left( \Phi^{1,d,d,d}_{\mathrm{mult};\frac{\epsilon}{4}} \right)}\left( \Phi^{1,d,d,d}_{\mathrm{mult};\frac{\epsilon}{4}} \right) \leq C_{\mathrm{sq}}  d^3.
\end{align*}
Moreover, we obtain by the induction hypothesis as well as Lemma \ref{lem:size}(a)(iii) in combination with Lemma \ref{lem:size}(b)(iv) that 
\begin{align*}
    M_1\left( \Phi^{1/2,d}_{k;\epsilon} \right) &= M_1\left(\Paral\left(\Phi^{1/2,d}_{2^j;\epsilon},\Phi^{1/2,d}_{t;\epsilon}\right) \right) = M_1\left(\Phi^{1/2,d}_{2^j;\epsilon} \right) + M_{1}\left(\Phi^{1/2,d}_{t;\epsilon} \right) \\
    &\leq C_{\mathrm{sq}}  d^3+ (j+1)  C_{\mathrm{sq}}  d^3 = (j+2)  C_{\mathrm{sq}}  d^3.
\end{align*}
This shows (iii). To show (iv), we perform a similar estimate as the one following \eqref{eq:estimate1111}. By the triangle inequality,
\begin{align}
    &\left\|\mathbf{matr}\left(\Realization_{\varrho}^{K_d^{1/2}}\left(\Phi^{1/2,d}_{k;\epsilon}\right)  (\vect(\mathbf{A})) \right)-\mathbf{A}^{k}\right\|_2 \nonumber\\ 
    &\leq \left\|\mathbf{matr}\left(\Realization_{\varrho}^{K_d^{1/2}}\left(\Phi^{1/2,d}_{k;\epsilon}\right) (\vect(\mathbf{A}))  \right)- \mathbf{A}^{2^j} \mathbf{matr}\left(\Realization_{\varrho}^{K_d^{1/2}}\left(\Phi^{1/2,d}_{t;\epsilon}\right) (\vect(\mathbf{A}))\right)\right\|_2 \nonumber \\
    &\qquad + \left\| \mathbf{A}^{2^j} \mathbf{matr}\left(\Realization_{\varrho}^{K_d^{1/2}}\left(\Phi^{1/2,d}_{t;\epsilon}\right) (\vect(\mathbf{A}))\right) - \mathbf{A}^{2^{j}}  \mathbf{A}^{t}  \right\|_2.\label{eq:anotherEstimate1241}
\end{align}
By the construction of $\Phi^{1/2, d}_{k;\epsilon}$ and the Proposition \ref{prop:Multiplikation}, we conclude that 
\begin{align*}
    &\bigg\|\mathbf{matr}\left(\Realization_{\varrho}^{K_d^{1/2}}\left(\Phi^{1/2,d}_{k;\epsilon}\right) (\vect(\mathbf{A}))  \right) \\ &\quad - \mathbf{matr}\left(\Realization_{\varrho}^{K_d^{1/2}}\left(\Phi^{1/2,d}_{2^j;\epsilon}\right) (\vect(\mathbf{A}))\right) \mathbf{matr}\left(\Realization_{\varrho}^{K_d^{1/2}}\left(\Phi^{1/2,d}_{t;\epsilon}\right) (\vect(\mathbf{A}))\right)\bigg\|_2 \\ &\qquad \leq \frac{\epsilon}{4}.
\end{align*}
Hence, using \eqref{eq:anotherEstimate1241}, we can estimate
\begin{align*}
    &\left\|\mathbf{matr}\left(\Realization_{\varrho}^{K_d^{1/2}}\left(\Phi^{1/2,d}_{k;\epsilon} (\vect(\mathbf{A}))\right)  \right)-\mathbf{A}^{k}\right\|_2 \nonumber\\ 
    &\leq \frac{\epsilon}{4} + \left\|\mathbf{matr}\left(\Realization_{\varrho}^{K_d^{1/2}}\left(\Phi^{1/2,d}_{2^j;\epsilon}\right) (\vect(\mathbf{A}))\right) \mathbf{matr}\left(\Realization_{\varrho}^{K_d^{1/2}}\left(\Phi^{1/2,d}_{t;\epsilon}\right) (\vect(\mathbf{A}))\right)\right.\\
    &\qquad -   \left. \mathbf{A}^{2^j} \mathbf{matr}\left(\Realization_{\varrho}^{K_d^{1/2}}\left(\Phi^{1/2,d}_{t;\epsilon}\right) (\vect(\mathbf{A}))\right)  \right\|_2\\
    &\qquad \qquad + \left\|\mathbf{A}^{2^{j}}  \mathbf{matr}\left(\Realization_{\varrho}^{K_d^{1/2}}\left(\Phi^{1/2,d}_{t;\epsilon}\right) (\vect(\mathbf{A}))\right) - \mathbf{A}^{k} \right\|_2\\
     &\leq \frac{\epsilon}{4} + \left\|\mathbf{matr}\left(\Realization_{\varrho}^{K_d^{1/2}}\left(\Phi^{1/2,d}_{t;\epsilon}\right) (\vect(\mathbf{A}))\right) \right\|_2   \left\|\mathbf{matr}\left(\Realization_{\varrho}^{K_d^{1/2}}\left(\Phi^{1/2,d}_{2^{j};\epsilon} \right)(\vect(\mathbf{A}))\right)- \mathbf{A}^{2^j}  \right\|_2   \\ 
    & \qquad +  \left\| \mathbf{A}^{2^j}\right\|_2  \left\|\mathbf{matr}\left(\Realization_{\varrho}^{K_d^{1/2}}\left(\Phi^{1/2,d}_{t;\epsilon}\right) (\vect(\mathbf{A}))\right) -\mathbf{A}^{t}\right\|_2 \eqqcolon \frac{\epsilon}{4} + \mathrm{I}+\mathrm{II}.
    \end{align*}
    We now consider two cases: If $t=1,$ then we know by the construction of $\Phi^{1/2, d}_{1;\epsilon}$ that $\mathrm{II}=0$. Thus 
\begin{align*}
    \frac{\epsilon}{4}+\mathrm{I}+\mathrm{II} = \frac{\epsilon}{4} + \mathrm{I}\leq \frac{\epsilon}{4}+\|\mathbf{A}\|_2  \epsilon \leq \frac{3\epsilon}{4}\leq \epsilon.
\end{align*}
If $t\geq 2$, then 
\begin{align*}
    \frac{\epsilon}{4}+\mathrm{I}+\mathrm{II}&\leq \frac{\epsilon}{4}+ \left(\epsilon + \|\mathbf{A}\|^{t}+\|\mathbf{A}\|^{2^j} \right)   \epsilon \leq \frac{\epsilon}{4} +  \left(\frac{1}{4}+  \left(\frac{1}{2}\right)^{t}+ \left(\frac{1}{2} \right)^{2^j} \right)  \epsilon \leq  \frac{\epsilon}{4}+\frac{3\epsilon}{4} = \epsilon,
\end{align*}
where we have used that $\left(\frac{1}{2}\right)^{t}\leq \frac{1}{4}$ for $t \geq 2$. This shows (iv). In addition, by an application of the triangle inequality, we have that 
\[
\left\|\mathbf{matr}\left(\Realization_{\varrho}^{K_d^{1/2}}\left(\Phi^{1/2, d}_{k;\epsilon}\right) (\vect(\mathbf{A}))\right) \right\|_2 \leq \epsilon + \left\| \mathbf{A}^{k} \right\|_2 \leq \epsilon + \|\mathbf{A}\|_2^{k} \leq \frac{1}{4}+\left(\frac{1}{2}\right)^k.
\] 
This shows (v). Now we analyze the size of $\Phi^{1/2,d}_{k;\epsilon}$. We have by Lemma \ref{lem:size}(a)(i)  in combination with Lemma \ref{lem:size}(b)(i) and by the induction hypothesis that 
\begin{align*}
    L\left( \Phi^{1/2,d}_{k;\epsilon}\right) &\leq L\left(\Phi^{1,d,d,d}_{\mathrm{mult};\frac{\epsilon}{4}} \right) + \max\left\{L\left(\Phi^{1/2,d}_{2^j;\epsilon}\right), L\left(\Phi^{1/2,d}_{t;\epsilon} \right) \right\} \\
    &\leq L\left(\Phi^{1,d,d,d}_{\mathrm{mult};\frac{\epsilon}{4}} \right) + \max\left\{L\left(\Phi^{1/2,d}_{2^j;\epsilon}\right), (j-1)  L\left(\Phi^{1,d,d,d}_{\mathrm{mult};\frac{\epsilon}{4}}\right) +L\left(\Phi^{1/2,d}_{2^{j-1};\epsilon}  \right) \right\} \\
    &\leq L\left(\Phi^{1,d,d,d}_{\mathrm{mult};\frac{\epsilon}{4}} \right) + \max\left\{(j-1)  L\left(\Phi^{1,d,d,d}_{\mathrm{mult};\frac{\epsilon}{4}}\right)+L\left(\Phi^{1/2,d}_{2^j;\epsilon}\right), (j-1)  L\left(\Phi^{1,d,d,d}_{\mathrm{mult};\frac{\epsilon}{4}}\right) +L\left(\Phi^{1/2,d}_{2^{j-1};\epsilon}  \right) \right\} \\
    &\leq j  L\left(\Phi^{1,d,d,d}_{\mathrm{mult};\frac{\epsilon}{4}} \right) + L\left(\Phi^{1/2,d}_{2^{j};\epsilon}   \right) \\
    &\leq C_{\mathrm{sq}}  j\cdot \left(\log_2(1/\epsilon)+\log_2(d)+ 2\right) +C_{\mathrm{sq}}  j\cdot\left( \log_2(1/\epsilon)+\log_2(d) \right) +2C_{\mathrm{sq}}\cdot (j-1) \\
    &\leq 2  C_{\mathrm{sq}}  j\cdot \left(\log_2(1/\epsilon)+\log_2(d) +2 \right),
\end{align*}
which implies (i). Finally, we address the number of non-zero weights of the resulting NN. We first observe that, by Lemma \ref{lem:size}(a)(ii),
\begin{align*}
    M\left(\Phi^{1/2,d}_{k;\epsilon} \right) &\leq \left(M\left(\Phi^{1,d,d,d}_{\mathrm{mult};\frac{\epsilon}{4}} \right) +M_1\left( \Phi^{1,d,d,d}_{\mathrm{mult};\frac{\epsilon}{4}}\right)\right)+ M\left( \Paral\left(\Phi^{1/2,d}_{2^j;\epsilon}, \Phi^{1/2,d}_{t;\epsilon}\right)\right) \\
    & \qquad +  M_{L\left(\Paral\left(\Phi^{1/2,d}_{2^j;\epsilon}, \Phi^{1/2,d}_{t;\epsilon}\right)\right)}\left(\Paral\left(\Phi^{1/2,d}_{2^j;\epsilon}, \Phi^{1/2,d}_{t;\epsilon}\right)\right)\\
    &\eqqcolon \mathrm{I'}+\mathrm{II'}(a)+\mathrm{II'}(b).
\end{align*}
Then, by the properties of the NN $\Phi^{1,d,d,d}_{\mathrm{mult};\frac{\epsilon}{4}}$, we obtain 
\begin{align*}
    \mathrm{I'}=M\left(\Phi^{1, d,d,d}_{\mathrm{mult};\frac{\epsilon}{4}}\right) + M_1\left(\Phi^{1,d,d,d}_{\mathrm{mult};\frac{\epsilon}{4}}\right)&\leq C_{\mathrm{sq}}  d^3\cdot \left(\log_2(1/\epsilon)+\log_2(d)+2 \right) + C_{\mathrm{sq}}  d^3 \\ &= C_{\mathrm{sq}}  d^3\cdot \left(\log_2(1/\epsilon)+\log_2(d)+3 \right).
\end{align*}
Next, we estimate
\begin{align*}
     \mathrm{II'}(a)+\mathrm{II'}(b)=M\left( \Paral\left(\Phi^{1/2, d}_{2^j;\epsilon}, \Phi^{1/2, d}_{t;\epsilon}\right)\right) + M_{L\left(\Paral\left(\Phi^{1/2, d}_{2^j;\epsilon}, \Phi^{1/2, d}_{t;\epsilon}\right)\right)}\left(\Paral\left(\Phi^{1/2, d}_{2^j;\epsilon}, \Phi^{1/2, d}_{t;\epsilon}\right)\right).
\end{align*}
Without loss of generality we assume that $L\coloneqq L\left(\Phi^{1/2, d}_{t;\epsilon}\right) - L\left(\Phi^{1/2, d}_{2^j;\epsilon}\right) > 0.$ The other cases follow similarly. We have that $L\leq 2C_{\mathrm{sq}}  j\cdot \left(\log_2(1/\epsilon)+\log_2(d)+2\right)$ and, by the definition of the parallelization of two NNs with a different number of layers that  
\begin{align*}
    \mathrm{II'}(a) &= M\left(\Paral\left(\Phi^{1/2,d}_{2^j;\epsilon}, \Phi^{1/2,d}_{t;\epsilon} \right) \right)\\
    &= M\left(\Paral\left(\Phi^{\identity}_{d^2,L} \odot \Phi^{1/2,d}_{2^j;\epsilon}, \Phi^{1/2,d}_{t;\epsilon} \right) \right) \\
    &=M\left(\Phi^{\identity}_{d^2,L} \odot \Phi^{1/2,d}_{2^j;\epsilon}\right) + M\left(\Phi^{1/2,d}_{t;\epsilon} \right) \\
    &\leq M\left(\Phi^{\identity}_{d^2,L}\right)  +M_1\left( \Phi^{\identity}_{d^2,L}\right) + M_{L\left(\Phi^{1/2,d}_{2^j;\epsilon} \right)}\left(\Phi^{1/2,d}_{2^j;\epsilon}  \right) +  M\left(\Phi^{1/2,d}_{2^j;\epsilon} \right) +M\left(\Phi^{1/2,d}_{t;\epsilon} \right)\\
    &\leq 2d^2(L+1) + C_{\mathrm{sq}}  d^3 + M\left(\Phi^{1/2,d}_{t;\epsilon}\right) +M\left(\Phi^{1/2,d}_{2^j;\epsilon} \right),
\end{align*}
where we have used the definition of the parallelization for the first two equalities,  Lemma  \ref{lem:size}(b)(iii) for the third equality, Lemma  \ref{lem:size}(a)(ii) for the fourth inequality as well as the properties of $\Phi^{\identity}_{d^2,L}$ in combination with Proposition \ref{prop:Pow2}(iii) for the last inequality. Moreover, by the definition of the parallelization of two NNs with different numbers of layers, we conclude that 
\begin{align*}
    \mathrm{II'}(b)=M_{L\left(\Paral\left(\Phi^{1/2, d}_{2^j;\epsilon}, \Phi^{1/2, d}_{t;\epsilon}\right)\right)}\left(\Paral\left(\Phi^{1/2,d}_{2^j;\epsilon}, \Phi^{1/2,d}_{d;\epsilon}\right)\right) 
    &\leq d^2+C_{\mathrm{sq}}   d^3.  
\end{align*}
Combining the estimates on $\mathrm{I'}$, $\mathrm{II'}(a)$, and $\mathrm{II'}(b)$, we obtain by using the induction hypothesis that \begin{align*}
    M\left(\Phi^{1/2,d}_{k;\epsilon} \right) &\leq C_{\mathrm{sq}}  d^3\cdot \left(\log_2(1/\epsilon) + \log_2(d)+3 \right) + 2  d^2 \cdot (L+1)+d^2+C_{\mathrm{sq}}  d^3 + M\left(\Phi^{1/2,d}_{t;\epsilon}\right) + M\left(\Phi^{1/2,d}_{2^j;\epsilon}\right) \\
    &\leq C_{\mathrm{sq}}  d^3\cdot \left(\log_2(1/\epsilon)+\log_2(d)+4 \right) + 2  d^2 \cdot (L+2) + M\left(\Phi^{1/2,d}_{t;\epsilon}\right) + M\left(\Phi^{1/2,d}_{2^j;\epsilon}\right) \\
    &\leq C_{\mathrm{sq}} \cdot  (j+1)  d^3\cdot \left(\log_2(1/\epsilon)+\log_2(d)+4 \right) + 2  d^2 \cdot  (L+2)+M\left(\Phi^{1/2,d}_{t;\epsilon}\right) \\
    &\leq C_{\mathrm{sq}}\cdot (j+1)  d^3\cdot \left(\log_2(1/\epsilon)+\log_2(d)+4 \right) +2C_{\mathrm{sq}}  j   d^2\cdot \left(\log_2(1/\epsilon)+\log_2(d)+2 \right)\\ &\qquad+4d^2+M\left(\Phi^{1/2,d}_{t;\epsilon}\right) \\
    &\leq 3C_{\mathrm{sq}}\cdot(j+1)  d^3\cdot \left(\log_2(1/\epsilon)+\log_2(d)+4\right) + M\left(\Phi^{1/2,d}_{t;\epsilon}\right) \\
    &\leq 3C_{\mathrm{sq}}  d^3\cdot \left(j+1+ \frac{j\cdot(j+1)}{2}\right) \cdot \left( \log_2(1/\epsilon) +\log_2(d)+4)\right) \\
    &= \frac{3}{2}  C_{\mathrm{sq}}\cdot (j+1)\cdot(j+2)  d^3\cdot \left(\log_2(1/\epsilon)+\log_2(d)+4 \right).
\end{align*}
\end{proof}

Proposition \ref{prop:PowGeneral} only provides a construction of a NN the ReLU-realization of which emulates a power of a matrix $\mathbf{A}$, under the assumption that $\|\mathbf{A}\|_2\leq 1/2$. We remove this restriction in the following corollary by presenting a construction of a NN $\Phi^{Z, d}_{k;\epsilon}$ the ReLU-realization of which approximates the map $\mathbf{A}\mapsto \mathbf{A}^k,$ on the set of all matrices $\mathbf{A}$ the norms of which are bounded by an arbitrary $Z>0$.
\begin{corollary}\label{cor:pow}
 There exists a universal constant $C_{\mathrm{pow}}>C_{\mathrm{sq}}$ such that for all $Z>0,$ $d\in \N$ and $k\in \N_0,$ 
there exists some NN  $\Phi^{Z,d}_{k;\epsilon}$ with the following properties:
 \begin{enumerate}[(i)]
    \item $L\left(\Phi^{Z, d}_{k;\epsilon}\right) \leq C_{\mathrm{pow}} \log_2\left(\max\{k,2\}\right) \cdot \left( \log_2(1/\epsilon)+\log_2(d)+k \log_2\left(\max\left\{1,Z\right\} \right) \right)
    $,
    \item $M\left(\Phi^{Z,d}_{k;\epsilon}\right) \leq C_{\mathrm{pow}}\log_2^2\left(\max\{k,2\}\right) d^3 \cdot \left( \log_2(1/\epsilon)+\log_2(d)+k \log_2\left(\max\left\{1,Z\right\} \right) \right)$,
    \item $M_1\left(\Phi^{Z,d}_{k;\epsilon} \right) \leq C_{\mathrm{pow}} \log_2\left(\max\{k,2\}\right)  d^3, \qquad \text{as well as} \qquad  M_{L\left(\Phi^{Z,d}_{k;\epsilon}\right)}\left( \Phi^{Z}_{k;\epsilon}\right)\leq C_{\mathrm{pow}} d^3$,
    \item $\sup_{\vect(\mathbf{A})\in  K^{Z}_d}\left\|\mathbf{A}^k- \mathbf{matr}\left(\Realization_{\varrho}^{K^{Z}_d}\left(\Phi^{Z,d}_{k;\epsilon}\right)(\vect(\mathbf{A}))\right)\right\|_2 \leq \epsilon$,
    \item for any $\vect(\mathbf{A})\in  K^{Z}_d$ we have 
    \[
    \left\| \mathbf{matr}\left(\Realization_{\varrho}^{K^{Z}_d}\left(\Phi^{Z,d}_{k;\epsilon}\right)(\vect(\mathbf{A}))\right)\right\|_2 \leq \epsilon + \|\mathbf{A}^k\|_2 \leq \epsilon +\|\mathbf{A}\|_2^k.
    \]
    \end{enumerate}
\end{corollary}

\begin{proof}
Let $\left((\mathbf{A}_1,\mathbf{b}_1),...,(\mathbf{A}_L,\mathbf{b}_L) \right) \coloneqq \Phi^{1/2,d}_{k;\frac{\epsilon}{2\max\left\{1,Z^k\right\}}}$ according to Proposition \ref{prop:PowGeneral}. Then the NN 
\[
    \Phi^{Z,d}_{k;\epsilon} \coloneqq \left(\left(\frac{1}{2Z}\mathbf{A}_1 ,\mathbf{b}_1\right), (\mathbf{A}_2,\mathbf{b}_2),...,(\mathbf{A}_{L-1},\mathbf{b}_{L-1}), \left( 2Z^k \mathbf{A}_L,2Z^k \mathbf{b}_L\right) \right)
\]
fulfills all of the desired properties.
\end{proof}

We have seen how to construct a NN that takes a matrix as an input and computes a power of this matrix. With this tool at hand, we are now ready to prove Theorem \ref{thm:Inverse}. 

\begin{proof}[Proof of Theorem \ref{thm:Inverse}]
By the properties of the partial sums of the Neumann series, for $m\in \N$ and every $\vect(\mathbf{A})\in K_d^{1-\delta},$ we have that  
\begin{align*}
    \left\|\left(\mathbf{Id}_{\R^{d}}-\mathbf{A}\right)^{-1}- \sum_{k=0}^{m}\mathbf{A}^k\right\|_2 &=\left\|\left(\mathbf{Id}_{\R^{d}}-\mathbf{A}\right)^{-1}  \mathbf{A}^{m+1}\right\|_2 \leq \left\|\left(\mathbf{Id}_{\R^{d}}-\mathbf{A}\right)^{-1}\right\|_2  \|\mathbf{A}\|_2^{m+1} \\
    &\leq \frac{1}{1-(1-\delta)}\cdot (1-\delta)^{m+1} = \frac{(1-\delta)^{m+1}}{\delta}.
\end{align*}
Hence, for 
\[
m(\epsilon,\delta) = \left\lceil\log_{1-\delta}(2)  \log_2\left(\frac{\epsilon\delta}{2}\right)\right\rceil = \left\lceil \frac{\log_2(\epsilon)+\log_2(\delta)-1}{\log_{2}(1-\delta)}\right\rceil \geq \frac{\log_2(\epsilon)+\log_2(\delta)-1}{\log_{2}(1-\delta)}
\]
we obtain
\begin{align*}
    \left\|\left(\mathbf{Id}_{\R^{d}}-\mathbf{A}\right)^{-1}- \sum_{k=0}^{m(\epsilon,\delta)}\mathbf{A}^k\right\|_2 \leq \frac{\epsilon}{2}.
\end{align*}

Let now
 \begin{align*}
    &\left((\mathbf{A}_1,\mathbf{b}_1),...,(\mathbf{A}_L,\mathbf{b}_L) \right) \\
    &\coloneqq \left(\left(\left(\mathbf{Id}_{\R^{d^2}}|...|\mathbf{Id}_{\R^{d^2}}\right),\mathbf{0}_{\R^{d^2}} \right)\right) \odot \Paral\left(\Phi^{1,d}_{1;\frac{\epsilon}{2\left(m(\epsilon,\delta)-1\right)}},..., \Phi^{1,d}_{ m(\epsilon,\delta);\frac{\epsilon}{2\left(m(\epsilon,\delta)-1\right)}}\right)\conc \left(\left(  \begin{pmatrix} \mathbf{Id}_{\R^{d^2}} \\ \vdots \\ \mathbf{Id}_{\R^{d^2}}  \end{pmatrix} , \mathbf{0}_{\R^{2m(\epsilon,\delta)d^2}}\right)\right),    
 \end{align*}
where $\left(\mathbf{Id}_{\R^{d^2}}|...|\mathbf{Id}_{\R^{d^2}}\right)\in \R^{ d^2\times m(\epsilon,\delta)\cdot d^2}$. Then we set
\[
\Phi^{1-\delta,d}_{\mathrm{inv};\epsilon} \coloneqq \left((\mathbf{A}_1,\mathbf{b}_1),...,\left(\mathbf{A}_L,\mathbf{b}_L+\vect\left(\mathbf{Id}_{\R^{d}}\right)\right) \right). \]
We have for any $\vect(\mathbf{A})\in K^{1-\delta}_{d}$
\begin{align*}
     & \left\|\left(\mathbf{Id}_{\R^{d}}-\mathbf{A} \right)^{-1} - \mathbf{matr}\left( \Realization_\varrho^{K^{1-\delta}_{d}}\left(\Phi^{1-\delta,d}_{\mathrm{inv};\epsilon}\right)(\vect(\mathbf{A}))\right)   \right\|_2 \\ 
     &\leq \left\|\left(\mathbf{Id}_{\R^{d}}-\mathbf{A}\right)^{-1}- \sum_{k=0}^{m(\epsilon,\delta)}\mathbf{A}^k\right\|_2 +  \left\|  \sum_{k=0}^{m(\epsilon,\delta)}\mathbf{A}^k - \mathbf{matr}\left( \Realization_\varrho^{K^{1-\delta}_{d}}\left(\Phi^{1-\delta,d}_{\mathrm{inv};\epsilon}\right)(\vect(\mathbf{A}))\right)  \right\|_2 \\ 
     &\leq \frac{\epsilon}{2} + \sum_{k=2}^{m(\epsilon,\delta)} \left\|\mathbf{A}^k- \mathbf{matr}\left(\Realization_\varrho^{K^{1-\delta}_{d}}\left(\Phi^{1,d}_{k;\frac{\epsilon}{2\left(m(\epsilon,\delta)-1\right)}}\right)(\vect(\mathbf{A})) \right) \right\|_2 \\
     &\leq \frac{\epsilon}{2} + \left(m(\epsilon,\delta)-1\right)  \frac{\epsilon}{2(m(\epsilon,\delta)-1)}  = \epsilon,
\end{align*}
where we have used that 
\begin{align*}
    \left\| \mathbf{A}-\mathbf{matr}\left(\Realization_\varrho^{K^{1-\delta}_{d}}\left(\Phi^{1, d}_{1;\frac{\epsilon}{2\left(m(\epsilon,\delta)-1\right)}}\right)(\vect(\mathbf{A})) \right) \right\|_2 = 0.
\end{align*}
This completes the proof of (iii). Moreover, (iv) is a direct consequence of (iii). Now we analyze the size of the resulting NN. First of all, we have by Lemma \ref{lem:size}(b)(i) and Corollary \ref{cor:pow} that
\begin{align*}
    L\left(\Phi^{1-\delta, d}_{\mathrm{inv};\epsilon} \right) &= \max_{k=1,...,m(\epsilon,\delta)} L\left(\Phi^{1, d}_{k;\frac{\epsilon}{2\left(m(\epsilon,\delta)-1\right)}}\right) \\
&\leq C_{\mathrm{pow}}  \log_2\left(m(\epsilon,\delta)-1 \right)\cdot  \left(\log_2\left( 1/\epsilon\right)+1+\log_2\left(m(\epsilon,\delta)-1 \right)+\log_2(d)\right) \\
&\leq C_{\mathrm{pow}}  \log_2\left(\frac{\log_2\left(0.5  \epsilon\delta\right)}{\log_{2}(1-\delta)} \right)\cdot  \left(\log_2\left(1/ \epsilon\right)+1+\log_2\left(\frac{\log_2\left(0.5  \epsilon\delta\right)}{\log_{2}(1-\delta)} \right)+\log_2(d)\right), 
\end{align*}  
which implies (i). Moreover, by Lemma \ref{lem:size}(b)(ii), Corollary \ref{cor:pow} and the monotonicity of the logarithm, we obtain

\begin{align*}
M\left(\Phi^{1-\delta, d}_{\mathrm{inv};\epsilon}\right) &\leq 3\cdot \left(\sum_{k=1}^{m(\epsilon,\delta)} M\left(\Phi^{1,d}_{k;\frac{\epsilon}{2\left(m(\epsilon,\delta)-1\right)}} \right) \right)\\ 
&\qquad+4C_{\mathrm{pow}}  m(\epsilon,\delta) d^2   \log_2\left(m(\epsilon,\delta)\right)\cdot  \left(\log_2\left(1/ \epsilon\right)+1+\log_2\left(m(\epsilon,\delta) \right)+\log_2(d)\right)\\
&\leq 3 C_{\mathrm{pow}}\cdot \left(\sum_{k=1}^{m(\epsilon,\delta)}\log_2^2(\max\{k,2\})\right)   d^3  \cdot \left(\log_2(1/\epsilon)+1+\log_2\left(m(\epsilon,\delta)\right)+ \log_2(d) \right) \\
&\qquad +5  m(\epsilon,\delta) d^2  C_{\mathrm{pow}}  \log_2\left(m(\epsilon,\delta)\right)\cdot  \left(\log_2\left(1/ \epsilon\right)+1+\log_2\left(m(\epsilon,\delta) \right)+\log_2(d)\right) \eqqcolon \mathrm{I}.
\end{align*}

Since $\sum_{k=1}^{m(\epsilon,\delta)}\log_2^2(\max\{k,2\})\leq m(\epsilon,\delta)\log_2^2(m(\epsilon,\delta)),$ we obtain for some constant  $C_{\mathrm{inv}}>C_{\mathrm{pow}}$ that 
\begin{align*}
    \mathrm{I} \leq C_{\mathrm{inv}} m(\epsilon,\delta) \log_2^2(m(\epsilon,\delta))d^3\cdot \left(\log_2(1/\epsilon) + \log_2\left(m(\epsilon,\delta)\right) + \log_2(d) \right).
\end{align*}
This completes the proof.
\end{proof}

\section{Proof of Theorem \ref{thm:NNCoefficientApproximation}} \label{app:PDENNProofs}
We start by establishing a bound on $ \left\|\mathbf{Id}_{\R^{d(\epsilontilde)}}-\alpha \mathbf{B}_{y,\epsilontilde}^{\mathrm{rb}} \right\|_2.$

\begin{proposition}\label{prop:IndependentAlpha}
For any  $\alpha \in \left(0, {1}/{C_{\mathrm{cont}}}\right)$ and $\delta\coloneqq \alpha C_{\mathrm{coer}}\in (0, 1)$ there holds
\begin{align*}
    \left\|\mathbf{Id}_{\R^{d(\epsilontilde)}}-\alpha \mathbf{B}_{y,\epsilontilde}^{\mathrm{rb}} \right\|_2 \leq 1-\delta <1, \quad \text{ for all } y\in \Ycal,~\epsilontilde>0.
\end{align*}
\begin{proof}
    Since $ \mathbf{B}_{y,\epsilontilde}^{\mathrm{rb}}$ is symmetric, there holds that 
    \begin{align*}
         \left\|\mathbf{Id}_{\R^{d(\epsilontilde)}}-\alpha \mathbf{B}_{y,\epsilontilde}^{\mathrm{rb}} \right\|_2 = \max_{\mu \in \sigma \left(\mathbf{B}_{y,\epsilontilde}^{\mathrm{rb}}\right)} \left|1-\alpha \mu \right| 
         \leq \max_{\mu \in [C_{\mathrm{coer}},C_{\mathrm{cont}}]}  \left| 1-\alpha \mu \right| 
         =1-\alpha C_{\mathrm{coer}} = 1-\delta<1,
    \end{align*}
for all $y\in \Ycal,~\epsilontilde>0.$
\end{proof}
\end{proposition}

With an approximation to the parameter-dependent stiffness matrices with respect to a RB, due to Assumption \ref{ass:ParameterStiffness}, we can next state a construction of a NN the ReLU-realization of which approximates the map $y\mapsto  \left(\mathbf{B}_{y,\epsilontilde}^{\mathrm{rb}}\right)^{-1}$. As a first step, we observe the following remark.

\begin{remark}\label{rem:ParameterStiffness}
It is not hard to see that if $ \left((\mathbf{A}^1_{\epsilontilde,\epsilon},\mathbf{b}^1_{\epsilontilde,\epsilon}),...,(\mathbf{A}^L_{\epsilontilde,\epsilon},\mathbf{b}^L_{\epsilontilde,\epsilon})\right)\coloneqq \Phi^{\mathbf{B}}_{\epsilontilde,\epsilon}$ is the NN of Assumption \ref{ass:ParameterStiffness}, then for 
\[
\Phi^{\mathbf{B},\mathbf{Id}}_{\epsilontilde,\epsilon}\coloneqq \left((\mathbf{A}^1_{\epsilontilde,\epsilon},\mathbf{b}^1_{\epsilontilde,\epsilon}),...,(- \mathbf{A}^L_{\epsilontilde,\epsilon},- \mathbf{b}^L_{\epsilontilde,\epsilon}+\vect\left(\mathbf{Id}_{\R^{d(\epsilontilde)}}) \right)\right)
\] 
we have that 
\begin{align*}
    \sup_{y\in \Ycal} \left\| \mathbf{Id}_{\R^{d(\epsilontilde)}} - \alpha \mathbf{B}^{\mathrm{rb}}_{y,\epsilontilde} -\mathbf{matr}\left(\Realization^\Ycal_{\varrho} \left(\Phi^{\mathbf{B},\mathbf{Id}}_{\epsilontilde,\epsilon} \right)(y) \right)  \right\|_2 \leq \epsilon,
\end{align*}
as well as $M\left(\Phi^{\mathbf{B},\mathbf{Id}}_{\epsilontilde,\epsilon}\right) \leq B_{M}(\epsilontilde, \epsilon) + d(\epsilontilde)^2$ and $L\left(\Phi^{\mathbf{B},\mathbf{Id}}_{\epsilontilde,\epsilon}\right) = B_{L}\left(\epsilontilde, \epsilon\right)$.
\end{remark}

Now we present the construction of the NN emulating $y\mapsto  \left(\mathbf{B}_{y,\epsilontilde}^{\mathrm{rb}}\right)^{-1}$.

\begin{proposition}\label{prop:InverseStiffnessNN}
Let $\epsilontilde\geq \epsilonhat,\epsilon\in \left(0,  \alpha/4 \cdot \min\{1,C_{\mathrm{coer}}\} \right)$ and $\epsilon'\coloneqq  3/8 \cdot \epsilon \alpha C_{\mathrm{coer}}^2 <\epsilon$. Assume that Assumption \ref{ass:ParameterStiffness} holds. We define
\begin{align*}
    \Phi^{\mathbf{B}}_{\mathrm{inv};\epsilontilde,\epsilon} \coloneqq \left(\left(\alpha\mathbf{Id}_{\R^{d(\epsilontilde)}},\mathbf{0}_{\R^{d(\epsilontilde)}}  \right) \right)  \conc \Phi^{1-\delta/2,d(\epsilontilde)}_{\mathrm{inv};\frac{\epsilon}{2\alpha}} \odot \Phi^{\mathbf{B},\mathbf{Id}}_{\epsilontilde,\epsilon'},
\end{align*}
which has $p$-dimensional input and $d(\epsilontilde)^2$- dimensional output. 

Then, there exists a constant $C_B=C_B(C_{\mathrm{coer}},C_{\mathrm{cont}} ) > 0$ such that 
\begin{enumerate}[(i)]
\item
$
L\left( \Phi^{\mathbf{B}}_{\mathrm{inv};\epsilontilde,\epsilon} \right) 
 \leq C_{B} \log_2(\log_2(1/\epsilon))\big(\log_2(1/\epsilon) + \log_2(\log_2(1/\epsilon))+\log_2(d(\epsilontilde))\big) + B_L(\epsilontilde, \epsilon'),
$
\item 
$
M\left( \Phi^{\mathbf{B}}_{\mathrm{inv};\epsilontilde,\epsilon} \right)
\leq C_{B} \log_2(1/\epsilon)  \log_2^2(\log_2(1/\epsilon)) d(\epsilontilde)^3\cdot \big(\log_2(1/\epsilon)+\log_2(\log_2(1/\epsilon))+\log_2(d(\epsilontilde)) \big) + 2 B_M\left(\epsilontilde,\epsilon'\right),
$ 
\item $
    \sup_{y\in \Ycal} \left\|\left(\mathbf{B}_{y,\epsilontilde}^{\mathrm{rb}}\right)^{-1} -\mathbf{matr}\left(\Realization_\varrho^\Ycal\left(\Phi^{\mathbf{B}}_{\mathrm{inv};\epsilontilde,\epsilon} \right)(y)\right) \right\|_2 \leq \epsilon$,
\item $
    \sup_{y\in \Ycal} \left\|\mathbf{G}^{1/2}\mathbf{V}_\epsilontilde \cdot \left(\left(\mathbf{B}_{y,\epsilontilde}^{\mathrm{rb}}\right)^{-1} -\mathbf{matr}\left(\Realization_\varrho^\Ycal\left(\Phi^{\mathbf{B}}_{\mathrm{inv};\epsilontilde,\epsilon} \right)(y)\right)\right) \right\|_2 \leq \epsilon$,
\item $\sup_{y\in \Ycal} \left\|\mathbf{matr}\left(\Realization_\varrho^\Ycal\left(\Phi^{\mathbf{B}}_{\mathrm{inv};\epsilontilde,\epsilon} \right)(y)\right) \right\|_2 \leq \epsilon + \frac{1}{C_{\mathrm{coer}}},$
\item $
    \sup_{y\in \Ycal} \left\|\mathbf{G}^{1/2}\mathbf{V}_\epsilontilde\mathbf{matr}\left(\Realization_\varrho^\Ycal\left(\Phi^{\mathbf{B}}_{\mathrm{inv};\epsilontilde,\epsilon} \right)(y)\right) \right\|_2 \leq \epsilon + \frac{1}{C_{\mathrm{coer}}}.
$ 
\end{enumerate}

\end{proposition}
\begin{proof}
First of all, for all $y\in \Ycal$ the matrix  $\mathbf{matr}\left(\Realization^{\Ycal}_{\varrho}\left(\Phi^{\mathbf{B}}_{\epsilontilde,\epsilon'}\right)(y) \right)$ is invertible. This can be deduced from the fact that 
\begin{align}\label{eq:NeumannEstimate003}
    \left\|\alpha  \mathbf{B}^{\mathrm{rb}}_{y,\epsilontilde}- \mathbf{matr}\left( \Realization^{\Ycal}_{\varrho}\left(\Phi^{\mathbf{B}}_{\epsilontilde,\epsilon'}\right)(y)\right)\right\|_2 \leq \epsilon'<\epsilon \leq  \frac{\alpha \min \{1,C_{\mathrm{coer}}\}}{4} \leq  \frac{\alpha C_{\mathrm{coer}}}{4}. 
\end{align}
Indeed, we estimate
\begin{align*}
&\min_{\mathbf{z} \in \R^{d(\epsilontilde)}\setminus \{0\}} \frac{\left|\mathbf{matr}\left( \Realization^{\Ycal}_{\varrho}\left(\Phi^{\mathbf{B}}_{\epsilontilde,\epsilon'}\right)(y)\right)\mathbf{z} \right|}{|\mathbf{z}|} \\
\text{\footnotesize [Reverse triangle inequality]} \quad 
&\quad \geq \min_{\mathbf{z} \in \R^{d(\epsilontilde)}\setminus \{0\}} \frac{\left|\alpha \mathbf{B}^{\mathrm{rb}}_{y,\epsilontilde} \mathbf{z}\right|}{|\mathbf{z}|} -  \max_{\mathbf{z} \in \R^{d(\epsilontilde)}\setminus \{0\}}\frac{\left|\alpha \mathbf{B}^{\mathrm{rb}}_{y,\epsilontilde} \mathbf{z}  - \mathbf{matr}\left( \Realization^{\Ycal}_{\varrho}\left(\Phi^{\mathbf{B}}_{\epsilontilde,\epsilon'}\right)(y)\right)\mathbf{z}\right|}{|\mathbf{z}|}\\
\text{\footnotesize [Definition of $\|.\|_2$]} \quad
& \quad \geq \left(\max_{\mathbf{z} \in \R^{d(\epsilontilde)}\setminus \{0\}} \frac{|\mathbf{z}|}{\left|\alpha \mathbf{B}^{\mathrm{rb}}_{y,\epsilontilde} \mathbf{z}\right|}\right)^{-1} - \left\|\alpha \mathbf{B}^{\mathrm{rb}}_{y,\epsilontilde} - \mathbf{matr}\left( \Realization^{\Ycal}_{\varrho}\left(\Phi^{\mathbf{B}}_{\epsilontilde,\epsilon'}\right)(y)\right)\right\|_2\\
\text{\footnotesize [Set $\tilde{\mathbf{z}} \coloneqq (\alpha \mathbf{B}^{\mathrm{rb}}_{y,\epsilontilde})\mathbf{z}$]} \quad 
 & \quad \geq \left(\max_{\tilde{\mathbf{z}} \in \R^{d(\epsilontilde)}\setminus \{0\}} \frac{|(\alpha \mathbf{B}^{\mathrm{rb}}_{y,\epsilontilde})^{-1} \tilde{\mathbf{z}}|}{\left|\tilde{\mathbf{z}}\right|}\right)^{-1} - \left\|\alpha \mathbf{B}^{\mathrm{rb}}_{y,\epsilontilde} - \mathbf{matr}\left( \Realization^{\Ycal}_{\varrho}\left(\Phi^{\mathbf{B}}_{\epsilontilde,\epsilon'}\right)(y)\right)\right\|_2\\
\text{\footnotesize [Definition of $\|\cdot\|_2$]} \quad 
& \quad  \geq \left \|\left(\alpha \mathbf{B}^{\mathrm{rb}}_{y,\epsilontilde}\right)^{-1}\right\|_2^{-1} - \left\|\alpha \mathbf{B}^{\mathrm{rb}}_{y,\epsilontilde} - \mathbf{matr}\left( \Realization^{\Ycal}_{\varrho}\left(\Phi^{\mathbf{B}}_{\epsilontilde,\epsilon'}\right)(y)\right)\right\|_2\\
\text{\footnotesize [By Equations \eqref{eq:NeumannEstimate003} and \eqref{eq:NormStiffness}]}\quad &\quad  \geq \alpha C_{\mathrm{coer}} - \frac{\alpha C_{\mathrm{coer}}}{4} \geq  \frac{3}{4} \alpha C_{\mathrm{coer}}.
\end{align*}
Thus it follows that 
\begin{align} \label{eq:NormInverseStiffness}
    \left\|\left(\mathbf{matr}\left( \Realization^{\Ycal}_{\varrho}\left(\Phi^{\mathbf{B}}_{\epsilontilde,\epsilon'}\right)(y)\right)\right)^{-1} \right\|_2 \leq \frac{4 }{3}\frac{1}{C_{\mathrm{coer}}\alpha}.
\end{align}
Then
\begin{align*}
    &\left\| \frac{1}{\alpha}\left(\mathbf{B}_{y,\epsilontilde}^{\mathrm{rb}} \right)^{-1} -\mathbf{matr}\left(\Realization_\varrho^\Ycal\left(  \Phi^{1-\delta/2,d(\epsilontilde)}_{\mathrm{inv};\frac{\epsilon}{2\alpha}} \odot \Phi^{\mathbf{B},\mathbf{Id}}_{\epsilontilde,\epsilon'}\right)(y) \right)\right\|_2 
    \\ 
    &\quad \leq \left\| \frac{1}{\alpha}\left(\mathbf{B}_{y,\epsilontilde}^{\mathrm{rb}} \right)^{-1} -\left(\mathbf{matr}\left(\Realization_\varrho^\Ycal\left( \Phi^{\mathbf{B}}_{\epsilontilde,\epsilon'}\right)(y) \right)\right)^{-1}\right\|_2 
    \\
    &\quad\quad + \left\| \left(\mathbf{matr}\left(\Realization_\varrho^\Ycal\left( \Phi^{\mathbf{B}}_{\epsilontilde,\epsilon'}\right)(y) \right)\right)^{-1} -\mathbf{matr}\left(\Realization_\varrho^\Ycal\left( \Phi^{1-\delta/2,d(\epsilontilde)}_{\mathrm{inv};\frac{\epsilon}{2\alpha}} \odot \Phi^{\mathbf{B},\mathbf{Id}}_{\epsilontilde,\epsilon'}\right)(y) \right)\right\|_2 \eqqcolon \mathrm{I}+\mathrm{II}.
\end{align*}
Due to the fact that for two invertible matrices $\mathbf{M},\mathbf{N},$  
\begin{align*}
    \left\|\mathbf{M}^{-1}-\mathbf{N}^{-1} \right\|_2 = \left\| \mathbf{M}^{-1}(\mathbf{N}-\mathbf{M})\mathbf{N}^{-1}\right\|_2 \leq \|\mathbf{M}-\mathbf{N}\|_2 \|\mathbf{M}^{-1}\|_2\|\mathbf{N}^{-1}\|_2,  
\end{align*}
we obtain
\begin{align*}
    \mathrm{I} &\leq \left\|\alpha\mathbf{B}^{\mathrm{rb}}_{y,\epsilontilde}- \mathbf{matr}\left(\Realization_\varrho^\Ycal\left( \Phi^{\mathbf{B}}_{\epsilontilde,\epsilon'}\right)(y) \right)\right\|_2   \left\|\left(\alpha\mathbf{B}_{y,\epsilontilde}^{\mathrm{rb}}\right)^{-1} \right\|_2  \left\| \left(\mathbf{matr}\left(\Realization_\varrho^\Ycal\left( \Phi^{\mathbf{B}}_{\epsilontilde,\epsilon'}\right)(y) \right)\right)^{-1}\right\|_2 
    \\
    &\leq \frac{3 }{8} \epsilon\alpha C_{\mathrm{coer}}^2   \frac{1}{\alpha C_{\mathrm{coer}}}   \frac{4}{3}\frac{1}{C_{\mathrm{coer}}\alpha} = \frac{\epsilon}{2\alpha},
\end{align*}
where we have used Assumption \ref{ass:ParameterStiffness}, Equation \eqref{eq:NormStiffness} and Equation \eqref{eq:NormInverseStiffness}.
Now we turn our attention to estimating II. First, observe that for every $y\in \Ycal$  by the triangle inequality and Remark \ref{rem:ParameterStiffness}, that
\begin{align*}
     \left\|  \mathbf{matr}\left(\Realization_{\varrho}^{\Ycal}\left(\Phi^{\mathbf{B},\mathbf{Id}}_{\epsilontilde,\epsilon'}\right)(y)\right)\right\|_2 &\leq \left\|  \mathbf{matr}\left(\Realization_{\varrho}^{\Ycal}\left(\Phi^{\mathbf{B},\mathbf{Id}}_{\epsilontilde,\epsilon'}\right)(y)\right)-\left(\mathbf{Id}_{\R^{d(\epsilontilde)}}-\alpha\mathbf{B}_{y,\epsilontilde}^{\mathrm{rb}}\right)\right\|_2 + \left\|\mathbf{Id}_{\R^{d(\epsilontilde)}}-\alpha\mathbf{B}_{y,\epsilontilde}^{\mathrm{rb}}\right\|_2 \\
    &\leq  \epsilon'+1-\delta \leq 1-\delta+\frac{\alpha C_{\mathrm{coer}}}{4} \leq 1-\delta+ \frac{\alpha C_{\mathrm{cont}}}{4} \leq 1-\delta+\frac{\delta}{2}=1-\frac{\delta}{2}.
\end{align*}
Moreover, have that $\epsilon/(2\alpha) \leq \alpha/(8\alpha)< 1/4.$ Hence, by Theorem \ref{thm:Inverse}, we obtain that $  \mathrm{II} \leq {\epsilon}/{2\alpha }.$ Putting everything together yields
\begin{align*}
    \sup_{y\in \Ycal} \left\|\frac{1}{\alpha}\left(\mathbf{B}_{y,\epsilontilde}^{\mathrm{rb}} \right)^{-1} -\mathbf{matr}\left(\Realization_\varrho^\Ycal\left(  \Phi^{1-\delta/2,d(\epsilontilde)}_{\mathrm{inv};\frac{\epsilon}{2\alpha}} \odot \Phi^{\mathbf{B},\mathbf{Id}}_{\epsilontilde,\epsilon'}\right)(y) \right)\right\|_2 \leq \mathrm{I}+\mathrm{II} \leq \frac{\epsilon}{\alpha}.
\end{align*}
Finally, by construction we can conclude that
\begin{align*}
    \sup_{y\in \Ycal} \left\|\left(\mathbf{B}_{y,\epsilontilde}^{\mathrm{rb}} \right)^{-1}-   \mathbf{matr}\left(\Realization^\Ycal_{\varrho}\left(\Phi^{\mathbf{B}}_{\mathrm{inv};\epsilontilde,\epsilon}\right)(y)\right)\right\|_2 \leq \epsilon.
\end{align*}
This implies (iii) of the assertion. Now, by Equation \eqref{eq:NormTrafo} we obtain
\begin{align*}
    \sup_{y\in \Ycal}\left\|\mathbf{G}^{1/2}\mathbf{V}_\epsilontilde\left(\mathbf{B}_{y,\epsilontilde}^{\mathrm{rb}}\right)^{-1} -\mathbf{G}^{1/2}\mathbf{V}_\epsilontilde\mathbf{matr}\left(\Realization_\varrho^\Ycal\left(\Phi^{\mathbf{B}}_{\mathrm{inv};\epsilontilde,\epsilon} \right)(y)\right) \right\|_2 \leq  \left\|\mathbf{G}^{1/2}\mathbf{V}_\epsilontilde \right\|_2  \epsilon=\epsilon,
\end{align*}
completing the proof of (iv).
Finally, for all $y\in \Ycal$ we estimate 
\begin{align*}
    &\left\|\mathbf{G}^{1/2}\mathbf{V}_\epsilontilde\mathbf{matr}\left(\Realization_\varrho^\Ycal\left(\Phi^{\mathbf{B}}_{\mathrm{inv};\epsilontilde,\epsilon} \right)(y)\right) \right\|_2 \\ & \leq   \left\|\mathbf{G}^{1/2}\mathbf{V}_\epsilontilde\cdot\left(\left(\mathbf{B}_{y,\epsilontilde}^{\mathrm{rb}}\right)^{-1} -\mathbf{matr}\left(\Realization_\varrho^\Ycal\left(\Phi^{\mathbf{B}}_{\mathrm{inv};\epsilontilde,\epsilon} \right)(y)\right)\right) \right\|_2+ \left\|\mathbf{G}^{1/2}\mathbf{V}_\epsilontilde\left(\mathbf{B}_{y,\epsilontilde}^{\mathrm{rb}}\right)^{-1} \right\|_2 \\
    &\leq \epsilon + \frac{1}{C_{\mathrm{coer}}}.
\end{align*}
This yields (vi). A minor modification of the calculation above yields (v). 
At last, we show (i) and (ii). First of all, it is clear that $L\left(\Phi^{\mathbf{B}}_{\mathrm{inv};\epsilontilde,\epsilon}\right) = L\left(  \Phi^{1-\delta/2,d(\epsilontilde)}_{\mathrm{inv};\frac{\epsilon}{2\alpha}} \odot \Phi^{\mathbf{B},\mathbf{Id}}_{\epsilontilde,\epsilon'} \right)$ and $M\left(\Phi^{\mathbf{B}}_{\mathrm{inv};\epsilontilde,\epsilon}\right) = M\left(  \Phi^{1-\delta/2,d(\epsilontilde)}_{\mathrm{inv};\frac{\epsilon}{2\alpha}} \odot \Phi^{\mathbf{B},\mathbf{Id}}_{\epsilontilde,\epsilon'} \right).$  Moreover, by Lemma \ref{lem:size}(a)(i) in combination with Theorem \ref{thm:Inverse} (i) we have
\begin{align*}
    L\left( \Phi^{\mathbf{B}}_{\mathrm{inv};\epsilontilde,\epsilon} \right) &\leq  C_{\mathrm{inv}} \log_2\left(m\left(\epsilon/(2\alpha),\delta/2\right) \right) \cdot \left(\log_2\left(2\alpha/\epsilon\right)+\log_2\left(m\left(\epsilon/(2\alpha),\delta/2\right) \right)+\log_2(d(\epsilontilde))\right)+ B_L(\epsilontilde,\epsilon')
\end{align*}
and, by Lemma \ref{lem:size}(a)(ii) in combination with Theorem \ref{thm:Inverse}(ii), we obtain
\begin{align*}
    &M\left( \Phi^{\mathbf{B}}_{\mathrm{inv};\epsilontilde,\epsilon} \right) \\ &\quad \leq 2C_{\mathrm{inv}}  m(\epsilon/(2\alpha),\delta/2)\log_2^2\left(m(\epsilon/(2\alpha),\delta/2) \right)   d(\epsilontilde)^3\cdot \left(\log_2\left(2\alpha/\epsilon\right)+ \log_2\left(m(\epsilon/(2\alpha),\delta/2)\right) +\log_2(d(\epsilontilde)) \right)  \\& \quad +2d(\epsilontilde)^2 + 2B_M(\epsilontilde, \epsilon').
\end{align*}

In addition, by using the definition of $m(\epsilon,\delta)$ in the statement of Theorem \ref{thm:Inverse}, for some constant $\tilde{C}>0$ there holds $m\left(\epsilon/(2\alpha),\delta/2\right) \leq \tilde{C}\log_2(1/\epsilon).$ 
Hence, the claim follows for a suitably chosen constant $C_B = C_B(C_{\mathrm{coer}},C_{\mathrm{cont}}) > 0$.
\end{proof}

\subsection{Proof of Theorem \ref{thm:NNCoefficientApproximation}}
We start with proving (i) by deducing the estimate for $\Phi^{\mathbf{u},\mathrm{h}}_{\epsilontilde,\epsilon}$. The estimate for $\Phi^{\mathbf{u},\mathrm{rb}}_{\epsilontilde,\epsilon}$ follows in a similar, but simpler way. For $y\in \Ycal$, we have that 
\begin{align*}
    &\left|\tilde{\mathbf{u}}^{\mathrm{h}}_{y,\epsilontilde} - \Realization_\varrho^\Ycal\left( \Phi^{\mathbf{u},\mathrm{h}}_{\epsilontilde,\epsilon}\right)(y) \right|_{\mathbf{G}} 
    \\ &=\left|\mathbf{G}^{1/2}\cdot\left(\mathbf{V}_\epsilontilde\left(\mathbf{B}_{y,\epsilontilde}^{\mathrm{rb}}\right)^{-1}\mathbf{f}_{y,\epsilontilde}^{\mathrm{rb}} - \Realization_\varrho^\Ycal\left( \Phi^{\mathbf{u},\mathrm{h}}_{\epsilontilde,\epsilon}\right)(y) \right)\right| \\
    & \leq \left|\mathbf{G}^{1/2}\mathbf{V}_\epsilontilde\cdot\left(\left(\mathbf{B}_{y,\epsilontilde}^{\mathrm{rb}}\right)^{-1}\mathbf{f}_{y,\epsilontilde}^{\mathrm{rb}} -\left(\mathbf{B}_{y,\epsilontilde}^{\mathrm{rb}}\right)^{-1} \Realization_\varrho^\Ycal\left(\Phi^{\mathbf{f}}_{\epsilontilde,\epsilon''} \right)(y)\right)\right| \\
    &\quad +\left|\mathbf{G}^{1/2}\mathbf{V}_\epsilontilde\cdot \left(\left(\mathbf{B}_{y,\epsilontilde}^{\mathrm{rb}}\right)^{-1} \Realization_\varrho^\Ycal\left(\Phi^{\mathbf{f}}_{\epsilontilde,\epsilon''} \right)(y) - \mathbf{matr}\left(\Realization_\varrho^\Ycal\left(\Phi^{\mathbf{B}}_{\mathrm{inv};\epsilontilde,\epsilon'} \right)(y) \right)  \Realization_\varrho^\Ycal\left(\Phi^{\mathbf{f}}_{\epsilontilde,\epsilon''} \right)(y)\right)\right| \\
    &\quad+ \left|\mathbf{G}^{1/2}\cdot \left(\mathbf{V}_\epsilontilde\mathbf{matr}\left(\Realization_\varrho^\Ycal\left(\Phi^{\mathbf{B}}_{\mathrm{inv};\epsilontilde,\epsilon'} \right)(y) \right)  \Realization_\varrho^\Ycal\left(\Phi^{\mathbf{f}}_{\epsilontilde,\epsilon''} \right)(y) - \Realization_\varrho^\Ycal\left( \Phi^{\mathbf{u},\mathrm{h}}_{\epsilontilde,\epsilon}\right)(y) \right)\right| \eqqcolon \mathrm{I} + \mathrm{II} + \mathrm{III}.
\end{align*}
We now estimate $\mathrm{I},\mathrm{II},\mathrm{III}$ separately. By Equation \eqref{eq:NormTrafo}, Equation \eqref{eq:NormStiffness}, Assumption \ref{ass:ParameterRHS}, and the definition of $\epsilon''$ there holds for $y\in \Ycal$ that 
\begin{align*}
    \mathrm{I} &\leq \left\|\mathbf{G}^{1/2}\mathbf{V}_\epsilontilde \right\|_2   \left\|\left(\mathbf{B}_{y,\epsilontilde}^{\mathrm{rb}}\right)^{-1} \right\|_2   \left|\mathbf{f}_{y,\epsilontilde}^{\mathrm{rb}}- \Realization_\varrho^\Ycal\left(\Phi^{\mathbf{f}}_{\epsilontilde,\epsilon''} \right)(y) \right| \leq\frac{1}{C_{\mathrm{coer}}}  \frac{\epsilon  C_{\mathrm{coer}}}{3} = \frac{\epsilon}{3}.
\end{align*}
We proceed with estimating II. 
It is not hard to see from Assumption \ref{ass:ParameterRHS} that 
\begin{align}\label{rem:ParameterRHS}
    \sup_{y\in \Ycal} \left|  \Realization^{\Ycal}_{\varrho}\left(\Phi^{\mathbf{f}}_{\epsilontilde,\epsilon}\right)(y)\right|\leq \epsilon+C_{\mathrm{rhs}}.
\end{align}
By definition, $\epsilon' = \epsilon/\max\{6,C_{\mathrm{rhs}}\} \leq \epsilon.$ 
Hence, by Assumption \ref{ass:ParameterStiffness} and \eqref{rem:ParameterRHS} in combination with Proposition \ref{prop:InverseStiffnessNN} (i), 
we obtain
\begin{align*}
    \mathrm{II} &\leq \left\|\mathbf{G}^{1/2}\mathbf{V}_\epsilontilde \cdot \left(\left(\mathbf{B}_{y,\epsilontilde}^{\mathrm{rb}}\right)^{-1} - \mathbf{matr}\left(\Realization_\varrho^\Ycal\left(\Phi^{\mathbf{B}}_{\mathrm{inv};\epsilontilde,\epsilon'} \right)(y) \right)\right) \right\|_2 
      \left|\Realization_\varrho^\Ycal\left(\Phi^{\mathbf{f}}_{\epsilontilde,\epsilon''} \right)(y) \right|
     \leq  \epsilon'\cdot\left(C_{\mathrm{rhs}}+\frac{\epsilon\cdot C_{\mathrm{coer}}}{3} \right) \\
     &\leq \frac{\epsilon}{\max\{6,C_{\mathrm{rhs}}\}}  C_{\mathrm{rhs}} + \frac{\epsilon  C_{\mathrm{coer}}}{\max\{6,C_{\mathrm{rhs}}\}}  \frac{\epsilon}{3} \leq \frac{2\epsilon}{6}=\frac{\epsilon}{3},
\end{align*}
where we have used that $C_{\mathrm{coer}}  \epsilon <C_{\mathrm{coer}}  {\alpha}/{4}<1.$ Finally, we estimate III.
Per construction, we have that 
$$
\Realization_\varrho^\Ycal \left(\Phi^{\mathbf{u},\mathrm{h}}_{\epsilontilde,\epsilon}\right)(y) = \mathbf{V}_\epsilontilde  \Realization_\varrho^\Ycal\left(\Phi^{\kappa,d(\epsilontilde),d(\epsilontilde),1}_{\mathrm{mult};\frac{\epsilon}{3}} \odot \Paral\left(\Phi^{\mathbf{B}}_{\mathrm{inv};\epsilontilde,\epsilon'},\Phi^{\mathbf{f}}_{\epsilontilde,\epsilon''} \right)\right)(y,y).
$$
Moreover, we have by Proposition \ref{prop:InverseStiffnessNN}(v)
\begin{align*}
    \left\|\mathbf{matr}\left(\Realization_\varrho^\Ycal\left(\Phi^{\mathbf{B}}_{\mathrm{inv};\epsilontilde,\epsilon'} \right)(y) \right) \right\|_2 \leq \epsilon + \frac{1}{C_{\mathrm{coer}}} \leq 1+\frac{1}{C_{\mathrm{coer}}}\leq \kappa
\end{align*}
and by \eqref{rem:ParameterRHS} that
\begin{align*}
     \left|\Realization_\varrho^\Ycal\left(\Phi^{\mathbf{f}}_{\epsilontilde,\epsilon''} \right)(y) \right| \leq \epsilon'' + C_{\mathrm{rhs}} \leq \epsilon   C_{\mathrm{coer}}+  C_{\mathrm{rhs}} \leq 1+ C_{\mathrm{rhs}} \leq \kappa.
\end{align*}
Hence, by the choice of $\kappa$ and Proposition \ref{prop:Multiplikation} we conclude that $\mathrm{III} \leq {\epsilon}/{3}$. Combining the estimates on $\mathrm{I}, \mathrm{II}$, and $\mathrm{III}$ yields (i) and using (i) implies (v). Now we estimate the size of the NNs. We start with proving (ii). First of all, we have by the definition of $\Phi^{\mathbf{u}, \mathrm{rb}}_{\epsilontilde,\epsilon}$ and $\Phi^{\mathbf{u}, \mathrm{h}}_{\epsilontilde,\epsilon}$ as well as Lemma \ref{lem:size}(a)(i) in combination with Proposition \ref{prop:Multiplikation} that
\begin{align}
L\left(\Phi^{\mathbf{u}, \mathrm{rb}}_{\epsilontilde,\epsilon}\right)&< L\left(\Phi^{\mathbf{u},\mathrm{h}}_{\epsilontilde,\epsilon}\right) \leq 1 + L\left( \Phi^{\kappa,d(\epsilontilde),d(\epsilontilde),1}_{\mathrm{mult};\frac{\epsilon}{3}}\right)+ L\left(\Paral\left(\Phi^{\mathbf{B}}_{\mathrm{inv};\epsilontilde,\epsilon'},\Phi^{\mathbf{f}}_{\epsilontilde,\epsilon''} \right)\right)\nonumber\\
&\leq 1 +C_{\mathrm{mult}}\cdot \left(\log_2(3/\epsilon)+3/2 \log_2(d(\epsilontilde))+\log_2(\kappa)\right)+\max\left\{L\left(\Phi^{\mathbf{B}}_{\mathrm{inv};\epsilontilde,\epsilon'} \right), F_L\left(\epsilontilde,\epsilon''\right) \right\}\nonumber\\
& \leq C^{\mathbf{u}}_L \max\left\{ \log_2(\log_2(1/\epsilon))\left(\log_2(1/\epsilon)+ \log_2(\log_2(1/\epsilon))+\log_2(d(\epsilontilde))\right) + B_L(\epsilontilde, \epsilon'''), F_L\left(\epsilontilde,\epsilon''\right)\right\} \label{eq:CLDef}
\end{align}
where we applied Proposition \ref{prop:InverseStiffnessNN}(i) and chose a suitable constant $$C^{\mathbf{u}}_L = C^{\mathbf{u}}_L(\kappa, \epsilon', C_B) = C^{\mathbf{u}}_L(C_{\mathrm{rhs}}, C_{\mathrm{coer}},C_{\mathrm{cont}}) >0.$$ 

We now note that if we establish (iii), then (iv) follows immediately by Lemma \ref{lem:size}(a)(ii).
Thus, we proceed with proving (iii). First of all, by Lemma \ref{lem:size}(a)(ii) in combination with Proposition \ref{prop:Multiplikation} we have 
\begin{align*}
    M\left(\Phi^{\mathbf{u}, \mathrm{rb}}_{\epsilontilde,\epsilon}\right) &\leq 2M\left( \Phi^{\kappa,d(\epsilontilde),d(\epsilontilde),1}_{\mathrm{mult};\frac{\epsilon}{3}}\right)+ 2M\left(\Paral\left(\Phi^{\mathbf{B}}_{\mathrm{inv};\epsilontilde,\epsilon'},\Phi^{\mathbf{f}}_{\epsilontilde,\epsilon''} \right)\right)\\
&\leq 2C_{\mathrm{mult}}d(\epsilontilde)^2 \cdot \left(\log_2(3/\epsilon)+3/2 \log_2(d(\epsilontilde))+\log_2(\kappa)\right)+ 2M\left(\Paral\left(\Phi^{\mathbf{B}}_{\mathrm{inv};\epsilontilde,\epsilon'},\Phi^{\mathbf{f}}_{\epsilontilde,\epsilon''} \right) \right).
\end{align*}
Next, by Lemma \ref{lem:size}(b)(ii) in combination with Proposition \ref{prop:InverseStiffnessNN} as well as Assumption \ref{ass:ParameterStiffness} and Assumption \ref{ass:ParameterRHS} we have that 
\begin{align}
    &M\left(\Paral\left(\Phi^{\mathbf{B}}_{\mathrm{inv};\epsilontilde,\epsilon'},\Phi^{\mathbf{f}}_{\epsilontilde,\epsilon''} \right)\right)\nonumber\\ 
    &\leq M\left( \Phi^{\mathbf{B}}_{\mathrm{inv};\epsilontilde,\epsilon'}\right) + M\left(\Phi^{\mathbf{f}}_{\epsilontilde,\epsilon''} \right) \nonumber\\ 
    &\quad+ 8d(\epsilontilde)^2 \max\left\{C^{\mathbf{u}}_L \log_2(\log_2(1/\epsilon'))\left(\log_2(1/\epsilon') + \log_2(\log_2(1/\epsilon'))+\log_2(d(\epsilontilde))\right) + B_L(\epsilontilde, \epsilon'''), F_L\left(\epsilontilde,\epsilon''\right)\right\}  \nonumber\\
    &\leq C_{B}  \log_2(1/\epsilon')  \log_2^2(\log_2(1/\epsilon'))   d(\epsilontilde)^3\cdot \left(\log_2(1/\epsilon')+\log_2(\log_2(1/\epsilon'))+\log_2(d(\epsilontilde)) \right)  \nonumber \\ 
    &\quad+ 8d(\epsilontilde)^2 \max\left\{C^{\mathbf{u}}_L \log_2(\log_2(1/\epsilon'))\left(\log_2(1/\epsilon') + \log_2(\log_2(1/\epsilon'))+\log_2(d(\epsilontilde))\right) + B_L(\epsilontilde, \epsilon'''),\nonumber F_L\left(\epsilontilde,\epsilon''\right)\right\}  
    \\ 
    &\quad + 2B_M\left(\epsilontilde,\epsilon'''\right) + F_M\left(\epsilontilde,\epsilon'' \right) \nonumber
    \\
    &\leq C^{\mathbf{u}}_M   d(\epsilontilde)^2 \cdot \bigg(d(\epsilontilde)\log_2(1/\epsilon)  \log_2^2(\log_2(1/\epsilon))\big(\log_2(1/\epsilon)+\log_2(\log_2(1/\epsilon)) + \log_2(d(\epsilontilde))\big)... \nonumber 
    \\  &\qquad \qquad \qquad \qquad \qquad \qquad \qquad \qquad \qquad  ...+ B_L(\epsilontilde,\epsilon''') + F_L\left(\epsilontilde,\epsilon''\right)\bigg) \quad   +2B_M(\epsilontilde,\epsilon''') +F_M\left(\epsilontilde,\epsilon''\right), \label{eq:CMDef}
\end{align}
for a suitably chosen constant $C^{\mathbf{u}}_M 
= C^{\mathbf{u}}_M(\epsilon', C_B, C^{\mathbf{u}}_L) = C^{\mathbf{u}}_L(C_{\mathrm{rhs}}, C_{\mathrm{coer}},C_{\mathrm{cont}})> 0$.
This shows the claim.

\end{document}